\documentclass[11pt,a4paper]{article}
\usepackage[T1]{fontenc}
\usepackage[utf8]{inputenc}
\usepackage{lmodern,amsmath,amssymb,amsthm,mathtools,mathrsfs}
\usepackage[a4paper,left=28mm,right=28mm,top=27mm,bottom=25mm,headheight=14pt,headsep=18pt]{geometry}
\usepackage{microtype,enumitem,needspace,titlesec,etoolbox}
\usepackage{xcolor,aliascnt}
\usepackage{cite}
\definecolor{ReferenceBlue}{HTML}{153F75}
\usepackage[colorlinks=true,linkcolor=blue,citecolor=blue,
urlcolor=blue,linktoc=all,bookmarksnumbered=true, bookmarksopen=true,bookmarksopenlevel=1]{hyperref}
\usepackage{bookmark}
\usepackage[capitalize,noabbrev]{cleveref}
\crefname{equation}{Equation}{Equations}
\crefname{theorem}{Theorem}{Theorems}
\crefname{proposition}{Proposition}{Propositions}
\crefname{lemma}{Lemma}{Lemmas}
\crefname{corollary}{Corollary}{Corollaries}
\crefname{definition}{Definition}{Definitions}
\crefname{remark}{Remark}{Remarks}
\crefname{section}{Section}{Sections}
\crefname{appendix}{Appendix}{Appendices}
\AtBeginDocument{}
\hypersetup{pdftitle={Quantitative thermoviscous Darcy limits in periodic and bounded domains},
pdfsubject={Two-domain homogenization with temperature-dependent viscosity and boundary-layer estimates}}
\titleformat{\section}{\large\bfseries}{\thesection}{0.8em}{}
\titleformat{\subsection}{\normalsize\bfseries}{\thesubsection}{0.7em}{}
\titlespacing*{\section}{0pt}{2.2ex plus 0.5ex minus 0.2ex}{1.1ex plus 0.2ex}
\titlespacing*{\subsection}{0pt}{1.7ex plus 0.4ex minus 0.2ex}{0.8ex plus 0.2ex}
\numberwithin{equation}{section}
\allowdisplaybreaks[1]
\newtheoremstyle{paper}{6pt plus 2pt minus 1pt}{6pt plus 2pt minus 1pt}
{\itshape}{}{\bfseries}{.}{0.5em}{}
\theoremstyle{paper}
\newtheorem{theorem}{Theorem}[section]
\newaliascnt{proposition}{theorem}
\newtheorem{proposition}[proposition]{Proposition}
\aliascntresetthe{proposition}
\newaliascnt{lemma}{theorem}
\newtheorem{lemma}[lemma]{Lemma}
\aliascntresetthe{lemma}
\newaliascnt{corollary}{theorem}
\newtheorem{corollary}[corollary]{Corollary}
\aliascntresetthe{corollary}
\theoremstyle{paper}
\newaliascnt{definition}{theorem}
\newtheorem{definition}[definition]{Definition}
\aliascntresetthe{definition}
\theoremstyle{paper}
\newaliascnt{remark}{theorem}
\newtheorem{remark}[remark]{Remark}
\aliascntresetthe{remark}
\newcommand{\statementspace}{\Needspace{3\baselineskip}}
\AtBeginEnvironment{theorem}{\statementspace}
\AtBeginEnvironment{lemma}{\statementspace}
\AtBeginEnvironment{proposition}{\statementspace}
\AtBeginEnvironment{corollary}{\statementspace}
\AtBeginEnvironment{definition}{\statementspace}
\AtBeginEnvironment{remark}{\statementspace}
\newcommand{\T}{\mathbb T^3}
\newcommand{\R}{\mathbb R}
\newcommand{\eps}{\varepsilon}
\newcommand{\sig}{\sigma_\eps}
\newcommand{\cellsize}{\ell_\eps}
\newcommand{\Oe}{\Omega_\eps}
\newcommand{\He}{\mathcal H_\eps}
\newcommand{\V}{\mathcal V}
\newcommand{\Ve}{\mathcal V_\eps}
\newcommand{\PP}{\mathscr P}
\newcommand{\dd}{\,\mathrm d}
\newcommand{\dx}{\dd \bx}
\newcommand{\dt}{\dd t}
\newcommand{\Div}{\operatorname{div}}
\newcommand{\av}[1]{\langle #1\rangle}
\newcommand{\norm}[2]{\left\|#1\right\|_{#2}}
\newcommand{\E}{\mathcal E}
\newcommand{\Lop}{\mathscr L}
\newcommand{\bu}{\boldsymbol{u}}
\newcommand{\bU}{\boldsymbol{U}}
\newcommand{\bv}{\boldsymbol{v}}
\newcommand{\bw}{\boldsymbol{w}}
\newcommand{\bfv}{\boldsymbol{f}}
\newcommand{\bg}{\boldsymbol{g}}
\newcommand{\ba}{\boldsymbol{a}}
\newcommand{\be}{\boldsymbol{e}}
\newcommand{\bV}{\boldsymbol{V}}
\newcommand{\bW}{\boldsymbol{W}}
\newcommand{\bM}{\mathbf{M}}
\newcommand{\bA}{\mathbf{A}}
\newcommand{\bI}{\mathbf{I}}
\newcommand{\bn}{\boldsymbol{n}}
\newcommand{\bnr}{\boldsymbol{n}_{r}}
\newcommand{\bnh}{\boldsymbol{n}_{h}}
\newcommand{\bx}{\boldsymbol{x}}
\newcommand{\bh}{\boldsymbol{h}}
\newcommand{\br}{\boldsymbol{r}}
\newcommand{\bH}{\boldsymbol{H}}
\newcommand{\bG}{\boldsymbol{G}}
\newcommand{\bF}{\boldsymbol{F}}
\newcommand{\bQ}{\boldsymbol{Q}}
\newcommand{\bphi}{\boldsymbol{\phi}}
\newcommand{\bzeta}{\boldsymbol{\zeta}}
\newcommand{\bb}{\boldsymbol{b}}
\newcommand{\bc}{\boldsymbol{c}}
\newcommand{\bd}{\boldsymbol{d}}
\newcommand{\bz}{\boldsymbol{z}}
\newcommand{\bpsi}{\boldsymbol{\psi}}
\newcommand{\bxi}{\boldsymbol{\xi}}
\newcommand{\bmom}{\boldsymbol{M}}
\newcommand{\Bop}{\mathscr B}
\newcommand{\Jop}{\mathscr J}
\newcommand{\Kop}{\mathscr K}
\newcommand{\Sop}{\mathscr S}

\newcommand{\sym}{\mathrm D}
\newcommand{\Xe}{\mathcal X_\eps}
\newcommand{\Hvel}{\mathscr H_\eps}
\newcommand{\Cgeo}{\mathcal C_\eps}
\newcommand{\deltae}{\delta_\eps}
\newcommand{\Dinit}{\mathcal D_\eps^0}

\newif\ifanonymous
\anonymousfalse 
\begin{document}
\thispagestyle{plain}
\begingroup
\renewcommand{\thefootnote}{\fnsymbol{footnote}}
\setcounter{footnote}{0}
\begin{center}
{\LARGE\bfseries Quantitative thermoviscous Darcy limits in periodic and bounded domains\par}
\vspace{1.4em}
{\large Jiaojiao Pan\footnote{\raggedright Institute of Applied Physics and Computational Mathematics, Beijing 100088, China. \texttt{panjiaojiao.math@gmail.com}}\qquad Luqi Wang\footnote{\raggedright School of Mathematics, Nanjing University, Nanjing 210093, China. \texttt{wangluqi@nju.edu.cn}}\par}
\end{center}
\vspace{0.4em}
\begin{small}
	\begin{abstract}
In this paper, we establish quantitative Darcy limits for a three-dimensional inhomogeneous incompressible fluid with temperature-dependent viscosity and quasi-static thermal feedback. The domain is either the flat torus or a bounded domain with $C^{3,\beta}$ boundary, $0<\beta<1$. The holes have size $\eps^\alpha$ and separation of order $\eps$, with $1<\alpha<3$. Before normalization, the fluid and solid conductivities are $\sig^2\kappa_f$ and $\sig^2\kappa_s$, where $\sig^2\sim\eps^{3-\alpha}$. This scaling preserves the thermal feedback in the limit. For a sufficiently regular reference solution and under an explicit absorption condition, the squared density, velocity and temperature errors are bounded by the weighted initial discrepancy together with $\eps^{\alpha-1}+\eps^{3-\alpha}$ on the torus and $\eps^{\alpha-1}+\eps^{(3-\alpha)/2}$ in a bounded domain. An interface-independent temperature estimate controls the viscosity discrepancy, and a weighted Stokes residual retains the coefficient and cell-pressure commutators. The bounded comparison uses a solenoidal wall correction and cellwise divergence repair. We also construct fixed-parameter microscopic weak solutions. The local strong effective solutions require a separate contraction condition in both settings and stronger boundary regularity in the bounded case.
\end{abstract}


\end{small}
\endgroup
\vspace{1.0em}

\section{Introduction}
Homogenization of viscous flow in perforated domains describes the effective motion produced by a large number of small holes (or solid obstacles). For periodically distributed holes in three dimensions, the ratio between the hole size and the cell spacing determines the limiting equation. In the classical work of Allaire \cite{AllaireI,AllaireII}, holes of size $\eps^\alpha$ and spacing of order $\eps$ lead, after the appropriate velocity normalization, to Darcy's law for $1<\alpha<3$, a Brinkman correction at $\alpha=3$ and the original Stokes or Navier--Stokes system when $\alpha>3$. Lu \cite{Lu2020} later gave a unified treatment of these regimes for the Stokes system.

The evolutionary incompressible problem has subsequently been developed for several hole-size ranges. Lu and Yang \cite{LY2023} treated the homogeneous incompressible system. For transported density, Pan \cite{Pan2025} obtained the critical Brinkman limit, while Lu, Pan and Yang \cite{LPY2025} considered the very-small-hole regime. In the Darcy range $1<\alpha<3$, Basari\'c, Oschmann and Pan \cite{BOP} derived a density-dependent incompressible Darcy law together with quantitative convergence rates by comparing microscopic weak solutions with a strong effective solution. Transported density also requires compactness of the momentum product $\rho_\eps\bu_\eps$ under the weak energy bounds for $\bu_\eps$.

Quantitative estimates require comparison fields that encode the perforation geometry and the effective resistance. Shen \cite{Shen2022} obtained sharp convergence rates for stationary Stokes flow in porous media. H\"ofer, Lu and Oschmann \cite{HLO} developed quantitative corrector estimates in the dilute Darcy regime, including the additional loss caused by the boundary of a bounded domain, while Jing, Lu and Prange \cite{JLP} obtained a unified quantitative Stokes analysis for independent hole and cell scales. Relative-energy methods provide the stability mechanism for nonlinear and density-dependent problems. See, in particular, \cite{BOP,CBSV2025,HNO2026}.

Thermal feedback has also been studied in other perforation regimes. Feireisl, Lu and Sun \cite[Section~1, equation~(1.5)]{FLS} treated an inhomogeneous heat-conducting fluid with temperature-dependent viscosity in the critical Brinkman regime, with a Dirichlet temperature condition. Lu, Pan and Wang \cite{LPW} considered $\alpha>3$, where the perforations disappear from the mechanical limit, while \cite[Proposition~3.3 and Lemmas~4.2--4.3]{PW2026critical} obtained a quantitative critical stability estimate for a quasi-static thermoviscous model with a Neumann temperature condition. Their analysis controls the viscosity discrepancy through the temperature error and estimates coefficient and cell-pressure commutators.

In the Darcy regime considered here, the effective resistance is $\mu(\vartheta)\bM_0\bu$. The kinetic energy carries the vanishing weight $\sig^4$, so it alone does not control the unweighted $L^2$ velocity error. The perforated Poincar\'e inequality instead converts the viscous dissipation into $L^2$ control. A frozen variable-viscosity Stokes lifting separates the geometric approximation error from the microscopic relative error. We estimate the coefficient and cell-pressure commutators produced by the spatially varying viscosity, and the temperature estimate on the entire fluid--solid composite controls the viscosity discrepancy in $L^\infty$. Under the absorption condition, we absorb the corresponding stress error into the viscous dissipation. In a bounded domain, we combine this comparison with the wall-correction and local divergence-repair tools of \cite{HLO}. The resulting estimate extends the transported-density Darcy framework of \cite{BOP} to thermal feedback on the full interval $1<\alpha<3$. The compressible relative-energy analysis of \cite{HNO2026} concerns a different density coupling.

We treat both the periodic geometry and a bounded domain. The conductivity scaling keeps the thermal equation coupled at order one after the Darcy normalization. The effective system couples a density-dependent Darcy law, with coefficient $\mu(\vartheta)\bM_0$, to an elliptic heat equation. Quantitatively, the periodic squared error contains $\eps^{\alpha-1}+\eps^{3-\alpha}$. The bounded-domain estimate acquires the additional wall contribution $\eps^{(3-\alpha)/2}$. The same estimates yield strong $L^2_tL^\infty_x$ convergence of the temperature and therefore of the viscosity.

The relative-energy estimate uses an absorption condition on the viscosity discrepancy. A separate contraction condition constructs a local strong effective solution. Neither condition is needed for fixed parameter microscopic weak existence. Matching lower bounds for the coupled estimates and stability under weaker contraction or absorption hypotheses remain open.

\subsection{Problem formulation}\label{sec:problem-formulation}
Let $\mathcal O$ be either the unit flat torus $\T=\R^3/\mathbb Z^3$ or a fixed bounded connected domain $\Omega\subset\R^3$ of class $C^{3,\beta}$, $0<\beta<1$. In the bounded case, $\bn$ denotes the outward unit normal on $\partial\Omega$ and $\partial_n=\bn\cdot\nabla$. For a finite-volume domain $D$, we write $\langle z\rangle_D=|D|^{-1}\int_Dz\dx$ and abbreviate this average to $\av z$ on the torus. Let $\mathbb N_0=\{0,1,2,\ldots\}$. For $1\le p\le\infty$ and $m\in\mathbb N_0$, set
\[
 L^p_\diamond(D)=\{z\in L^p(D):\langle z\rangle_D=0\},\qquad
 H^m_\diamond(D)=H^m(D)\cap L^2_\diamond(D)
\]
on every finite-volume domain $D$. On the torus these are periodic spaces, while $H^1_0(D)$ denotes the space with zero boundary trace. Vector and tensor spaces are understood componentwise.

We write $\delta_{ij}$ for the Kronecker symbol and $r_+=\max\{r,0\}$. For $\beta\in\mathbb N_0^3$, $ \partial^\beta=\partial_1^{\beta_1}\partial_2^{\beta_2}\partial_3^{\beta_3}$, $ |\beta|=\beta_1+\beta_2+\beta_3$, and for operators $A,B$ we set $[A,B]=AB-BA$. For any $T>0$, the abbreviations
\[
 \|z\|_{p,D}=\|z\|_{L^p(D)},\qquad
 L^p_tX=L^p(0,T;X),\qquad
 C^j_tX=C^j([0,T];X),\qquad C_tX=C^0_tX
\]
are used, with $\|z\|_p=\|z\|_{L^p}$ when the spatial domain is clear. The notation $C_{\mathrm{weak}}$ means weak continuity. The symbol $X^*$ denotes the continuous dual. We write $\mathcal L(X,Y)$ for the space of bounded linear maps from $X$ to $Y$ and $X\hookrightarrow\hookrightarrow Y$ for compact embedding. The notation $\mathcal D'(0,T;X)$ denotes distributions in time with values in $X$. Duality brackets also stand for $L^2$ pairings when no ambiguity can arise.

Fix a bounded simply connected $C^\infty$ reference obstacle $T_0$ containing the origin, with connected exterior and $\overline{T_0}\subset B(0,1/8)$. Throughout the paper, the hole size, cell spacing and associated scales are
\[
 a_\eps=\eps^\alpha,\qquad 1<\alpha<3,\qquad
 \cellsize=
 \begin{cases}
  \lceil\eps^{-1}\rceil^{-1},&\mathcal O=\T,\\
  \eps,&\mathcal O=\Omega,
 \end{cases}
 \qquad
 \sig^2=\frac{\cellsize^3}{a_\eps},\qquad
 \tau_\eps=\frac{a_\eps}{\cellsize}.
\]
For the index set $\mathcal I_\eps$ specified by \eqref{eq:geometry} or \eqref{eq:bounded-geometry}, define
\[
 \begin{gathered}
 Q_{\eps,k}=\cellsize k+\cellsize[-1/2,1/2)^3,\ T_{\eps,k}=\cellsize k+a_\eps\overline{T_0},\  \He=\bigcup_{k\in\mathcal I_\eps}T_{\eps,k},\ 
 \Oe=\mathcal O\setminus\He,\quad h_\eps=|\He|.
 \end{gathered}
\]
Thus $\cellsize/\eps\to1$, $\sig\to0$ and $\tau_\eps\to0$. The torus is tiled by compatible cells, whereas the bounded geometry contains holes only in complete interior cells.

In the periodic case, set
\begin{equation}\label{eq:continuous-parameter}
 N_\eps=\lceil\eps^{-1}\rceil,\qquad
 \mathcal I_\eps=\mathcal I_\eps^{\mathrm{per}}
     =\{0,\ldots,N_\eps-1\}^3.
\end{equation}
The spacing equals $N_\eps^{-1}$ and satisfies
\begin{equation}\label{eq:mesh-comparison}
 \frac{\eps}{1+\eps}\le\cellsize\le\eps,\qquad
 0\le\eps-\cellsize\le\eps^2,\qquad
 \frac{\cellsize}{\eps}\to1.
\end{equation}
Indices and obstacle locations are understood modulo $N_\eps$ and $\mathbb Z^3$, respectively. The cells tile $\T$, and the total solid volume is
\begin{equation}\label{eq:geometry}
 h_\eps=|T_0|\frac{a_\eps^3}{\cellsize^3}=|T_0|\tau_\eps^3.
\end{equation}
For sufficiently small $\eps$, all holes lie strictly inside their cells and remain pairwise disjoint. We may also assume $a_\eps<\cellsize$, $\sig\le1$ and $h_\eps\le1$. The exact and asymptotic scale relations are
\[
 \begin{gathered}
 \sig^2=\left(\frac{\cellsize}{\eps}\right)^3\eps^{3-\alpha}\sim\eps^{3-\alpha},\\
 \left(\frac{\cellsize}{\sig}\right)^2=\tau_\eps=\frac{\eps^\alpha}{\cellsize}\sim\eps^{\alpha-1},\\
 h_\eps=|T_0|\left(\frac{\eps}{\cellsize}\right)^3\eps^{3(\alpha-1)}\sim |T_0|\eps^{3(\alpha-1)}.
 \end{gathered}
\]
When $\eps=1/N$, the compatible mesh has spacing $\eps$. For general real $\eps$, \eqref{eq:mesh-comparison} changes the spacing by $O(\eps^2)$ without changing the prescribed hole size. The chosen resistance scale keeps the cell-average resistance exact.

In the bounded case, the spacing is $\eps$ and the index set is
\begin{equation}\label{eq:bounded-geometry}
 \mathcal I_\eps=\mathcal I_\eps^{\Omega}
   =\{k\in\mathbb Z^3:\overline{Q_{\eps,k}}\subset\Omega\}.
\end{equation}
In particular, $\sig^2=\eps^{3-\alpha}$ and $\tau_\eps=\eps^{\alpha-1}$.
Only complete interior cubes carry holes. In particular, no hole is cut by $\partial\Omega$. Set
\[
 \mathcal Q_\eps=\bigcup_{k\in \mathcal I_\eps^{\Omega}}\overline{Q_{\eps,k}},
 \qquad
 S_\ell=\{\bx\in\Omega:\operatorname{dist}(\bx,\partial\Omega)<\ell\},
\]
where $\ell$ denotes the boundary-strip thickness. There exists a fixed tubular-neighborhood radius $\ell_0>0$ such that, for $0<\ell<\ell_0$,
\begin{equation}\label{eq:bounded-count}
 \begin{gathered}
 \Omega\setminus\mathcal Q_\eps\subset S_{2\sqrt3\eps},
 \quad |S_\ell|\le C_\Omega\ell,\quad
 h_\eps=\#\mathcal I_\eps^{\Omega}\,a_\eps^3|T_0|\le C_\Omega\eps^{3(\alpha-1)}.
 \end{gathered}
\end{equation}
Moreover, $\eps^3\#\mathcal I_\eps^{\Omega}=|\Omega|+O(\eps)$ and $\sig^2a_\eps/\cellsize^3=1$. With $R_0=\sup_{y\in T_0}|y|<1/8$ and $a_\eps\le\eps$,
\begin{equation*}
 \operatorname{dist}(T_{\eps,k},\partial\Omega)
       \ge\frac{\eps}{2}-a_\eps R_0
       \ge\frac{3\eps}{8}\quad(k\in \mathcal I_\eps^{\Omega}).
\end{equation*}
The omitted boundary strip has vanishing measure, so its normalized contribution to the interior resistance tends to zero. The $C^{3,\beta}$ boundary suffices for the solenoidal wall correction in the homogenization argument. The independent $H^s$ construction of the macroscopic solution uses the stronger regularity stated in \cref{thm:bounded-local}.

Unless stated otherwise, spatial norms are taken over the ambient domain, namely $\T$ in the periodic problem and $\Omega$ in the bounded problem. Unextended fluid fields are measured on $\Oe$. Constants carrying only fixed-data subscripts are uniform in $\eps$ and in Galerkin dimensions. The constants $C_\eps$, $C_{\eps,T}$ and $C_D$ may depend on the fixed perforated domain, while $K_{\eps,n}$ may also depend on $n$. Finally, $A_\eps\sim B_\eps$ means $A_\eps/B_\eps\to1$.

A tilde denotes zero extension from $\Oe$ to $\mathcal O$, while $\mathbf1_E$ is the indicator of $E$. Our tensor conventions are
\[
 (\nabla \bv)_{ij}=\partial_jv_i,\qquad
 \sym\bv=\frac{\nabla\bv+(\nabla\bv)^{\mathrm T}}2,\qquad
 (\boldsymbol a\otimes\boldsymbol b)_{ij}=a_ib_j,\qquad
 \mathbf A:\mathbf B=\sum_{i,j}A_{ij}B_{ij}.
\]

With the perforated domain fixed by \eqref{eq:geometry} or \eqref{eq:bounded-geometry} and $T>0$, we consider the normalized microscopic system
\begin{equation}\label{eq:micro-model}
\left\{
\begin{aligned}
 \partial_t\rho_\eps+\Div(\rho_\eps\bu_\eps)&=0,
     &&\text{in }(0,T)\times\Oe,\\
 \Div\bu_\eps&=0,
     &&\text{in }(0,T)\times\Oe,\\
 \sig^4\!\left[\partial_t(\rho_\eps\bu_\eps)
  +\Div(\rho_\eps\bu_\eps\otimes\bu_\eps)\right]
  -\Div\!\left(2\sig^2\mu(\vartheta_\eps)\sym\bu_\eps\right)
  +\nabla p_\eps
 &=\rho_\eps\bfv-\vartheta_\eps\bg,
     &&\text{in }(0,T)\times\Oe,\\
 -\Div(\bar\kappa_\eps\nabla\vartheta_\eps)
 &=\widetilde\bu_\eps\cdot\bg,
     &&\text{in }(0,T)\times\mathcal O.
\end{aligned}
\right.
\end{equation}
Here $\rho_\eps$, $\bu_\eps$, $p_\eps$ and $\vartheta_\eps$ are the density, velocity, pressure and temperature fluctuation. The mechanical variables are defined in $\Oe$, while $\vartheta_\eps$ is defined on the entire conducting composite $\mathcal O$. The prescribed data are independent of time and satisfy
\begin{equation}\label{eq:force-data}
 \bfv\in C^\infty(\mathcal O;\R^3),\qquad
 F\in C^\infty(\mathcal O),\qquad \bg=\nabla F.
\end{equation}
On the torus these functions are periodic. In the bounded case,
$C^\infty(\mathcal O)$ means restriction of a smooth function from a
neighborhood of $\overline\Omega$. These spatial assumptions are used in the
Sobolev estimates for the effective solution.

The coefficients in \eqref{eq:micro-model} satisfy
\begin{equation}\label{eq:coefficients}
 \begin{gathered}
 \mu\in C^\infty(\R),\quad
 0<\mu_*\le\mu(r)\le\mu^*<\infty,\quad
 L_\mu=\|\mu'\|_{L^\infty(\R)}<\infty,\\
 \bar\kappa_\eps=\kappa_f\mathbf1_{\Oe}+\kappa_s\mathbf1_{\He},\quad
 \kappa_f,\kappa_s>0,\quad
 \kappa_*=\min\{\kappa_f,\kappa_s\},\quad
 \kappa^*=\max\{\kappa_f,\kappa_s\}.
 \end{gathered}
\end{equation}
Here $\mu$ is the temperature-dependent viscosity, $\kappa_f$ and $\kappa_s$ are the fluid and solid conductivities. The corresponding composite conductivity is $\bar\kappa_\eps$. Higher derivatives of $\mu$ are required only on the temperature range reached by the solutions. We fix common density bounds $0<\rho_*\le\rho^*<\infty$.

The initial and boundary conditions associated with \eqref{eq:micro-model} are
\begin{equation}\label{eq:micro-reference}
 \begin{gathered}
 \bu_\eps|_{\partial\He}=0,\quad
 \langle\vartheta_\eps\rangle_{\mathcal O}=0,\\
 \rho_\eps(0)=\rho_\eps^0,\quad \bu_\eps(0)=\bu_\eps^0,
 \quad \rho_*\le\rho_\eps^0\le\rho^*.
 \end{gathered}
\end{equation}
In the periodic case all unknowns are periodic. In the bounded case,
\begin{equation*}
 \bu_\eps=0,\qquad
 \bar\kappa_\eps\nabla\vartheta_\eps\cdot\bn=0
 \quad\text{on }\partial\Omega.
\end{equation*}
Thus the microscopic velocity satisfies the no-slip condition on the obstacle boundaries and, in the bounded-domain setting, also on $\partial\Omega$. The heat flux through $\partial\Omega$ is zero. The heat equation is understood variationally and encodes the transmission of temperature and normal heat flux across the fluid--solid interfaces. Since $\widetilde\bu_\eps$ is solenoidal and $\bg=\nabla F$, the source $\widetilde\bu_\eps\cdot\bg$ in the heat equation has zero mean.

The normalized variables in \eqref{eq:micro-model} are related to the primitive variables by
\begin{equation}\label{eq:normalization}
 \begin{gathered}
 t=\sig^2s,\quad
 \rho_\eps(t,\bx)=R_\eps(t/\sig^2,\bx),\quad
 \bu_\eps(t,\bx)=\sig^{-2}\bU_\eps(t/\sig^2,\bx),\\
 \vartheta_\eps(t,\bx)=\Theta_\eps(t/\sig^2,\bx)-\theta_0,\quad
 p_\eps(t,\bx)=\Pi_\eps(t/\sig^2,\bx)+\theta_0F(\bx).
 \end{gathered}
\end{equation}
Here $s$ is the original time variable, while $R_\eps$, $\bU_\eps$, $\Theta_\eps$ and $\Pi_\eps$ denote the primitive density, velocity, absolute temperature and pressure. The constant $\theta_0$ is the reference temperature. We use $\mu(r)$ for the viscosity evaluated at $\theta_0+r$. Before normalization, the conductivity equals $\sig^2\bar\kappa_\eps$ and the thermal force is $-\Theta_\eps\bg$. The pressure shift in \eqref{eq:normalization} absorbs the conservative part $-\theta_0\bg$. Consequently, both inertial terms carry the factor $\sig^4$, whereas the viscous term carries $\sig^2$. The common factor $\sig^2$ cancels from the heat equation, which keeps the thermal coupling at order one in the Darcy limit. The resulting heat equation is quasi-static, with thermal inertia and advection absent from the chosen model.

The effective resistance is determined by the exterior Stokes problems
\[
\left\{
\begin{aligned}
 -\Delta\bV^i+\nabla Q^i&=0,\quad \Div\bV^i=0,
     &&\text{in }\R^3\setminus\overline{T_0},\\
 \bV^i&=0,
     &&\text{on }\partial T_0,\\
 \bV^i(\bx)&\to\be_i,
     &&\text{as }|\bx|\to\infty,
\end{aligned}
\right.
\qquad i=1,2,3.
\]
Here $\be_i$ denotes the $i$-th vector of the standard basis of $\R^3$. The pair $(\bV^i,Q^i)$ belongs to the finite-energy class $\bV^i-\be_i\in L^6(\R^3\setminus\overline{T_0})$, $\nabla\bV^i\in L^2$ and $Q^i\in L^2$, with the pressure normalized by decay at infinity. The resistance matrix is
\begin{equation}\label{eq:resistance}
 \begin{gathered}
 (\bM_0)_{ij}=\int_{\R^3\setminus\overline{T_0}}
        \nabla \bV^i:\nabla \bV^j\dx
 =2\int_{\R^3\setminus\overline{T_0}}
        \sym \bV^i:\sym \bV^j\dx,\quad i,j\in\{1,2,3\},\\
 \lambda=\lambda_{\min}(\bM_0)>0,\quad
 \Lambda=\lambda_{\max}(\bM_0).
 \end{gathered}
\end{equation}
The matrix $\bM_0$ is symmetric positive definite. The two representations in \eqref{eq:resistance} agree by integration by parts and the decay of the exterior Stokes fields.

The limiting Darcy--heat system is
\begin{equation}\label{eq:limit}
\left\{
\begin{aligned}
 \partial_t\rho+\bu\cdot\nabla\rho&=0,\quad \Div\bu=0,
     &&\text{in }(0,T)\times\mathcal O,\\
 \mu(\vartheta)\bM_0\bu+\nabla p&=\rho\bfv-\vartheta\bg,
     &&\text{in }(0,T)\times\mathcal O,\\
 -\kappa_f\Delta\vartheta&=\bu\cdot\bg,
     &&\text{in }(0,T)\times\mathcal O,\\
 \langle\vartheta\rangle_{\mathcal O}&=0,\quad
 \langle p\rangle_{\mathcal O}=0,\quad \rho(0)=\rho_0.
\end{aligned}
\right.
\end{equation}
Here $\rho$, $\bu$, $\vartheta$ and $p$ are the effective density, velocity, temperature fluctuation and pressure. On the torus they are periodic. In a bounded domain they satisfy
\begin{equation*}
 \bu\cdot\bn=0,\qquad \partial_n\vartheta=0
 \quad\text{on }\partial\Omega.
\end{equation*}
Accordingly, the effective boundary condition retains only impermeability, $\bu\cdot\bn=0$. The tangential component of the Darcy velocity remains free. The limiting conductivity is $\kappa_f$, because the total volume of the holes tends to zero while the two reference conductivities stay fixed and positive. Testing the Darcy and heat equations with $\bu$ and $\vartheta$ gives
\[
 \int_{\mathcal O}\mu(\vartheta)\bM_0\bu\cdot \bu\dx
+\kappa_f\int_{\mathcal O}|\nabla\vartheta|^2\dx
=\int_{\mathcal O}\rho \bfv\cdot \bu\dx,
\]
and this nonnegative thermal contribution is retained in the relative-energy argument.

\subsection{Weak solutions}
We first define weak solutions on the torus. The microscopic problem is \eqref{eq:micro-model} with $\mathcal O=\T$, geometry \eqref{eq:geometry}, boundary condition \eqref{eq:micro-reference} and normalization \eqref{eq:normalization}. The coefficient bounds in \eqref{eq:coefficients} make the symmetric stress coercive in the energy space.

For either geometry, extend the density by
\[
 \widehat\rho_\eps=\rho_\eps\mathbf1_{\Oe}+\rho_*\mathbf1_{\He}.
\]
Assume the initial density satisfies the bounds in \eqref{eq:micro-reference}. In the periodic case, define
\[
 \Ve=\{\bv\in H^1(\Oe;\R^3): \bv\text{ is periodic},
       \bv|_{\partial\He}=0,\ \Div \bv=0\}.
\]
We also use
\[
 \Xe=\{\bv\in H^1(\Oe;\R^3):\bv\text{ is periodic},
                         \bv|_{\partial\He}=0\},\quad
 \Hvel=\overline{\Ve}^{\,L^2(\Oe)}.
\]
Periodicity is imposed on the unit torus. General data may vary from one cell to the next. Interior cell faces carry no boundary condition. The admissible initial velocity satisfies $\bu_\eps^0\in\Hvel$ and $\int_{\Oe}\rho_\eps^0|\bu_\eps^0|^2\dx<\infty$. Then $\widetilde \bu_\eps$ is solenoidal on $\T$ and $\partial_t\widehat\rho_\eps+\Div(\widehat\rho_\eps\widetilde \bu_\eps)=0$. In particular $\av{\widetilde \bu_\eps\cdot \bg}=0$. The temperature equation is posed on the whole torus and imposes the transmission conditions for temperature and normal heat flux.

\begin{definition}[Finite-energy weak solution]\label{def:weak}
A triple $(\rho_\eps,\bu_\eps,\vartheta_\eps)$ is a finite-energy weak solution on $(0,T)$ if
\[
 \rho_*\le\widehat\rho_\eps\le\rho^*
 \quad\text{a.e.\ in }(0,T)\times\T,
\]
and
\[
 \begin{gathered}
 \widehat\rho_\eps\in C([0,T];L^q(\T))\quad(1\le q<\infty),\\
 \bu_\eps\in L^\infty(0,T;L^2(\Oe))\cap L^2(0,T;\Ve),\quad
 \vartheta_\eps\in L^2(0,T;H^1(\T)),\quad
 \av{\vartheta_\eps}=0.
 \end{gathered}
\]
For $b\in C^1(\R)$ and any smooth periodic scalar test $\zeta$ on $[0,T)\times\T$ that vanishes for $t$ sufficiently close to $T$, the extended continuity equation holds in the renormalized form
\[
 \int_0^T\!\!\int_{\T} b(\widehat\rho_\eps)
 (\partial_t\zeta+\widetilde \bu_\eps\cdot\nabla\zeta)\dx\dt
 =-\int_{\T} b(\widehat\rho_\eps^0)\zeta(0)\dx.
\]
For any smooth solenoidal vector test $\bphi$ on $[0,T)\times\overline{\Oe}$, periodic on the unit torus, vanishing on $\partial\He$ and vanishing for $t$ sufficiently close to $T$, the momentum identity is
\begin{equation}\label{eq:weak-momentum}
\begin{aligned}
 &-\sig^4\int_0^T\!\!\int_{\Oe}
  \bigl(\rho_\eps \bu_\eps\cdot\partial_t\bphi
  +\rho_\eps \bu_\eps\otimes \bu_\eps:\nabla\bphi\bigr)\dx\dt+2\sig^2\int_0^T\!\!\int_{\Oe}
       \mu(\vartheta_\eps)\sym \bu_\eps:\sym\bphi\dx\dt
 \\
 &=\int_0^T\!\!\int_{\Oe}(\rho_\eps \bfv-\vartheta_\eps \bg)\cdot\bphi\dx\dt
   +\sig^4\int_{\Oe}\rho_\eps^0\bu_\eps^0\cdot\bphi(0)\dx.
\end{aligned}
\end{equation}
For almost every $t\in(0,T)$, the weak heat equation is
\begin{equation}\label{eq:weak-heat}
 \int_{\T}\bar\kappa_\eps\nabla\vartheta_\eps(t)\cdot\nabla\psi\dx
 =\int_{\T}(\widetilde \bu_\eps(t)\cdot \bg)\psi\dx
 \quad\bigl(\psi\in H^1(\T),\ \av\psi=0\bigr).
\end{equation}
For almost every $\tau\in(0,T)$, the energy inequality is
\begin{equation}\label{eq:energy}
\begin{aligned}
 &\frac{\sig^4}{2}\int_{\Oe}\rho_\eps(\tau)|\bu_\eps(\tau)|^2\dx
 +2\sig^2\int_0^\tau\!\!\int_{\Oe}
                \mu(\vartheta_\eps)|\sym \bu_\eps|^2\dx\dt+\int_0^\tau\!\!\int_{\T}
                 \bar\kappa_\eps|\nabla\vartheta_\eps|^2\dx\dt
 \\
 &\le \frac{\sig^4}{2}\int_{\Oe}\rho_\eps^0|\bu_\eps^0|^2\dx
       +\int_0^\tau\!\!\int_{\Oe}\rho_\eps \bfv\cdot \bu_\eps\dx\dt.
\end{aligned}
\end{equation}
The weakly continuous representatives satisfy
\[
 \bu_\eps\in C_{\mathrm{weak}}([0,T];\Hvel),\qquad
 \rho_\eps\bu_\eps\in
 C_{\mathrm{weak}}([0,T];L^2(\Oe;\R^3)),
\]
with initial traces
\[
 \bu_\eps(0)=\bu_\eps^0,\qquad
 (\rho_\eps\bu_\eps)(0)=\rho_\eps^0\bu_\eps^0.
\]
\end{definition}
The renormalized identity follows from periodic spatial mollification and \cite[Lemma~II.1(ii)]{DiPernaLions}, applied to the periodic extensions of the bounded density $\widehat\rho_\eps$ and the solenoidal velocity $\widetilde\bu_\eps\in L^2(0,T;H^1(\T))$. The weakly continuous representatives and initial traces are recovered in \cref{lem:mass-operator,sec:weak-existence}. A distributional microscopic pressure is reconstructed on each fixed fluid domain by the divergence right inverse used in the existence proof.

For later use, the uniform initial-energy hypothesis is
\begin{equation}\label{eq:initial-weighted-bound}
 \sup_{0<\eps<\eps_0}\sig^4\int_{\Oe}\rho_\eps^0|\bu_\eps^0|^2\dx<\infty.
\end{equation}

The periodic fixed-parameter existence theory is developed in \cref{sec:weak-existence}. At fixed $\eps$, the compactness constants may depend on the perforated domain but remain independent of the Galerkin dimension. Once the initial weighted kinetic energies are uniformly bounded, the energy inequality yields the $\eps$-uniform estimates in \eqref{eq:uniform}.

For a bounded domain, the weak formulation uses the same mechanical equations with a different solenoidal space and an insulated thermal boundary condition.

The macroscopic solenoidal space is
\begin{equation*}
 \V_\Omega=\left\{\bv\in L^2(\Omega;\R^3):
       \int_\Omega \bv\cdot\nabla\phi\dx=0
          \text{ for }\phi\in H^1(\Omega)\right\}.
\end{equation*}
This expresses both $\Div \bv=0$ and the zero normal trace. Testing with $\phi=x_i$ also gives $\int_\Omega v_i\dx=0$. Thus the bounded-domain solenoidal space excludes nonzero constant velocities, whereas such velocities are admissible on the torus. The microscopic spaces are
\[
 \V_{\eps,\Omega}=H^1_0(\Oe;\R^3)\cap\{\Div \bv=0\},\quad
 \mathscr H_{\eps,\Omega}
=\overline{\V_{\eps,\Omega}}^{\,L^2(\Oe)}.
\]
Zero extension from $\Oe$ to $\Omega$ sends $\V_{\eps,\Omega}$ into $H^1_0(\Omega)\cap\V_\Omega$. The data and coefficients satisfy \eqref{eq:force-data} and \eqref{eq:coefficients}.

In the bounded domain, the mechanical equations retain the form \eqref{eq:micro-model} with $\mathcal O=\Omega$. The thermal part of the bounded-domain problem is
\begin{equation*}
\left\{
\begin{aligned}
 -\Div(\bar\kappa_\eps\nabla\vartheta_\eps)
      &=\widetilde\bu_\eps\cdot\bg,
     &&\text{in }\Omega,\\
 \bar\kappa_\eps\nabla\vartheta_\eps\cdot\bn&=0,
     &&\text{on }\partial\Omega,\\
 \int_\Omega\vartheta_\eps\dx&=0.
\end{aligned}
\right.
\end{equation*}
Here the velocity vanishes on both $\partial\Omega$ and the obstacle boundaries, while the temperature is defined on the entire conducting composite, including the solid inclusions. The conormal condition is understood variationally. The weak heat equation is
\begin{equation}\label{eq:bounded-weak-heat}
 \int_\Omega\bar\kappa_\eps\nabla\vartheta_\eps\cdot\nabla\psi\dx
 =\int_\Omega(\widetilde \bu_\eps\cdot \bg)\psi\dx
 \quad(\psi\in H^1(\Omega)).
\end{equation}
The source in this Neumann problem has zero mean because
\begin{equation}\label{eq:bounded-compatibility}
 \int_\Omega \bv\cdot \bg\dx
 =\int_\Omega \bv\cdot\nabla F\dx=0
       \quad(\bv\in\V_\Omega).
\end{equation}

\begin{definition}[Bounded-domain finite-energy weak solution]
\label{def:bounded-weak}
A triple $(\rho_\eps,\bu_\eps,\vartheta_\eps)$ is a bounded-domain finite-energy weak solution on $(0,T)$ if
\[
 \rho_*\le\widehat\rho_\eps\le\rho^*
 \quad\text{a.e.\ in }(0,T)\times\Omega,
\]
and
\[
 \begin{gathered}
 \widehat\rho_\eps\in C([0,T];L^q(\Omega))\quad(1\le q<\infty),\quad \bu_\eps\in L^\infty(0,T;L^2(\Oe))\cap L^2(0,T;\V_{\eps,\Omega}),\\
 \vartheta_\eps\in L^2(0,T;H^1(\Omega)),\quad \int_\Omega\vartheta_\eps\dx=0.
 \end{gathered}
\]
For $b\in C^1(\R)$ and any smooth scalar test $\zeta$ on $[0,T)\times\overline\Omega$ that vanishes for $t$ sufficiently close to $T$,
\[
 \int_0^T\!\!\int_\Omega b(\widehat\rho_\eps)
 (\partial_t\zeta+\widetilde\bu_\eps\cdot\nabla\zeta)\dx\dt
 =-\int_\Omega b(\widehat\rho_\eps^0)\zeta(0)\dx.
\]
For any smooth solenoidal vector test $\bphi$ on $[0,T)\times\overline{\Oe}$ that vanishes for $t$ sufficiently close to $T$ and satisfies $\bphi=0$ on $\partial\Oe$,
\[
\begin{aligned}
 &-\sig^4\int_0^T\!\!\int_{\Oe}
  \bigl(\rho_\eps\bu_\eps\cdot\partial_t\bphi
       +\rho_\eps\bu_\eps\otimes\bu_\eps:\nabla\bphi\bigr)\dx\dt+2\sig^2\int_0^T\!\!\int_{\Oe}
       \mu(\vartheta_\eps)\sym\bu_\eps:\sym\bphi\dx\dt\\
 &=\int_0^T\!\!\int_{\Oe}
       (\rho_\eps\bfv-\vartheta_\eps\bg)\cdot\bphi\dx\dt
   +\sig^4\int_{\Oe}\rho_\eps^0\bu_\eps^0\cdot\bphi(0)\dx.
\end{aligned}
\]
For almost every $t\in(0,T)$, $\vartheta_\eps(t)$ satisfies \eqref{eq:bounded-weak-heat}. For almost every $\tau\in(0,T)$,
\[
\begin{aligned}
 &\frac{\sig^4}{2}\int_{\Oe}\rho_\eps(\tau)|\bu_\eps(\tau)|^2\dx
 +2\sig^2\int_0^\tau\!\!\int_{\Oe}
       \mu(\vartheta_\eps)|\sym\bu_\eps|^2\dx\dt+\int_0^\tau\!\!\int_\Omega
       \bar\kappa_\eps|\nabla\vartheta_\eps|^2\dx\dt\\
 &\le \frac{\sig^4}{2}\int_{\Oe}\rho_\eps^0|\bu_\eps^0|^2\dx
       +\int_0^\tau\!\!\int_{\Oe}\rho_\eps\bfv\cdot\bu_\eps\dx\dt.
\end{aligned}
\]
The weakly continuous representatives satisfy
\[
 \bu_\eps\in C_{\mathrm{weak}}([0,T];\mathscr H_{\eps,\Omega}),\qquad
 \rho_\eps\bu_\eps\in
 C_{\mathrm{weak}}([0,T];L^2(\Oe;\R^3)),
\]
with initial traces
\[
 \bu_\eps(0)=\bu_\eps^0,\qquad
 (\rho_\eps\bu_\eps)(0)=\rho_\eps^0\bu_\eps^0.
\]
\end{definition}

The bounded-domain fixed-parameter construction is given in \cref{app:bounded-existence}. Its Galerkin argument uses the insulated Neumann problem for the temperature together with the no-slip velocity space. 

\subsection{Main results}\label{sec:main-results}
We next state the effective well-posedness and quantitative homogenization results. The periodic problem isolates the interior mechanism, while the bounded-domain statements display the additional wall contribution.

Consider first the periodic effective problem and the assumptions that provide the reference solution for homogenization. The system is \eqref{eq:limit} on $\T$, with periodic unknowns and mean-zero temperature and pressure. Its resistance matrix $\bM_0$ and eigenvalues $\lambda,\Lambda$ are defined in \eqref{eq:resistance}. We equip the solenoidal space $\V=\{\bv\in L^2(\T;\R^3):\Div \bv=0\}$ with the $L^2$ norm. On mean-zero scalars we use $\Lop=-\kappa_f\Delta:H^2_\diamond(\T)\longrightarrow L^2_\diamond(\T)$. Define $\Bop:\V\longrightarrow L^2_\diamond(\T)$ by $\Bop\bv=\bv\cdot\bg$. Since $\av{\Bop\bv}=0$ for $\bv\in\V$, the constant
\begin{equation*}
 C_\infty=\|\Lop^{-1}\Bop\|_{\mathcal L(\V,L^\infty)},\quad
 c_0=\mu_*\lambda,\quad H_0=\rho^*\|\bfv\|_2
\end{equation*}
is finite. Introduce the macroscopic smallness parameter
\begin{equation*}
 q_0=\frac{\Lambda L_\mu C_\infty H_0}{c_0^2}.
\end{equation*}

At each time the elliptic Darcy--heat subsystem determines $(\bu,\vartheta,p)$ from the density. The bound on $q_0$ makes this instantaneous map contractive at the $L^2$ level and allows the transport equation to be closed locally in time.

\begin{theorem}[Local strong solutions]\label{thm:local}
Let $s\ge3$ be an integer and let $\rho_0\in H^s(\T)$ satisfy $0<\rho_*\le\rho_0\le\rho^*$. Assume \eqref{eq:force-data}, \eqref{eq:coefficients} and $q_0\le1/2$. Then there exists $T_*>0$ such that \eqref{eq:limit} admits a unique strong solution on $[0,T_*]$ satisfying
\begin{align*}
 \rho&\in C([0,T_*];H^s)\cap C^1([0,T_*];H^{s-1}),\\
 \bu&\in C([0,T_*];H^s(\T;\R^3))
             \cap C^1([0,T_*];H^{s-1}(\T;\R^3)),\\
 \vartheta&\in C([0,T_*];H^{s+2})
             \cap C^1([0,T_*];H^{s+1}),\quad
 p\in C([0,T_*];H^{s+1}).
\end{align*}
The density bounds are preserved and $\sup_{t\le T_*}\|\bu(t)\|_2\le H_0/c_0$. The time $T_*$ and the Sobolev bounds depend only on the stated data and finitely many derivatives of $\mu$ on the attained temperature interval.
\end{theorem}

The perforated Poincar\'e inequality, proved in \cite[Lemma~3.4.1]{AllaireII} and in the present dilute scaling in \cite[Lemma~2.3]{HLO}, gives a constant $C_P$ independent of $\eps$ such that
\begin{equation}\label{eq:poincare}
 \|\bv\|_{L^2(\Oe)}\le C_P\sig\|\sym \bv\|_{L^2(\Oe)}
 \quad(\bv\in\Ve).
\end{equation}
For solenoidal fields, integration by parts on the unit torus
after zero extension through the holes gives $2\|\sym \bv\|_2^2=\|\nabla \bv\|_2^2$. For a measurable scalar coefficient $a$ with $\kappa_*\le a\le\kappa^*$ and $b\in L^2_\diamond(\T)$, let $\psi\in H^1_\diamond(\T)$ solve the corresponding elliptic problem. Let $C_H$ be a common constant in the estimate
\begin{equation}\label{eq:heat-operator-bound}
 \|\psi\|_\infty\le C_H\|b\|_2,\quad
 -\Div(a\nabla\psi)=b,\quad\av\psi=\av b=0,\quad
 \kappa_*\le a\le\kappa^*.
\end{equation}
Such a constant independent of the measurable coefficient $a$ exists by \cref{lem:scalar}. Put
\begin{equation*}
 C_g=C_H\|\bg\|_\infty,\quad
 W_0=\frac{C_P\mu^*\Lambda H_0}{2\mu_*c_0},\quad
 \chi_0=\frac{2L_\mu W_0 C_g C_P}{\mu_*},
\end{equation*}
where we impose $\chi_0\le1/2$ for the viscosity-error absorption.

\begin{remark}[Roles and non-emptiness of the coupling conditions]

The condition $q_0\le1/2$ makes the instantaneous macroscopic solution map contractive and allows its differentiation. The condition $\chi_0\le1/2$ permits absorption of the viscosity difference in relative energy.

With all other data fixed, replacing $\bfv$ by $\bfv_\delta=\delta \bfv$, $\delta>0$, gives $H_0(\bfv_\delta)=\delta H_0(\bfv)$ and hence $q_0(\bfv_\delta)=\delta q_0(\bfv)$, $\chi_0(\bfv_\delta)=\delta\chi_0(\bfv)$. The same identities hold for the domain-dependent constants in \eqref{eq:bounded-smallness}. Thus both coupling conditions allow nonzero forcing. A state with nonconstant temperature and viscosity is exhibited in \cref{sec:nonconstant-state}.
\end{remark}

For a strong solution of \eqref{eq:limit}, write $m(t,\bx)=\mu(\vartheta(t,\bx))$. At each time define $\bw_\eps\in\Ve$ by
\begin{equation}\label{eq:lift}
 2\sig^2\int_{\Oe}m\,\sym \bw_\eps:\sym\bphi\dx
 =\int_{\Oe}m\bM_0\bu\cdot\bphi\dx
 \quad(\bphi\in\Ve).
\end{equation}
Coercivity defines this lifting uniquely. Set
\begin{equation*}
 \eta_\eps^2=\|\widetilde{\bw}_\eps-\bu\|_{L^2((0,T)\times\T)}^2,
 \quad \deltae=\frac{\cellsize}{\sig}+\sig.
\end{equation*}
The quantitative approximation is established in \cref{sec:lift}, where \cref{lem:geometric-tools} is combined with the coefficient and cell-pressure commutators generated by the variable weight. The underlying two-scale construction is recorded in \cref{app:periodic-geometry}.

For the initial data, the natural discrepancy is
\begin{equation}\label{eq:natural-initial}
 \Dinit=\frac{\sig^4}{2}\int_{\Oe}\rho_\eps^0
                 |\bu_\eps^0-\bu(0)|^2\dx
              +\frac12\|\widehat\rho_\eps^0-\rho_0\|_2^2.
\end{equation}
For the proof we also use
\begin{equation}\label{eq:initial}
 \E_\eps^0=\frac{\sig^4}{2}\int_{\Oe}\rho_\eps^0
                  |\bu_\eps^0-\bw_\eps(0)|^2\dx
              +\frac12\|\widehat\rho_\eps^0-\rho_0\|_2^2.
\end{equation}
Lemma~\ref{lem:initial-equivalence} compares these discrepancies up to an additive $C\sig^4$. For brevity set
\begin{equation}\label{eq:error-norm}
\begin{aligned}
 \mathfrak E_\eps(T)
 &=\sup_{0\le t\le T}\|\widehat\rho_\eps(t)-\rho(t)\|_2^2
      +\|\widetilde \bu_\eps-\bu\|_{L^2_tL^2_x}^2+\|\vartheta_\eps-\vartheta\|_{L^2_tH^1_x}^2
      +\|\vartheta_\eps-\vartheta\|_{L^2_tL^\infty_x}^2.
\end{aligned}
\end{equation}

The decomposition through $\bw_\eps$ separates the geometric corrector error from the relative error between the microscopic velocity and the lifted macroscopic field. The next estimate controls the latter by the initial discrepancy, the lifting error and the conductivity defect on the holes.

\begin{theorem}[Relative-energy stability estimate]\label{thm:stability}
Assume the periodic geometry defined by \eqref{eq:continuous-parameter}--\eqref{eq:geometry}, the data in \eqref{eq:force-data}, the coefficients in \eqref{eq:coefficients} and the initial density bounds in \eqref{eq:micro-reference}. Let $(\rho,\bu,\vartheta,p)$ be a macroscopic solution on $[0,T]$ with the regularity stated in \cref{thm:local}, satisfying $\rho_*\le\rho\le\rho^*$, and suppose $\chi_0\le1/2$. Let $(\rho_\eps,\bu_\eps,\vartheta_\eps)$ be any finite-energy weak solution in \cref{def:weak}. For fixed $q>3$ there exists a constant $C_{T,q}$ and a threshold $\eps_1>0$ such that, for all $0<\eps<\eps_1$,
\begin{equation}\label{eq:main-estimate}
 \mathfrak E_\eps(T)
 \le C_{T,q}\left(\Dinit+\eta_\eps^2
       +|\kappa_s-\kappa_f|^2h_\eps^{2/q}+\sig^3\right).
\end{equation}
The estimate also holds with $\E_\eps^0$ in place of $\Dinit$. The constants $C_{T,q}$ and $\eps_1$ are independent of $\eps$ and the microscopic weak solution.
\end{theorem}
The constants in \cref{thm:stability} depend on $T,q$, the fixed data and the reference-solution norms on $[0,T]$. The lifting error appears explicitly in \eqref{eq:main-estimate}. The stability argument requires only $\chi_0\le1/2$. Under $q_0\le1/2$, \cref{thm:local} supplies the reference solution for $0<T\le T_*$.

The explicit periodic rate follows from \cref{thm:stability} together with the Stokes approximation \eqref{eq:eta-rate}. With $q=4$, the conductivity contribution is of no larger order than the geometric error.

\begin{corollary}[Quantitative homogenization]\label{thm:homogenization}
Under the assumptions of \cref{thm:stability},
\begin{equation}\label{eq:explicit-estimate}
 \mathfrak E_\eps(T)
 \le C_T\left(\Dinit+\eps^{\alpha-1}+\eps^{3-\alpha}\right).
\end{equation}
If $\Dinit\to0$, the density converges in $L^\infty_tL^2_x$, the velocity in $L^2_tL^2_x$ and the temperature in both $L^2_tH^1_x$ and $L^2_tL^\infty_x$. Moreover,
\begin{equation}\label{eq:viscosity-convergence}
 \|\mu(\vartheta_\eps)-\mu(\vartheta)\|_{L^2_tL^\infty_x}
 \to0.
\end{equation}
\end{corollary}

Since \eqref{eq:explicit-estimate} bounds a sum of squares, its exponents translate into rates for each norm by taking square roots. The initial discrepancy enters through its own decay rate.

\begin{corollary}[Rates for the norms]\label{cor:norm-rate}
Under the hypotheses of \cref{thm:homogenization}, if $\Dinit\le C\eps^{\gamma_0}$ for some $\gamma_0>0$, then the squared error in \eqref{eq:error-norm} is $O(\eps^{\gamma_{\mathrm{sq}}})$, where
\[
 \gamma_{\mathrm{sq}}=\min\{\gamma_0,\alpha-1,3-\alpha\}.
\]
Each corresponding norm is $O(\eps^{\gamma_{\mathrm{norm}}})$, where $\gamma_{\mathrm{norm}}=\gamma_{\mathrm{sq}}/2$. More generally, under the hypotheses of \cref{thm:stability}, if $\Dinit\le C\eps^{\gamma_0}$ and $\eta_\eps\le C\eps^{\beta_{\mathrm{lift}}}$ for some $\gamma_0,\beta_{\mathrm{lift}}>0$, the squared-error exponent is
\begin{equation}\label{eq:conditional-rate}
 \min\left\{\gamma_0,\,2\beta_{\mathrm{lift}},\,
         \frac{6(\alpha-1)}q,\,\frac{3(3-\alpha)}2\right\},\quad q>3.
\end{equation}
When $\kappa_s=\kappa_f$, the conductivity entry is omitted from this last minimum.
\end{corollary}

\begin{remark}[Admissible families and the initial layer]
The velocity contribution to \eqref{eq:natural-initial} is weighted by $\sig^4$. In particular, the initial velocities need not converge in unweighted $L^2$. For example, take $\rho_\eps^0=\rho_0|_{\Oe}$ and $\bu_\eps^0=0\in\Hvel$. The density extension differs from $\rho_0$ only on the holes and
\begin{equation*}
 \Dinit\le C\bigl(\sig^4+h_\eps\bigr)\to0.
\end{equation*}
These data satisfy \eqref{eq:initial-weighted-bound} as well. If $\bu(0)\ne0$, their initial velocity discrepancy is of order one. This scaling permits an initial layer and is compatible with the asserted $L^2(0,T;L^2)$ velocity convergence.
\end{remark}

The bounded geometry retains the interior resistance $\bM_0$ and effective conductivity $\kappa_f$. The outer boundary $\partial\Omega$ changes the velocity trace and the convergence rate. The effective system satisfies
\begin{equation}\label{eq:bounded-limit}
\left\{
\begin{aligned}
 \partial_t\rho+\bu\cdot\nabla\rho&=0,\quad \Div\bu=0,
     &&\text{in }(0,T)\times\Omega,\\
 \mu(\vartheta)\bM_0\bu+\nabla p&=\rho\bfv-\vartheta\bg,
     &&\text{in }(0,T)\times\Omega,\\
 -\kappa_f\Delta\vartheta&=\bu\cdot\bg,
     &&\text{in }(0,T)\times\Omega,\\
 \bu\cdot\bn&=0,\quad\partial_n\vartheta=0,
     &&\text{on }(0,T)\times\partial\Omega,\\
 \int_\Omega\vartheta\dx&=0,\quad
 \int_\Omega p\dx=0,\quad \rho(0)=\rho_0.
\end{aligned}
\right.
\end{equation}
The normal condition in \eqref{eq:bounded-limit} makes the Darcy velocity tangent to the boundary, so no inflow density is prescribed. We next define the bounded-domain elliptic operators and the coupling conditions for local strong solutions.

The effective system has the same instantaneous thermal coupling as on the torus, with the inverse Neumann Laplacian replacing the periodic inverse. Let $\Lop_\Omega=-\kappa_f\Delta_N$, where $\Delta_N$ has domain $\{\psi\in H^2_\diamond(\Omega):\partial_n\psi=0\}$ and codomain $L^2_\diamond(\Omega)$. Define $\Bop_\Omega:\V_\Omega\longrightarrow L^2_\diamond(\Omega)$ and $\Jop_\Omega$ by
\[
 \Bop_\Omega\bv=\bv\cdot\bg,\quad
 \Jop_\Omega=\Lop_\Omega^{-1}\Bop_\Omega.
\]
By \eqref{eq:bounded-compatibility}, the source has zero mean for $\bv\in\V_\Omega$ and Neumann regularity gives $\Jop_\Omega:\V_\Omega\longrightarrow H^2_\diamond(\Omega)$. We fix uniform constants $C_{P,\Omega}$ and $C_{H,\Omega}$ such that
\[
 \|\bv\|_2\le C_{P,\Omega}\sig\|\sym \bv\|_2
       \quad(\bv\in\V_{\eps,\Omega}),\quad
 \|\psi\|_\infty\le C_{H,\Omega}\|b\|_2,
\]
where $b\in L^2(\Omega)$ and the second inequality concerns the mean-zero variational Neumann problem $-\Div(a\nabla\psi)=b$ with $\kappa_*\le a\le\kappa^*$ and $\int_\Omega b\dx=0$. Lemmas~\ref{lem:bounded-tools} and~\ref{lem:bounded-scalar} provide these estimates. Put $C_{g,\Omega}=C_{H,\Omega}\|\bg\|_\infty$ and define
\begin{equation}\label{eq:bounded-smallness}
\begin{gathered}
 C_{\infty,\Omega}=\|\Jop_\Omega\|_{\mathcal L(\V_\Omega,L^\infty)},
 \quad H_{0,\Omega}=\rho^*\|\bfv\|_{L^2(\Omega)},\\
 q_{0,\Omega}=\frac{\Lambda L_\mu C_{\infty,\Omega}H_{0,\Omega}}{c_0^2},
 \quad W_{0,\Omega}
   =\frac{C_{P,\Omega}\mu^*\Lambda H_{0,\Omega}}{2\mu_*c_0},\quad \chi_{0,\Omega}
      =\frac{2L_\mu W_{0,\Omega}C_{g,\Omega}C_{P,\Omega}}{\mu_*}.
\end{gathered}
\end{equation}
As in the periodic problem, $q_{0,\Omega}$ controls the construction of the macroscopic solution, whereas $\chi_{0,\Omega}$ controls the viscosity error in homogenization. The contraction condition, combined with conormal pressure estimates and transport along a tangent velocity, yields the following local result.
\begin{theorem}[Local bounded-domain strong solutions]\label{thm:bounded-local}
Let $s\ge4$ be a fixed integer and assume \eqref{eq:force-data} and \eqref{eq:coefficients}. Suppose that $\partial\Omega$ is of class $C^{s+1,1}$ and $\rho_0\in H^s(\Omega)$ satisfies $0<\rho_*\le\rho_0\le\rho^*$. Assume $q_{0,\Omega}\le1/2$. There exists $T_{*,\Omega}>0$ such that \eqref{eq:bounded-limit} has a unique strong solution on $[0,T_{*,\Omega}]$, with
\begin{align*}
 \rho&\in C_tH^s(\Omega)\cap C^1_tH^{s-1}(\Omega),\\
 \bu&\in C_t(H^s(\Omega;\R^3)\cap\V_\Omega)
                         \cap C^1_tH^{s-1}(\Omega;\R^3),\\
 \vartheta&\in C_tH^{s+2}_\diamond(\Omega)
                         \cap C^1_tH^{s+1}(\Omega),\quad
 p\in C_tH^{s+1}_\diamond(\Omega).
\end{align*}
The density bounds are preserved and $\sup_t\|\bu(t)\|_2\le H_{0,\Omega}/c_0$.
\end{theorem}

We construct this solution in \cref{sec:bounded-macro}. The assumption $s\ge4$ also yields the $W^{2,\infty}$ velocity bound required by the boundary correction.

For the conditional stability theorem, we isolate the regularity of the macroscopic reference solution actually used by the boundary correction and the time-dependent lifting.

\begin{definition}[Admissible reference solution in the bounded domain]
\label{def:bounded-reference}
Let $T>0$. A quadruple $(\rho,\bu,\vartheta,p)$ solving \eqref{eq:bounded-limit} on $[0,T]$ is called an admissible reference solution if
\[
\begin{gathered}
 \rho\in L^\infty(0,T;W^{1,\infty}(\Omega)),\qquad
 \rho_*\le \rho\le \rho^*,\\
 \bu\in L^\infty(0,T;W^{2,\infty}(\Omega;\R^3))
      \cap W^{1,\infty}(0,T;L^2(\Omega;\R^3)),\\
 \vartheta\in L^\infty(0,T;W^{1,\infty}(\Omega))
      \cap W^{1,\infty}(0,T;L^\infty(\Omega)),\qquad
 p\in L^\infty(0,T;H^1_\diamond(\Omega)).
\end{gathered}
\]
The continuity and momentum equations hold distributionally, with the displayed regularity. $\Div\bu=0$ and $\bu\cdot\bn=0$ hold in the weak normal-trace sense. The heat equation with $\partial_n\vartheta=0$ is understood variationally on $H^1(\Omega)$, with its stated zero mean.
\end{definition}
The heat source belongs to $L^2(\Omega)$. On a $C^{3,\beta}$ domain, \cref{lem:constant-conormal} gives $H^2$ Neumann regularity and hence a trace interpretation of the normal derivative. Set
\[
\begin{aligned}
 \mathcal N_{T,\Omega}:={}&
 \|\rho\|_{L^\infty_tW^{1,\infty}_x}
 +\|\bu\|_{L^\infty_tW^{2,\infty}_x}
 +\|\partial_t\bu\|_{L^\infty_tL^2_x}+\|\vartheta\|_{L^\infty_tW^{1,\infty}_x}
 +\|\partial_t\vartheta\|_{L^\infty_tL^\infty_x}
 +\|p\|_{L^\infty_tH^1_x}.
\end{aligned}
\]
The transport equation and the displayed bounds imply $\partial_t\rho=-\bu\cdot\nabla\rho\in L^\infty(0,T;L^2(\Omega))$. We use the resulting representative $\rho\in W^{1,\infty}(0,T;L^2(\Omega))\subset C([0,T];L^2(\Omega))$, with $\rho(0)=\rho_0$. The $L^2$-continuous representative of $\bu$ and the $L^\infty$-continuous representative of $\vartheta$ are used at time zero. Their essential spatial bounds extend to the endpoints by weak-* compactness of the bounded spatial derivatives and identification with these continuous representatives. Thus the initial discrepancy and the density supremum in \eqref{eq:bounded-error} are well defined.

The boundary $\partial\Omega$ prevents direct use of the periodic comparison field since the Darcy velocity is tangent to the wall but need not vanish there. We define the exact Stokes lifting as before and estimate its error by a boundary correction in \cref{sec:bounded-cutoff}. Let $m\in W^{1,\infty}(\Omega)$ satisfy $m_*\le m\le m^*$ for fixed constants $0<m_*\le m^*<\infty$. Let $\bv\in\V_\Omega$. We define $\bw_\eps[m,\bv]\in\V_{\eps,\Omega}$ as the solution of
\begin{equation}\label{eq:bounded-lift}
 2\sig^2\int_{\Oe}m\sym \bw_\eps[m,\bv]:\sym\bphi\dx
       =\int_{\Oe}m\bM_0\bv\cdot\bphi\dx
           \quad(\bphi\in\V_{\eps,\Omega}).
\end{equation}
For a macroscopic solution set $m=\mu(\vartheta)$, $\bw_\eps(t)=\bw_\eps[m(t),\bu(t)]$ and
\[
 \eta_{\eps,\Omega}^2
=\|\widetilde{\bw}_\eps-\bu\|_{L^2(0,T;L^2(\Omega))}^2.
\]
For the bounded-domain comparison, the natural initial discrepancy is
\begin{equation*}
 \mathcal D_{\eps,\Omega}^0
 =\frac{\sig^4}{2}\int_{\Oe}\rho_\eps^0
                         |\bu_\eps^0-\bu(0)|^2\dx
       +\frac12\|\widehat\rho_\eps^0-\rho_0\|_{L^2(\Omega)}^2,
\end{equation*}
and set
\begin{equation}\label{eq:bounded-error}
\begin{aligned}
 \mathfrak E_{\eps,\Omega}(T)
  &=\sup_{0\le t\le T}\|\widehat\rho_\eps(t)-\rho(t)\|_2^2
                  +\|\widetilde \bu_\eps-\bu\|_{L^2_tL^2_x}^2+\|\vartheta_\eps-\vartheta\|_{L^2_tH^1_x}^2
                  +\|\vartheta_\eps-\vartheta\|_{L^2_tL^\infty_x}^2,
\end{aligned}
\end{equation}
where all four ambient spatial norms are on $\Omega$.

The admissible initial class contains sequences with $\mathcal D_{\eps,\Omega}^0\to0$. For example, $\rho_\eps^0=\rho_0|_{\Oe}$ and $\bu_\eps^0=0\in\mathscr H_{\eps,\Omega}$ give $\mathcal D_{\eps,\Omega}^0\le C(\sig^4+h_\eps)\to0$. Thus the bounded-domain theorem, like its periodic counterpart, allows an $O(1)$ initial velocity mismatch because the kinetic contribution is weighted by $\sig^4$.

The lifting \eqref{eq:bounded-lift}, the Neumann temperature estimate from \cref{sec:bounded} and the kinetic--density comparison from the periodic argument yield stability in terms of the lifting error. The wall-corrected Stokes estimate then converts this stability bound into the explicit bounded-domain rate.

\begin{theorem}[Bounded-domain stability and homogenization]
\label{thm:bounded-homogenization}
Assume the bounded geometry \eqref{eq:bounded-geometry} and the boundary regularity $\partial\Omega\in C^{3,\beta}$ stated in \cref{sec:problem-formulation}, together with the data and constitutive assumptions \eqref{eq:force-data} and \eqref{eq:coefficients}. Let $(\rho,\bu,\vartheta,p)$ be an admissible reference solution in the sense of \cref{def:bounded-reference} on $[0,T]$ and assume $\chi_{0,\Omega}\le1/2$. Let $(\rho_\eps,\bu_\eps,\vartheta_\eps)$ be any finite-energy weak solution in \cref{def:bounded-weak}. For fixed $q>3$, there exists $\eps_1>0$ such that, whenever $0<\eps<\eps_1$,
\begin{equation}\label{eq:bounded-stability}
 \mathfrak E_{\eps,\Omega}(T)
  \le C_{T,\Omega,q}\left(\mathcal D_{\eps,\Omega}^0
      +\eta_{\eps,\Omega}^2
      +|\kappa_s-\kappa_f|^2h_\eps^{2/q}+\sig^3\right).
\end{equation}
The squared error also satisfies
\begin{equation}\label{eq:bounded-final-rate}
\displaystyle
 \mathfrak E_{\eps,\Omega}(T)
 \le C_{T,\Omega}\left(\mathcal D_{\eps,\Omega}^0
      +\eps^{\alpha-1}+\eps^{(3-\alpha)/2}\right).
\end{equation}
If $\mathcal D_{\eps,\Omega}^0\to0$, there is strong convergence of density in $L^\infty_tL^2_x$, velocity in $L^2_tL^2_x$, temperature in $L^2_tH^1_x\cap L^2_tL^\infty_x$ and viscosity in $L^2_tL^\infty_x$. The constants and the small-$\eps$ threshold depend only on the fixed domain, $T$, the density, force, geometric and constitutive bounds and $\mathcal N_{T,\Omega}$. The constant in \eqref{eq:bounded-stability} also depends on $q$.
\end{theorem}
Under the hypotheses of \cref{thm:bounded-local}, the Sobolev embeddings $H^s(\Omega)\hookrightarrow W^{2,\infty}(\Omega)$ and $H^{s+2}(\Omega)\hookrightarrow W^{1,\infty}(\Omega)$ give an admissible reference solution for $0<T\le T_{*,\Omega}$. The estimates are uniform over the microscopic weak-solution class.

The proof in \cref{sec:bounded-homogenization} combines the bounded lifting and temperature estimates. The additional wall contribution is quantified in \cref{prop:bounded-stokes-rate}. Taking square roots gives the norm rates in \cref{cor:bounded-rates}. If the reference velocity has zero full boundary trace, the boundary correction disappears and the periodic geometric rate is recovered.

\begin{corollary}[Norm rates and the special boundary-compatible case]
\label{cor:bounded-rates}
Under the hypotheses of \cref{thm:bounded-homogenization}, if $\mathcal D_{\eps,\Omega}^0\le C\eps^{\gamma_0}$ for some $\gamma_0>0$, the squared error in \eqref{eq:bounded-error} has exponent
\[
 \gamma_{\mathrm{sq}}^{\Omega}
=\min\{\gamma_0,\alpha-1,(3-\alpha)/2\}.
\]
Each norm in \eqref{eq:bounded-error} has exponent $\gamma_{\mathrm{sq}}^{\Omega}/2$. If $\bu(t)|_{\partial\Omega}=0$ for $t\in[0,T]$, then
\[
 \mathfrak E_{\eps,\Omega}(T)\le C_{T,\Omega}\bigl(\mathcal D_{\eps,\Omega}^0+\eps^{\alpha-1}+\eps^{3-\alpha}\bigr).
\]
Under the same initial-rate hypothesis, each norm then has exponent $\frac12\min\{\gamma_0,\alpha-1,3-\alpha\}$.
\end{corollary}

At $\alpha=2$, with initial discrepancies $o(\eps)$, \eqref{eq:explicit-estimate} and \eqref{eq:bounded-final-rate} give norm bounds of $O(\eps^{1/2})$ on the torus and $O(\eps^{1/4})$ in a bounded domain. The comparisons hold on the existence interval of the reference solution, while the microscopic weak solutions exist on every finite interval.

The thermal and lifting estimates are proved in \cref{sec:heat,sec:lift}, followed by the periodic stability estimate in \cref{sec:periodic-quantitative} and the reference-solution construction in \cref{sec:periodic-macro}. The bounded-domain tools and comparison are developed in \cref{sec:bounded,sec:bounded-quantitative}. The appendices contain the periodic cell construction and the fixed-parameter weak-existence proofs.

\section{Preliminaries and thermal estimates}\label{sec:heat}
This section establishes the elliptic estimate that converts a velocity error into a temperature and viscosity error. Its constants depend only on the uniform ellipticity bounds and the ambient domain, so they remain uniform with respect to the $\eps$-dependent fluid--solid interfaces.

\begin{lemma}\label{lem:scalar}
Let $a$ be a measurable scalar coefficient on $\T$ with $\kappa_*\le a\le\kappa^*$. If $b\in L^2(\T)$ has mean zero, $\bG\in L^q(\T;\R^3)$ with $q>3$ and $\psi\in H^1(\T)$ is the mean-zero solution of
\[
-\Div(a\nabla\psi)=b+\Div\bG,
\]
then
\begin{equation}\label{eq:scalar}
 \|\psi\|_{H^1}+\|\psi\|_\infty
 \le C_q(\|b\|_2+\|\bG\|_q).
\end{equation}
The constant depends only on the fixed torus, $\kappa_*,\kappa^*$ and $q$. In particular it is independent of any interfaces of $a$.
\end{lemma}
\begin{proof}
Existence and the $H^1$ estimate follow from the Lax--Milgram theorem and the mean-zero Poincar\'e inequality. We prove the $L^\infty$ estimate by the truncation argument of \cite{Stampacchia}. Put $M=\|b\|_2+\|\bG\|_q$. If $M=0$, coercivity and the zero mean give $\psi=0$, so assume $M>0$. The energy estimate gives $\|\psi\|_2\le C_1M$. Since $|\T|=1$, Chebyshev's inequality with $k_0=\sqrt{2}C_1M$ gives $|\{\psi>k_0\}|\le1/2$. For $k\ge k_0$ set $A_k=\{\psi>k\}$ and $\varphi_k=(\psi-k)_+$. Since $\varphi_k$ vanishes on a set of measure at least $1/2$, $\|\varphi_k\|_6\le C\|\nabla\varphi_k\|_2$. Since $\av b=0$, testing with $\varphi_k-\av{\varphi_k}$ gives the same identity as testing with $\varphi_k$. Hence
\[
 \kappa_*\|\nabla\varphi_k\|_2^2
 \le C\|b\|_2|A_k|^{1/3}\|\nabla\varphi_k\|_2
+\|\bG\|_q|A_k|^{1/2-1/q}\|\nabla\varphi_k\|_2.
\]
With $\gamma=\min\{1/3,\,1/2-1/q\}>1/6$, it follows that
\[
(\ell-k)|A_\ell|^{1/6}\le\|\varphi_k\|_6
 \le CM|A_k|^\gamma\quad(\ell>k).
\]
Thus $|A_\ell|\le[CM/(\ell-k)]^6|A_k|^{6\gamma}$. Set $a=6\gamma>1$, $\nu=6/(a-1)$ and choose $d>0$ so that $(C/d)^6 2^{6+\nu}\le1$. For
\[
 k_j=k_0+dM(1-2^{-j}),\qquad Y_j=|A_{k_j}|,
\]
the level-set estimate gives
\[
 Y_{j+1}\le(C/d)^6 2^{6(j+1)}Y_j^a.
\]
Since $Y_0\le1/2$ and $\nu(a-1)=6$, induction yields $Y_j\le Y_0 2^{-\nu j}$. Hence $\psi\le k_0+dM$ almost everywhere, with $d$ independent of $M$. Applying the same argument to $-\psi$ proves \eqref{eq:scalar}.
\end{proof}

\begin{proof}[Proof of \eqref{eq:heat-operator-bound}]
Apply \cref{lem:scalar} with $\bG=0$ and, for definiteness, $q=4$. Estimate \eqref{eq:scalar} gives $\|\psi\|_\infty\le C_4\|b\|_2$. We take $C_H=C_4$, which depends only on $\kappa_*,\kappa^*$ and the fixed torus.
\end{proof}

The energy inequality, \eqref{eq:poincare} and the mean-zero Poincar\'e inequality give, whenever the initial weighted kinetic energy is bounded by a common constant $E_*$ independent of $\eps$,
\begin{equation}\label{eq:uniform}
 \sig^2\|\sqrt{\rho_\eps}\bu_\eps\|_{L^\infty_tL^2_x}
 +\|\widetilde \bu_\eps\|_{L^2_tL^2_x}
 +\sig\|\nabla \bu_\eps\|_{L^2_tL^2_x}
 +\|\vartheta_\eps\|_{L^2_tH^1_x}\le C_T.
\end{equation}
The constant $C_T$ depends on $E_*$ and the fixed data, not on the individual perforated domain. Indeed, the work of $\bfv$ is bounded by $\mu_*\sig^2\|\sym \bu_\eps\|_{L^2_{t,x}}^2+C_T$ and is absorbed. \cref{lem:scalar} further gives
\begin{equation*}
 \|\vartheta_\eps(t)\|_\infty
 \le C_g\|\widetilde \bu_\eps(t)\|_2
\end{equation*}
for almost every $t\in(0,T)$. Together with \eqref{eq:uniform}, this gives an $\eps$-uniform bound in $L^2_tL^\infty_x$. For each fixed $\eps$, the energy inequality also yields a bound in $L^\infty_tL^\infty_x$, but this bound need not be uniform as $\eps\to0$. More precisely, under the same common initial energy bound $E_*$ used in \eqref{eq:uniform},
\[
 \|\vartheta_\eps\|_{L^\infty_tL^\infty_x}
 \le \frac{C_gC_T}{\sqrt{\rho_*}}\sig^{-2}.
\]

In either geometry, write $z_\eps=\vartheta_\eps-\vartheta$ for the temperature error and $\bv_\eps=\widetilde \bu_\eps-\bu$ for the velocity error. Subtracting the two heat equations gives
\begin{equation}\label{eq:heat-difference}
 -\Div(\bar\kappa_\eps\nabla z_\eps)
 =\bv_\eps\cdot \bg+
  \Div\bigl((\bar\kappa_\eps-\kappa_f)\nabla\vartheta\bigr).
\end{equation}
Since $\bv_\eps$ is solenoidal and $\bg=\nabla F$, the scalar source has zero mean. By linearity, we estimate its contribution using \eqref{eq:heat-operator-bound} and the divergence source using \cref{lem:scalar}. We obtain
\begin{equation}\label{eq:temperature-error-infinity}
 \|z_\eps(t)\|_\infty
 \le C_g\|\bv_\eps(t)\|_2
 +C_q|\kappa_s-\kappa_f|h_\eps^{1/q}
                      \|\nabla\vartheta(t)\|_\infty.
\end{equation}
We use this temperature bound to estimate the viscosity residual in the lifting comparison.

\section{Periodic correctors and the frozen Stokes lifting}\label{sec:lift}
For the periodic problem, the comparison velocity must encode the dilute Stokes resistance while accommodating the temperature-dependent viscosity. The construction starts from the unit-viscosity cell fields, restores solenoidality and subsequently quantifies the commutators generated by the frozen coefficient.

\subsection{Cell correctors and comparison fields}
We begin with the unit-viscosity geometric estimates. The dual space $\Xe^*$ is defined using the full $H^1(\Oe)$ norm on $\Xe$. Its test functions have zero trace on the holes and may be nonsolenoidal. Thus $\Xe^*$ is the dual of the zero-hole-trace $H^1$ space and, in the periodic setting, plays the role of the $H^{-1}(\Oe)$ space used in the corrector estimate. The periodic construction is described in \cref{app:periodic-geometry}.

\begin{lemma}[Geometric estimates]\label{lem:geometric-tools}
For the geometry in \eqref{eq:geometry}, there are a constant $C$, a matrix field $\bW_\eps=(\bW_\eps^1,\bW_\eps^2,\bW_\eps^3)$ and a pressure vector $\bQ_\eps=(Q_\eps^1,Q_\eps^2,Q_\eps^3)$ whose components correspond to the three columns, such that $\bW_\eps\in H^1(\T;\R^{3\times3})$ after zero extension through the holes and
\begin{gather}
 \bW_\eps^i|_{\Oe}\in\Ve,\quad \|\bW_\eps\|_\infty\le C,\quad \|\bW_\eps-\bI\|_2\le C\tau_\eps,\quad \sig\bigl(\|\nabla \bW_\eps\|_{2,\Oe}+\|\bQ_\eps\|_{2,\Oe}\bigr)\le C,\label{eq:cell-L2}\\
 \left|\left\langle
 \sig^2(-\Delta \bW_\eps^i+\nabla Q_\eps^i)-\bM_0\be_i,
 \bzeta\right\rangle\right|
 \le C\cellsize\|\bzeta\|_{H^1(\Oe)}
 \quad(\bzeta\in\Xe).\label{eq:cell-residual}
\end{gather}
Moreover, for $\bv\in\Xe$,
\begin{equation}\label{eq:general-poincare}
 \|\bv\|_2\le C\sig\|\nabla \bv\|_2.
\end{equation}
There is a linear right inverse of divergence $\mathcal B_\eps:L^2_\diamond(\Oe)\longrightarrow\Xe$ satisfying
\begin{equation}\label{eq:bogovskii}
 \Div\mathcal B_\eps b=b,\quad
 \|\nabla\mathcal B_\eps b\|_2
 \le C(1+\sig^{-1})\|b\|_2.
\end{equation}
For fixed $\alpha\in(1,3)$ the constants are independent of sufficiently small $\eps$ and of $N_\eps$.
\end{lemma}
\begin{proof}
We use the exterior-cell construction of \cite[Proposition~2.7 and Appendix~A]{HLO}, with exponent $2$ and with the resistance normalization \eqref{eq:resistance}. In that proposition, the $W^{-1,2}$ residual is stated at the Darcy scale. \cref{app:periodic-geometry} repeats the rescaling for the independent cell and hole lengths $\cellsize$ and $a_\eps$ used here. The inverse-divergence estimate is the $q=2$ case of \cite[Theorem~2.3]{DFL}. The corresponding perforated-domain formulation is also recorded in \cite[Proposition~2.5]{HLO}. The periodic construction in \cref{app:periodic-geometry} has norm $C(1+\sig^{-1})$.

The pressures $Q_\eps^i$ are assembled before normalization. Subtracting one global constant from each component gives $\int_{\Oe}Q_\eps^i\dx=0$. This leaves the residual unchanged, because $\int_{\Oe}\Div\bzeta\dx=0$ for $\bzeta\in\Xe$ and preserves the $L^2$ bound up to a fixed factor. Renormalizing the already assembled cell or annular pieces separately would alter the pressure jumps and hence the distributional gradient.
\end{proof}

\begin{remark}[Constant viscosity in the cell problem]
For a constant viscosity $m_k>0$, the normalized exterior solution is $(\bV^i,m_kQ^i)$, since
\[
-m_k\Delta \bV^i+\nabla(m_kQ^i)=0,\quad \bV^i\to \be_i.
\]
The velocity profile and its far-field normalization are unchanged. Pressure and traction are multiplied by $m_k$ and the resistance matrix is therefore $m_k\bM_0$.
\end{remark}

To obtain an admissible comparison, the cell fields are corrected by a divergence right inverse. For $\bv\in W^{1,\infty}(\T;\R^3)\cap\V$, define the divergence defect and the periodic comparison by
\begin{equation*}
 \begin{split}
 \mathfrak d_\eps(\bv)&=\Div(\bW_\eps\bv)
   =\sum_{i,j=1}^3\bigl((\bW_\eps)_{ij}-\delta_{ij}\bigr)\partial_iv_j,\\
 \bb_\eps(\bv)&=\mathcal B_\eps[\mathfrak d_\eps(\bv)],\quad
 \Cgeo \bv=\bW_\eps\bv-\bb_\eps(\bv).
 \end{split}
\end{equation*}
The defect has zero mean on $\Oe$ because $\bW_\eps\bv$ has zero hole trace and is periodic. Consequently $\Cgeo \bv\in\Ve$. By \eqref{eq:cell-L2}, \eqref{eq:bogovskii} and \eqref{eq:general-poincare},
\begin{equation}\label{eq:comparison-L2}
\begin{aligned}
 \|\mathfrak d_\eps(\bv)\|_2
   &\le C\tau_\eps\|\nabla \bv\|_\infty,\\
 \|\bb_\eps(\bv)\|_2+\sig\|\nabla \bb_\eps(\bv)\|_2
   &\le C\tau_\eps\|\nabla \bv\|_\infty,\\
 \|\widetilde{\Cgeo \bv}-\bv\|_2
   &\le C\tau_\eps\|\bv\|_{W^{1,\infty}},\\
 \|\Cgeo \bv\|_2+\sig\|\nabla\Cgeo \bv\|_2
   &\le C\|\bv\|_{W^{1,\infty}}.
\end{aligned}
\end{equation}
Thus the small divergence defect compensates for the factor $\sig^{-1}$ in the inverse-divergence estimate at the weighted energy level.

\subsection{Estimates for the frozen lifting}
A spatially varying coefficient produces both stress and cell-pressure commutators. Since $\Cgeo\bv$ already has the correct divergence and hole trace, it remains to compare its weighted stress with the Darcy resistance and to quantify these two commutator terms.

\begin{lemma}[The variable-coefficient residual]\label{lem:weighted-residual}
Let $\bv\in W^{1,\infty}(\T;\R^3)\cap\V$ and $m\in W^{1,\infty}(\T)$ satisfy $m_*\le m\le m^*$ for fixed constants $0<m_*\le m^*<\infty$. Then, for $\bphi\in\Ve$,
\begin{equation*}
 \left|2\sig^2\int_{\Oe}m\,\sym(\Cgeo \bv):\sym\bphi\dx
      -\int_{\Oe}m\bM_0\bv\cdot\bphi\dx\right|
 \le C\deltae\,\sig\|\sym\bphi\|_2.
\end{equation*}
The constant depends only on the geometric data, $m_*,m^*$, $\|m\|_{W^{1,\infty}}$ and $\|\bv\|_{W^{1,\infty}}$. It is independent of $\eps$.
\end{lemma}
\begin{proof}
With $\bU=\bW_\eps\bv$ and $Z=\sig\|\sym\bphi\|_2$, \eqref{eq:poincare} and $2\|\sym\bphi\|_2^2=\|\nabla\bphi\|_2^2$ give $\|\bphi\|_2\le C Z$ and $\|\nabla\bphi\|_2\le C\sig^{-1}Z$. The contribution of the divergence correction satisfies
\[
2\sig^2\left|\int_{\Oe}m\,\sym \bb_\eps(\bv):\sym\bphi\dx\right|
 \le C\tau_\eps Z.
\]
The remaining contribution contains $\bU$. This field has zero trace on the holes, although it need not be solenoidal. Using $\Div\bphi=0$ and integrating by parts, we obtain
\begin{equation}\label{eq:weighted-transpose}
\begin{aligned}
 2\sig^2\int_{\Oe}m\,\sym \bU:\sym\bphi\dx
 &=\sig^2\int_{\Oe}m\nabla \bU:\nabla\bphi\dx+T_m,\\
 T_m&=-\sig^2\sum_{i,j}\int_{\Oe}
                  U_j(\partial_i m)(\partial_j\phi_i)\dx.
\end{aligned}
\end{equation}
The zero trace of $\bU$ removes the boundary term. Density then extends the identity to every $\bphi\in\Ve$. Since $\|\bU\|_2\le \|\bW_\eps\|_\infty\|\bv\|_\infty|\Oe|^{1/2}\le C$ and $\|\nabla\bphi\|_2\le C\sig^{-1}Z$,
\[
|T_m|\le \sig^2\|\nabla m\|_\infty\|\bU\|_2\|\nabla\bphi\|_2\le C\sig Z.
\]

For each column we introduce $\omega_i=mv_i$. The product rule then gives
\[
 \nabla \bU=\sum_i v_i\nabla \bW_\eps^i+\bW_\eps\nabla \bv.
\]
Since $\|\bW_\eps\|_\infty\le C$, the product term satisfies
\[
 \sig^2\left|\int_{\Oe}m\,\bW_\eps\nabla\bv:\nabla\bphi\dx\right|
 \le C\sig^2\|\bW_\eps\|_\infty\|\nabla\bv\|_\infty\|\nabla\bphi\|_2
 \le C\sig Z.
\]
For the other terms set $\mathcal R_\eps^i=\sig^2(-\Delta \bW_\eps^i+\nabla Q_\eps^i)-\bM_0\be_i$. Since $\omega_i\in W^{1,\infty}(\T)$ and $\bphi\in\Ve$, the product $\omega_i\bphi$ is periodic, belongs to $H^1(\Oe;\R^3)$ and has zero trace on every hole boundary. Hence $\omega_i\bphi\in\Xe$, with $\|\omega_i\bphi\|_{H^1}\le C\|\omega_i\|_{W^{1,\infty}}\|\bphi\|_{H^1}$. The identity $\Div(\omega_i\bphi)=\nabla\omega_i\cdot\bphi$ requires pairing the residual in the full zero-hole-trace dual space $\Xe^*$. Pairing the residual with this product gives
\begin{align*}
 &\sig^2\int_{\Oe}\omega_i\nabla \bW_\eps^i:\nabla\bphi\dx
       -\int_{\Oe}\omega_i\bM_0\be_i\cdot\bphi\dx\\
 &=\langle\mathcal R_\eps^i,\omega_i\bphi\rangle
   -\sig^2\int_{\Oe}\nabla \bW_\eps^i:
                         (\bphi\otimes\nabla\omega_i)\dx
   +\sig^2\int_{\Oe}Q_\eps^i\nabla\omega_i\cdot\bphi\dx.
\end{align*}
The last term is the pressure commutator generated by $\nabla\omega_i$. The bounds $\|\bphi\|_2\le CZ$, $\|\nabla\bphi\|_2\le C\sig^{-1}Z$, $\|\nabla\bW_\eps^i\|_2+\|Q_\eps^i\|_2\le C\sig^{-1}$ and $\|\omega_i\|_{W^{1,\infty}}\le C$ give
\begin{align*}
 |\langle\mathcal R_\eps^i,\omega_i\bphi\rangle|
 &\le C\cellsize\|\omega_i\bphi\|_{H^1}
 \le C\cellsize(\|\bphi\|_2+\|\nabla\bphi\|_2)
 \le C\frac{\cellsize}{\sig}Z,\\
 \sig^2\left|\int_{\Oe}\nabla \bW_\eps^i:(\bphi\otimes\nabla\omega_i)\dx\right|
 &\le C\sig^2\|\nabla\bW_\eps^i\|_2\|\nabla\omega_i\|_\infty\|\bphi\|_2
 \le C\sig Z,\\
 \sig^2\left|\int_{\Oe}Q_\eps^i\nabla\omega_i\cdot\bphi\dx\right|
 &\le C\sig^2\|Q_\eps^i\|_2\|\nabla\omega_i\|_\infty\|\bphi\|_2
 \le C\sig Z.
\end{align*}
We sum over $i$ and add the transpose term in \eqref{eq:weighted-transpose} and the divergence correction. Since $\tau_\eps=(\cellsize/\sig)^2\le\cellsize/\sig$ for small $\eps$, this proves the lemma.
\end{proof}

The residual estimate now yields the frozen Stokes approximation.

\begin{proposition}[Quantitative frozen Stokes approximation]\label{prop:stokes-approx}
Let $\bv\in W^{1,\infty}(\T;\R^3)\cap\V$ and $m\in W^{1,\infty}(\T)$, with fixed constants $0<m_*\le m^*<\infty$ satisfying $m_*\le m\le m^*$. Let $\bw_\eps[m,\bv]\in\Ve$ be the solution of \eqref{eq:lift} with $\bu=\bv$ and the prescribed coefficient $m$. Then
\begin{equation}\label{eq:stokes-rate}
 \|\widetilde{\bw}_\eps[m,\bv]-\bv\|_2\le C\deltae,
\end{equation}
where $C$ depends on the geometric data, $m_*,m^*$, $\|m\|_{W^{1,\infty}}$ and $\|\bv\|_{W^{1,\infty}}$, but not on $\eps$. The estimate is uniform when these quantities have common bounds. For a reference solution satisfying the hypotheses of \cref{thm:stability} on $[0,T]$,
\begin{equation}\label{eq:eta-rate}
 \eta_\eps^2\le C_T(\tau_\eps+\sig^2)
              \le C_T\bigl(\eps^{\alpha-1}+\eps^{3-\alpha}\bigr).
\end{equation}
\end{proposition}
The proof first compares the exact lifting with $\Cgeo\bv$. The time-dependent estimate then follows from the same bound with $m=\mu(\vartheta)$ and $\bv=\bu$.

\begin{proof}[Proof of \cref{prop:stokes-approx}]
By the Lax--Milgram theorem, $\bw=\bw_\eps[m,\bv]$ is well defined. We subtract the comparison relation from its variational equation and take $\bphi=\bw-\Cgeo \bv\in\Ve$. Then \cref{lem:weighted-residual} gives
\[
2m_*\sig^2\|\sym(\bw-\Cgeo \bv)\|_2^2
 \le C\deltae\,\sig\|\sym(\bw-\Cgeo \bv)\|_2.
\]
Coercivity and \eqref{eq:poincare} give $\sig\|\nabla(\bw-\Cgeo \bv)\|_2+\|\bw-\Cgeo \bv\|_2\le C\deltae$. Together with \eqref{eq:comparison-L2}, this proves \eqref{eq:stokes-rate}. For $m=\mu(\vartheta)$ and $\bv=\bu$, the assumed reference regularity and $H^s(\T)\hookrightarrow W^{1,\infty}(\T)$ for $s\ge3$ give uniform $W^{1,\infty}$ bounds on $[0,T]$. Therefore
\[
 \eta_\eps^2\le C_T\deltae^2
 \le C_T\left(\frac{\cellsize^2}{\sig^2}+\sig^2\right)
=C_T(\tau_\eps+\sig^2)
 \le C_T\bigl(\eps^{\alpha-1}+\eps^{3-\alpha}\bigr).
\]
This proves the stated estimate for the time-dependent macroscopic coefficient.
\end{proof}

We next record the energy and time-derivative bounds needed in the relative-energy argument. Testing \eqref{eq:lift} with $\bw_\eps$ gives
\[
2\mu_*\sig^2\|\sym \bw_\eps\|_2^2
 \le\mu^*\Lambda\|\bu\|_2\|\bw_\eps\|_2.
\]
Together with \eqref{eq:poincare}, this implies
\begin{equation}\label{eq:lift-bound}
 \sig\|\sym \bw_\eps(t)\|_2
 \le\frac{C_P\mu^*\Lambda}{2\mu_*}\|\bu(t)\|_2
 \le W_0.
\end{equation}
For a zero-hole-trace field $z$, the fixed-torus Sobolev inequality and \eqref{eq:general-poincare}, with $\sig\le1$, give
\[
 \|\widetilde z\|_6
 \le C(\|z\|_2+\|\nabla z\|_2)
 \le C\|\nabla z\|_2.
\]
Zero extension and interpolation give
\[
 \|\bw_\eps\|_2+\sig\|\nabla \bw_\eps\|_2
+\sig^{1/2}\|\bw_\eps\|_3\le C\|\bu\|_2.
\]
These estimates are uniform in $\eps$. Differentiating the lifting with respect to time gives
\begin{equation}\label{eq:lift-time}
\begin{aligned}
 2\sig^2\int_{\Oe}m\,\sym\partial_t\bw_\eps:\sym\bphi\dx
 &=\int_{\Oe}(m\bM_0\partial_t\bu+\partial_tm\,\bM_0\bu)\cdot\bphi\dx\\
 &\quad-2\sig^2\int_{\Oe}\partial_tm\,
                         \sym \bw_\eps:\sym\bphi\dx.
\end{aligned}
\end{equation}
Here $\partial_tm=\mu'(\vartheta)\partial_t\vartheta$. The assumed macroscopic regularity gives $\partial_t\vartheta\in C_tL^\infty_x$. Testing \eqref{eq:lift-time} with $\partial_t\bw_\eps$ yields
\begin{equation}\label{eq:lift-time-bound}
 \|\partial_t\bw_\eps\|_2+\sig\|\nabla\partial_t\bw_\eps\|_2
 \le C\bigl(\|\partial_t\bu\|_2
       +\|\partial_tm\|_\infty\|\bu\|_2\bigr).
\end{equation}
The differentiation is justified by difference quotients, using $m\ge\mu_*$ and $m\in C^1_tL^\infty_x$. The bounds are uniform in time on each fixed strong-solution interval.

For fixed $\eps$, define $\mathcal A_\eps(t):\Ve\longrightarrow\Ve^*$ by
\[
 \langle\mathcal A_\eps(t)\bv,\bphi\rangle
=2\sig^2\int_{\Oe}m(t)\sym \bv:\sym\bphi\dx.
\]
The positive lower bound for $m$ and \eqref{eq:poincare} make this an isomorphism on the fixed domain. Moreover, $\mathcal A_\eps\in C^1([0,T];\mathcal L(\Ve,\Ve^*))$ because $m\in C^1_tL^\infty_x$ and the force functional $\bphi\mapsto\int_{\Oe}m\bM_0\bu\cdot\bphi\dx$ belongs to $C^1_t\Ve^*$. The inverse identity
\[
 \mathcal A_\eps(t)^{-1}-\mathcal A_\eps(s)^{-1}
=\mathcal A_\eps(t)^{-1} [\mathcal A_\eps(s)-\mathcal A_\eps(t)]\mathcal A_\eps(s)^{-1}
\]
therefore implies
\begin{equation}\label{eq:lift-test-regularity}
 \bw_\eps\in C^1([0,T];\Ve)
 \subset C([0,T];\Ve)\cap W^{1,2}(0,T;\Hvel).
\end{equation}
This operator argument is performed at fixed $\eps$, whereas \eqref{eq:lift-bound}--\eqref{eq:lift-time-bound} are uniform weighted estimates.

The comparison is initialized with $\bw_\eps(0)$, whereas the stated error uses $\bu(0)$. The uniform lifting bound compares these two initial discrepancies directly.

\begin{lemma}[Equivalent initial discrepancies]\label{lem:initial-equivalence}
There is a constant independent of $\eps$ such that
\begin{equation*}
 \E_\eps^0\le C\Dinit+C\sig^4,
 \quad \Dinit\le C\E_\eps^0+C\sig^4.
\end{equation*}
In particular either discrepancy tends to zero if and only if the other does.
\end{lemma}
\begin{proof}
The lifting energy bound gives $\|\bw_\eps(0)\|_2\le C$ and the macroscopic velocity has a fixed $L^2$ bound. Thus $\|\bw_\eps(0)-\bu(0)\|_{L^2(\Oe)}\le C$. Apply $|a+b|^2\le2|a|^2+2|b|^2$ and $\rho_\eps^0\le\rho^*$ to $\bu_\eps^0-\bw_\eps(0)=(\bu_\eps^0-\bu(0))+(\bu(0)-\bw_\eps(0))$ and then reverse the roles of the two comparison velocities. The density terms coincide. The uniform bound for $\bw_\eps(0)$ already gives the stated equivalence. With the specified geometry, \cref{prop:stokes-approx} sharpens the additive term in both inequalities to $C\sig^4\deltae^2$. The coarser bound suffices for the stability estimate.
\end{proof}

It remains to control the viscosity residual created by replacing $\mu(\vartheta_\eps)$ with the frozen coefficient. For either geometry, set $\be_\eps=\bu_\eps-\bw_\eps$ on $\Oe$ and $\br_\eps=\widetilde{\bw}_\eps-\bu$ on $\mathcal O$. In whole-domain integrals we identify $\be_\eps$ with its zero extension, so that $\bv_\eps=\widetilde \be_\eps+\br_\eps=\be_\eps+\br_\eps$. The tilde is retained when the extension itself is being emphasized. Define
\begin{equation}\label{eq:viscosity-residual}
X_\eps=\sig\|\sym \be_\eps\|_2,\qquad R_{\mu,\eps}=-2\sig^2\int_{\Oe}
  \bigl(\mu(\vartheta_\eps)-\mu(\vartheta)\bigr)
          \sym \bw_\eps:\sym \be_\eps\dx.
\end{equation}
By \eqref{eq:poincare}, \eqref{eq:temperature-error-infinity} and \eqref{eq:lift-bound},
\begin{align*}
 |R_{\mu,\eps}|
 &\le2L_\mu W_0 X_\eps\|z_\eps\|_\infty
 \le\mu_*\chi_0X_\eps^2
  +2L_\mu W_0 X_\eps
   \bigl(C_g\|\br_\eps\|_2
     +C_q|\kappa_s-\kappa_f|h_\eps^{1/q}
                         \|\nabla\vartheta\|_\infty\bigr).
\end{align*}
Under $\chi_0\le1/2$, Young's inequality gives
\begin{equation}\label{eq:viscosity-absorbed}
 |R_{\mu,\eps}|
 \le\mu_*X_\eps^2+
 C_{T,q}\bigl(\|\br_\eps\|_2^2
       +|\kappa_s-\kappa_f|^2h_\eps^{2/q}\bigr).
\end{equation}
Indeed, $\chi_0\le1/2$ bounds the first term in \eqref{eq:viscosity-absorbed} by $(\mu_*/2)X_\eps^2$. Young's inequality is chosen so that the remaining product consumes at most another $(\mu_*/2)X_\eps^2$. Hence $|R_{\mu,\eps}|$ consumes at most $\mu_*X_\eps^2$ of the viscous lower bound $2\mu_*X_\eps^2$, leaving at least $\mu_*X_\eps^2$ for the other remainders.

\section{Periodic quantitative homogenization}\label{sec:periodic-quantitative}
With the lifting estimates in hand, we turn to the relative-energy comparison. Following the relative-energy framework of \cite{BOP,CBSV2025,HNO2026}, we keep the density and velocity errors coupled and retain the thermal work and viscosity residual until the cancellation and absorption steps.

\subsection{Relative-energy framework}
The lifting belongs to the energy space, whereas the defining weak tests are smooth. The next density result allows both the comparison field and its scalar square to be used in the momentum and transport identities.

\begin{lemma}[Energy-space density of comparison tests]\label{lem:comparison-density}
Fix $\eps$ and put $D=\Oe$, $V=\Ve$ and $H=\Hvel$. If
\[
 \bU\in C([0,T];V)\cap W^{1,2}(0,T;H),
\]
then $\eta\bU$ is admissible in \eqref{eq:weak-momentum} and $\eta|\widetilde \bU|^2/2$ is admissible in the extended renormalized continuity equation for any smooth scalar function $\eta$ on $[0,T)$ that vanishes for $t$ sufficiently close to $T$. The resulting identities hold on $[0,\tau]$ for almost every $\tau\in(0,T)$, with the initial traces specified in \cref{def:weak}.
\end{lemma}
\begin{proof}
On the fixed domain $D$, let $\PP_J$ be the $H$-orthogonal projection onto the first $J$ Stokes eigenfunctions of the form $\int_D\nabla \bv:\nabla\bz\dx$ on $V$. The eigenfunctions are smooth up to the fluid boundary, solenoidal and vanish on the holes. The projections $\PP_J$ are contractions in the spectral norms of $H$ and $V$ and converge strongly in both spaces. Since the range of $\bU\in C_tV$ is compact and $\partial_t\bU\in L^2_tH$, we obtain
\[
 \PP_J \bU\to \bU\quad\hbox{in }C_tV,\quad
 \PP_J\partial_t\bU\to\partial_t\bU\quad\hbox{in }L^2_tH.
\]
Reflection beyond the time endpoints and time mollification of finite spectral sums yield a diagonal sequence $\bU_j$ smooth in time and up to the fluid boundary, with
\[
 \bU_j\to \bU\quad\hbox{in }C_tV,\quad
 \partial_t\bU_j\to\partial_t\bU\quad\hbox{in }L^2_tH.
\]
Multiplication by a fixed cutoff $\eta$ as in the statement preserves these convergences.

The microscopic energy bounds and three-dimensional interpolation imply
\[
 \|\bu_\eps\|_{L^{8/3}_tL^4(D)}
 \le C_D\|\bu_\eps\|_{L^\infty_tH}^{1/4}
            \|\bu_\eps\|_{L^2_tV}^{3/4},\quad
 \rho_\eps \bu_\eps\otimes \bu_\eps\in L^{4/3}_tL^2(D).
\]
The convective and viscous terms are therefore continuous under convergence of the test field in $L^4_tV$ and $L^2_tV$, respectively. The time term is continuous when the time derivatives converge in $L^1_tH$ and the force belongs to $L^2_tL^2(D)$. The approximation converges in each of these norms and in $H$ at time zero. Passing to the limit gives the momentum identity with test $\eta\bU$.

For the scalar test, set $\zeta_j=|\widetilde \bU_j|^2/2$ and $\zeta=|\widetilde \bU|^2/2$. Zero extension maps $V$ into $H^1(\T)$ and $H^1(\T)\hookrightarrow L^6(\T)$. The product rule and H\"older's inequality give
\[
 \zeta_j\to\zeta\quad\hbox{in }C_tW^{1,3/2}(\T),\quad
 \partial_t\zeta_j\to\partial_t\zeta
                    \quad\hbox{in }L^2_tL^{3/2}(\T).
\]
Indeed,
\begin{align*}
 \|\nabla(\zeta_j-\zeta)\|_{C_tL^{3/2}}
 &\le C\|\widetilde \bU_j-\widetilde \bU\|_{C_tH^1}
          (\|\widetilde \bU_j\|_{C_tH^1}+\|\widetilde \bU\|_{C_tH^1}),\\
 \|\partial_t(\zeta_j-\zeta)\|_{L^2_tL^{3/2}}
 &\le \|\widetilde \bU_j-\widetilde \bU\|_{C_tL^6}
                           \|\partial_t\widetilde \bU_j\|_{L^2_tL^2}+\|\widetilde \bU\|_{C_tL^6}
                 \|\partial_t\widetilde \bU_j-\partial_t\widetilde \bU\|_{L^2_tL^2}.
\end{align*}
Since the density is bounded and $\widehat\rho_\eps\widetilde \bu_\eps\in L^2_tL^3(\T)$, these convergences suffice to pass to the limit in the continuity equation. After multiplication by $\eta$, periodic spatial mollification and time mollification provide smooth scalar tests in the same norms. The argument also applies to $b(\widehat\rho_\eps)$, because $b$ is bounded on the density range. In particular,
\begin{equation}\label{eq:quadratic-comparison-test}
 \left[\int_D\frac12\rho_\eps|\bU|^2\dx\right]_0^\tau
 =\int_0^\tau\!\!\int_D\rho_\eps
   \left(\bU\cdot\partial_t\bU+
                 \bu_\eps\cdot\nabla\frac{|\bU|^2}{2}\right)\dx\dt.
\end{equation}
Smooth time cutoffs converging to the indicator of $[0,\tau]$ give the endpoint terms. At time zero, these follow from the weak momentum trace and $\bU\in C_tH$ and from density continuity and $\zeta\in C_tL^3$. At the terminal time, Lebesgue differentiation applies for almost every $\tau\in(0,T)$.
\end{proof}

By \cref{lem:comparison-density}, the regularity \eqref{eq:lift-test-regularity} makes time cutoffs of $\bw_\eps$ and $|\widetilde{\bw}_\eps|^2/2$ admissible comparison tests.

The preceding approximation exploits the smooth spectral basis available at fixed periodic geometry. For the bounded-domain argument, finite boundary regularity calls for the following geometric version of the same testing principle.

\begin{lemma}[Comparison tests on a finitely regular domain]
\label{lem:finite-boundary-density}
Let $D\subset\R^3$ be a fixed bounded $C^2$ domain, possibly with several boundary components. Put $V_D=H^1_0(D;\R^3)\cap\{\Div\bv=0\}$ and $H_D=\overline{V_D}^{\,L^2(D)}$. For $\bU\in C([0,T];V_D)\cap W^{1,2}(0,T;H_D)$ there are space-time smooth solenoidal fields $\bU_j$, compactly supported in $D$ in space, such that
\[
 \bU_j\to\bU\quad\text{in }C_tV_D,
 \qquad \partial_t\bU_j\to\partial_t\bU\quad\text{in }L^2_tH_D.
\]
The momentum and quadratic transport tests in \cref{lem:comparison-density} are valid on $D$.
\end{lemma}
\begin{proof}
We choose a smooth compactly supported ambient vector field that points strictly inward on every component of $\partial D$. Such a field is obtained from finitely many local inward directions and a partition of unity. Its flow $F_\delta$, for small positive $\delta$, satisfies $F_\delta(\overline D)\subset\subset D$ and converges to the identity in $C^2$ as $\delta\to0$. Zero extension of a field in $V_D$, or in its $L^2$ closure $H_D$, is solenoidal in $\R^3$. For this extension define its Piola push-forward by
\[
(\mathscr T_\delta\bv)(F_\delta(\boldsymbol y)) =\frac{\nabla F_\delta(\boldsymbol y)}{\det\nabla F_\delta(\boldsymbol y)}
     \widetilde\bv(\boldsymbol y).
\]
For $\phi\in C_c^\infty(\R^3)$, change of variables gives
\[
 \int_{\R^3}\mathscr T_\delta\bv\cdot\nabla\phi\dx
 =\int_{\R^3}\widetilde\bv\cdot\nabla(\phi\circ F_\delta)\dx=0.
\]
Thus $\mathscr T_\delta\bv$ is solenoidal and supported in $F_\delta(\overline D)$. The operators are uniformly bounded on $L^2$ and $H^1$ and converge strongly to the identity in those norms. These statements follow first for smooth ambient fields and then by density and the uniform bounds on $F_\delta$ and its inverse.

Convolve $\mathscr T_\delta\bv$ with a standard mollifier whose radius is smaller than half the distance from $F_\delta(\overline D)$ to $\partial D$ and tends to zero as $\delta\to0$. The resulting linear operators $\mathscr S_\delta$ map into $C^\infty_c(D;\R^3)\cap\{\Div=0\}$, are uniformly bounded on $H_D$ and $V_D$ and converge strongly to the identity on each space. Compactness of the range of $\bU\in C_tV_D$ gives $\mathscr S_\delta\bU\to\bU$ in $C_tV_D$. Dominated convergence gives $\mathscr S_\delta\partial_t\bU\to\partial_t\bU$ in $L^2_tH_D$. Time mollification after reflection, followed by a diagonal choice, proves the stated approximation. Multiplication by a fixed smooth time cutoff that vanishes for $t$ sufficiently close to $T$ preserves both convergences.

The momentum pairings pass to the limit using the convective and stress bounds in the proof of \cref{lem:comparison-density}. For the scalar squares, choose $L>0$ with $\overline D\subset(-L/2,L/2)^3$ and identify opposite faces of this cube to obtain the auxiliary torus $\mathbb T_L^3=(\R/L\mathbb Z)^3$. Extend $\bU_j$ and $\bU$ by zero outside $D$. Since these fields have zero boundary trace, their extensions belong to $H^1(\mathbb T_L^3)$ and vanish near the identified faces. The product estimates in the proof of \cref{lem:comparison-density} give
\[
\begin{aligned}
 \frac12|\widetilde\bU_j|^2&\to\frac12|\widetilde\bU|^2
     &&\text{in }C_tW^{1,3/2}(\mathbb T_L^3),\quad  \partial_t\frac12|\widetilde\bU_j|^2&\to\partial_t\frac12|\widetilde\bU|^2
     &&\text{in }L^2_tL^{3/2}(\mathbb T_L^3).
\end{aligned}
\]
Periodic mollification produces smooth scalar tests in the same norms. Restrict these tests to the ambient domain of the continuity equation. The time-endpoint argument in the proof of \cref{lem:comparison-density} applies to these convergences. 
\end{proof}

The lifting is therefore an admissible comparison in the weak momentum
and transport identities. We now combine these identities with the
density and heat equations.

For either geometry write $d_\eps=\widehat\rho_\eps-\rho$ for the density error. In the periodic case define
\begin{equation*}
 \E_\eps(t)=\frac{\sig^4}{2}\int_{\Oe}\rho_\eps|\be_\eps|^2\dx
                 +\frac12\|d_\eps\|_2^2.
\end{equation*}
Its initial value is \eqref{eq:initial}. Combining the kinetic comparison with the density and temperature equations gives the following inequality. The viscosity estimate \eqref{eq:viscosity-absorbed} leaves a positive part of the dissipation available to absorb the inertial remainders.

\begin{lemma}[Relative-energy inequality]\label{lem:relative-energy}
Under the assumptions of \cref{thm:stability}, for almost every $\tau\in(0,T)$,
\begin{equation}\label{eq:rei}
\begin{aligned}
 &\E_\eps(\tau)+c\sig^2\int_0^\tau\|\sym \be_\eps\|_2^2\dt
       +c\int_0^\tau\|z_\eps\|_{H^1}^2\dt
 \\
 &\le\E_\eps^0+C\int_0^\tau\E_\eps(t)\dt
       +C_{T,q}\bigl(\eta_\eps^2+
         |\kappa_s-\kappa_f|^2h_\eps^{2/q}+\sig^3\bigr),
\end{aligned}
\end{equation}
where $c>0$ and $C_{T,q}$ are independent of small $\eps$.
\end{lemma}
\begin{proof}
All applications of Young's inequality in this proof use positive parameters chosen independently and sufficiently small so that their total contribution is absorbed by the fixed viscous and thermal coercivity constants. The remaining constants may depend on these choices but stay uniform in $\eps$. Test \eqref{eq:weak-heat} with $\vartheta_\eps(t)$ to obtain
\begin{equation*}
 \int_{\T}\bar\kappa_\eps|\nabla\vartheta_\eps|^2\dx
 =\int_{\Oe}\vartheta_\eps \bg\cdot \bu_\eps\dx
 \quad\text{for almost every }t\in(0,T).
\end{equation*}
Using this identity, rewrite \eqref{eq:energy} as the kinetic energy inequality with force $(\rho_\eps \bfv-\vartheta_\eps \bg)$ and viscous dissipation. Combine this inequality with the momentum equation tested with $\bw_\eps$ and the continuity equation tested with $|\widetilde{\bw}_\eps|^2/2$. Both tests are justified at fixed $\eps$ by \eqref{eq:lift-test-regularity}, \cref{lem:comparison-density} and \eqref{eq:quadratic-comparison-test}. The heat equation for the difference is added after the density estimate. After passing to the limit in the tests, we use the uniform estimates \eqref{eq:poincare}, \eqref{eq:scalar} and \eqref{eq:lift-bound}--\eqref{eq:lift-time-bound}, together with the macroscopic norms. All identities are understood in integrated form.

For the relative kinetic energy, set $\mathcal K_\eps(t)=\sig^4\int_{\Oe}\rho_\eps|\be_\eps|^2\dx/2$ and
\[
 \mathscr I_\eps(t)=\sig^4\int_{\Oe}\rho_\eps \be_\eps\cdot
        \bigl(\partial_t\bw_\eps+(\bu_\eps\cdot\nabla)\bw_\eps\bigr)\dx.
\]
Testing the momentum equation with $\bw_\eps$ gives
\[
\begin{aligned}
 \sig^4\left[\int_{\Oe}\rho_\eps\bu_\eps\cdot\bw_\eps\dx\right]_0^\tau
 &=\sig^4\int_0^\tau\!\!\int_{\Oe}
 \bigl(\rho_\eps\bu_\eps\cdot\partial_t\bw_\eps
       +\rho_\eps\bu_\eps\otimes\bu_\eps:\nabla\bw_\eps\bigr)\dx\dt\\
 &\quad-2\sig^2\int_0^\tau\!\!\int_{\Oe}
       \mu(\vartheta_\eps)\sym\bu_\eps:\sym\bw_\eps\dx\dt
       +\int_0^\tau\!\!\int_{\Oe}
       (\rho_\eps\bfv-\vartheta_\eps\bg)\cdot\bw_\eps\dx\dt.
\end{aligned}
\]
Testing the continuity equation with $|\widetilde{\bw}_\eps|^2/2$ gives
\[
 \frac{\sig^4}{2}
 \left[\int_{\Oe}\rho_\eps|\bw_\eps|^2\dx\right]_0^\tau
 =\sig^4\int_0^\tau\!\!\int_{\Oe}\rho_\eps\bw_\eps\cdot
       \bigl(\partial_t\bw_\eps+(\bu_\eps\cdot\nabla)\bw_\eps\bigr)\dx\dt.
\]
Subtracting the first identity from the kinetic energy inequality and adding the second combines the time and convection terms into $-\mathscr I_\eps$. This gives
\begin{equation}\label{eq:kinetic-prelift}
\begin{aligned}
 &\mathcal K_\eps(\tau)-\mathcal K_\eps(0)
 +2\sig^2\int_0^\tau\!\!\int_{\Oe}
       \mu(\vartheta_\eps)|\sym\be_\eps|^2\dx\dt\\
 &\le \int_0^\tau\!\!\int_{\Oe}
       (\rho_\eps\bfv-\vartheta_\eps\bg)\cdot\be_\eps\dx\dt
       -\int_0^\tau \mathscr I_\eps(t)\dt
       -2\sig^2\int_0^\tau\!\!\int_{\Oe}
       \mu(\vartheta_\eps)\sym\bw_\eps:\sym\be_\eps\dx\dt .
\end{aligned}
\end{equation}
The stress terms in \eqref{eq:kinetic-prelift} follow by expanding $\sym\bu_\eps=\sym\be_\eps+\sym\bw_\eps$. 

Now split
$\mu(\vartheta_\eps)=m+[\mu(\vartheta_\eps)-m]$ in the last term of \eqref{eq:kinetic-prelift}. The second part is precisely $R_{\mu,\eps}$, whereas the lifting equation \eqref{eq:lift}, tested with $\be_\eps$, gives
\[
 2\sig^2\int_{\Oe}m\,\sym\bw_\eps:\sym\be_\eps\dx
 =\int_{\Oe}m\bM_0\bu\cdot\be_\eps\dx.
\]
Substituting $m\bM_0\bu=\rho\bfv-\vartheta\bg-\nabla p$ leaves the force differences
$d_\eps\bfv$ and $-z_\eps\bg$. Finally,
$\int_{\Oe}\nabla p\cdot\be_\eps\dx=0$ because $\be_\eps$ is solenoidal and has zero trace on every hole boundary. Hence, for almost every $\tau\in(0,T)$,
\begin{equation}\label{eq:kinetic-error}
\begin{aligned}
 &\mathcal K_\eps(\tau)-\mathcal K_\eps(0)
 +2\sig^2\int_0^\tau\!\!\int_{\Oe}
                   \mu(\vartheta_\eps)|\sym \be_\eps|^2\dx\dt\\
 &\le\int_0^\tau\left[
       \int_{\Oe}d_\eps \bfv\cdot \be_\eps\dx
       -\int_{\Oe}z_\eps \bg\cdot \be_\eps\dx
       +R_{\mu,\eps}-\mathscr I_\eps\right]\dt.
\end{aligned}
\end{equation}

The density error equation is
\[
 \partial_td_\eps+\widetilde \bu_\eps\cdot\nabla d_\eps
=-(\be_\eps+\br_\eps)\cdot\nabla\rho
 \quad\text{on }\T.
\]
The stated Sobolev regularity suffices for the density comparison.
To obtain the squared-error identity, use renormalization
with $b(r)=r^2$ for $\widehat\rho_\eps$ and test its continuity equation with $\rho$. This test is justified by space-time mollification, since
$\rho\in C_tH^s\cap C^1_tH^{s-1}$, $s\ge3$, and
$\widetilde\bu_\eps\in L^2_tL^2_x$. The effective transport equation also
preserves $\int_\T\rho^2\dx$. Finally,
$\int_\T\rho(\widetilde\bu_\eps-\bu)\cdot\nabla\rho\dx=0$
by solenoidality, with $\rho^2/2$ as an admissible $H^1$ test.
Expanding $\|\widehat\rho_\eps-\rho\|_2^2$ therefore gives the integrated identity
\begin{equation*}
 \frac12\|d_\eps(\tau)\|_2^2-\frac12\|d_\eps(0)\|_2^2
 =-\int_0^\tau\!\!\int_{\T}
 d_\eps(\be_\eps+\br_\eps)\cdot\nabla\rho\dx\dt.
\end{equation*}
The integrand belongs to $L^1_{t,x}$. Thus the squared density error is absolutely continuous. For any $\delta>0$, Young's inequality and \eqref{eq:poincare} give
\begin{equation*}
 \frac12\frac{\dd}{\dt}\|d_\eps\|_2^2
 \le \delta\sig^2\|\sym \be_\eps\|_2^2
       +C_\delta\|d_\eps\|_2^2+C_\delta\|\br_\eps\|_2^2
 \quad\text{for almost every }t\in(0,T).
\end{equation*}
Here \eqref{eq:poincare} controls the velocity error through dissipation.

Taking $z_\eps$ as a test function in \eqref{eq:heat-difference}, we obtain
\begin{align*}
 \int_{\T}\bar\kappa_\eps|\nabla z_\eps|^2\dx
 &=\int_{\T}(\be_\eps+\br_\eps)\cdot \bg\,z_\eps\dx-\int_{\T}(\bar\kappa_\eps-\kappa_f)
                         \nabla\vartheta\cdot\nabla z_\eps\dx.
\end{align*}
The $\be_\eps$ term cancels exactly with the thermal work in \eqref{eq:kinetic-error}. Since $z_\eps$ has mean zero on $\T$, the mean-zero Poincar\'e inequality gives $\|z_\eps\|_{H^1}\le C\|\nabla z_\eps\|_2$. The remaining terms satisfy
\[
 \begin{aligned}
 \left|\int_{\T} \br_\eps\cdot \bg\,z_\eps\dx\right|
 &\le\delta\|\nabla z_\eps\|_2^2+C_\delta\|\br_\eps\|_2^2,\\
 \left|\int_{\T}(\bar\kappa_\eps-\kappa_f)
                   \nabla\vartheta\cdot\nabla z_\eps\dx\right|
 &\le\delta\|\nabla z_\eps\|_2^2
       +C_{T,\delta}|\kappa_s-\kappa_f|^2h_\eps.
 \end{aligned}
\]
The force difference is bounded by $\delta\sig^2\|\sym \be_\eps\|_2^2+C_\delta\|d_\eps\|_2^2$.

We estimate the inertial terms. Put $D_\eps=\|\nabla \be_\eps\|_2$. The perforated Poincar\'e inequality and Sobolev interpolation give
\[
 \|\be_\eps\|_4^2\le \|\be_\eps\|_2^{1/2}\|\be_\eps\|_6^{3/2}
 \le C\sig^{1/2}D_\eps^2.
\]
The lifting estimates imply $\|\bw_\eps\|_2\le C$, $\|\nabla\bw_\eps\|_2\le C\sig^{-1}$ and $\|\partial_t\bw_\eps\|_2\le C$. Consequently $\|\bw_\eps\|_4\le C\|\bw_\eps\|_2^{1/4}\|\bw_\eps\|_6^{3/4}\le C\sig^{-3/4}$. Splitting $\bu_\eps=\be_\eps+\bw_\eps$ and applying H\"older's inequality with exponents $(4,4,2)$ yields
\begin{align*}
& \sig^4\left|\int_{\Oe}\rho_\eps \be_\eps\cdot(\be_\eps\cdot\nabla)\bw_\eps\dx\right|
 \le C\sig^4\|\be_\eps\|_4^2\|\nabla\bw_\eps\|_2
 \le C\sig^{7/2}D_\eps^2,\\
& \sig^4\left|\int_{\Oe}\rho_\eps \be_\eps\cdot(\bw_\eps\cdot\nabla)\bw_\eps\dx\right|
 \le C\sig^4\|\be_\eps\|_4\|\bw_\eps\|_4\|\nabla\bw_\eps\|_2
 \le C\sig^{5/2}D_\eps\le\delta\sig^2D_\eps^2+C_\delta\sig^3,\\
& \sig^4\left|\int_{\Oe}\rho_\eps \be_\eps\cdot\partial_t\bw_\eps\dx\right|
 \le C\sig^4\|\be_\eps\|_2\|\partial_t\bw_\eps\|_2
 \le C\sig^5D_\eps\le\delta\sig^2D_\eps^2+C_\delta\sig^8.
\end{align*}
The argument uses only the stated $L^2$-based bounds for $\bw_\eps$. Since $\be_\eps$ is periodic, solenoidal and has zero trace on every hole boundary, integration by parts on the unit torus after zero extension gives
$D_\eps^2=2\|\sym \be_\eps\|_2^2$. After \eqref{eq:viscosity-absorbed}, at least $\mu_*\sig^2\|\sym\be_\eps\|_2^2$ remains. Choose the Young parameters $\delta$ so that their total viscous contribution uses at most one quarter of this remainder. Since $\sig\to0$ as $\eps\to0$, we can then choose $\eps$ small enough that $2C\sig^{3/2}\le\mu_*/4$, which gives
\[
 C\sig^{7/2}D_\eps^2
 \le\frac{\mu_*}{4}\sig^2\|\sym\be_\eps\|_2^2.
\] Since $0\le h_\eps\le1$ and $q>3$, one has $h_\eps\le h_\eps^{2/q}$. Integration in time gives \eqref{eq:rei}.
\end{proof}

\subsection{Stability and convergence rates}
\cref{lem:relative-energy} now yields the stability estimate and its explicit rates.
\begin{proof}[Proof of \cref{thm:stability}]
Gronwall's lemma applied to \eqref{eq:rei} bounds the essential supremum of the relative energy and both dissipation integrals by $C_{T,q}(\E_\eps^0+\eta_\eps^2+ |\kappa_s-\kappa_f|^2h_\eps^{2/q}+\sig^3)$. Continuity of the density in $L^2$ permits us to replace its essential supremum by the supremum on $[0,T]$. Moreover,
\[
 \|\widetilde \bu_\eps-\bu\|_{L^2_tL^2_x}^2
 \le2C_P^2\sig^2\|\sym \be_\eps\|_{L^2_tL^2_x}^2+2\eta_\eps^2.
\]
The $L^2_tL^\infty_x$ temperature bound follows by squaring \eqref{eq:temperature-error-infinity} and integrating in time. We have therefore proved \eqref{eq:main-estimate} with $\E_\eps^0$. \cref{lem:initial-equivalence} and $\sig^4\le\sig^3$ give the version with $\Dinit$. The estimate remains valid with $\eta_\eps$ kept explicitly on the right-hand side.
\end{proof}

\begin{proof}[Proof of \crefrange{thm:homogenization}{cor:norm-rate}]
The choice $q=4$ in \cref{thm:stability}, together with \eqref{eq:eta-rate}, yields the stated rate. For fixed phase conductivities,
\[
|\kappa_s-\kappa_f|^2h_\eps^{1/2}
 \le C\eps^{3(\alpha-1)/2}\le C\eps^{\alpha-1},
 \quad \sig^3\le\sig^2\le\eps^{3-\alpha},
\]
which implies \eqref{eq:explicit-estimate}. Since $1<\alpha<3$, all remainder terms vanish and $\Dinit\to0$ gives the asserted strong convergences. The estimate $\|\mu(\vartheta_\eps)-\mu(\vartheta)\|_{L^2_tL^\infty_x}\le L_\mu\|\vartheta_\eps-\vartheta\|_{L^2_tL^\infty_x}$ yields \eqref{eq:viscosity-convergence}. Inserting the initial rate and taking square roots gives \cref{cor:norm-rate}. Finally, \eqref{eq:conditional-rate} follows from \eqref{eq:main-estimate} and the exact identity
\[
 h_\eps^{2/q}=|T_0|^{2/q}\tau_\eps^{6/q}
 =|T_0|^{2/q}\left(\frac{\eps}{\cellsize}\right)^{6/q}
   \eps^{6(\alpha-1)/q}
 \le C_q\eps^{6(\alpha-1)/q},
\]
where \eqref{eq:mesh-comparison} was used in the last step. The same uniform comparison shows that all limits and rate statements hold as $\eps\to0$.
\end{proof}

\section{Local well-posedness of the periodic Darcy--heat system}\label{sec:periodic-macro}
We construct the strong periodic reference solution used in \cref{thm:stability}. The instantaneous Darcy--heat problem provides the velocity map for the transport iteration.

\subsection{Instantaneous problem and regularity}
Let $\PP:L^2(\T;\R^3)\longrightarrow\V$ be the orthogonal projection and use $\V$, $\Lop$ and $\Bop$ from the periodic part of \cref{sec:main-results}. Define
\[
 \Jop=\Lop^{-1}\Bop,\quad
 \Kop=\Bop^*\Jop,\quad
 \Bop^*\psi=\PP(\psi\bg).
\]
The adjoint is taken in $L^2$, with scalar domain $L^2_\diamond(\T)$. Since $\bg=\nabla F$, the operator $\Jop:\V\longrightarrow H^2_\diamond(\T)$ is bounded, $\|\Jop\bv\|_\infty\le C_\infty\|\bv\|_2$ and $\Kop$ is self-adjoint and nonnegative. Throughout this section $q_0\le1/2$. For a prescribed density, set $\bh=\rho\bfv$.

\begin{proposition}[The instantaneous problem]\label{prop:nonlinear-frozen}
Let $\bh\in L^2(\T;\R^3)$ satisfy $\norm \bh{L^2}\le H_0$. Then there exists a unique triple $(\bu,\vartheta,p)\in\V\times H^2_\diamond\times H^1_\diamond$ such that
\begin{equation}\label{eq:nonlinear-instantaneous}
\left\{
\begin{aligned}
& \mu(\vartheta)\bM_0\bu+\nabla p=\bh-\vartheta\bg,\\
& \vartheta=\Jop\bu,\quad \Div\bu=0,\quad
 \av p=\av\vartheta=0.
\end{aligned}
\right.
\end{equation}
Moreover, $\norm \bu{L^2}\le H_0/c_0$ and $\norm\vartheta{L^\infty}\le C_\infty H_0/c_0$. For two such forces and their solutions,
\begin{equation}\label{eq:nonlinear-L2-stability}
 \norm{\bu_1-\bu_2}{L^2}
 \le\frac{\norm{\bh_1-\bh_2}{L^2}}{c_0(1-q_0)},\quad
 \norm{\vartheta_1-\vartheta_2}{H^2}
 \le C\norm{\bu_1-\bu_2}{L^2}.
\end{equation}
\end{proposition}
\begin{proof}
For each $\tau\in L^\infty_\diamond(\T)$, let $\bu_\tau\in\V$ be the unique solution of
\[
 \int_{\T}\mu(\tau)\bM_0\bu_\tau\cdot \bv\dx
+\langle \Kop\bu_\tau,\bv\rangle=\langle \bh,\bv\rangle\quad(\bv\in\V).
\]
Lax--Milgram and coercivity give $\norm{\bu_\tau}{L^2}\le H_0/c_0$. For $\tau_1,\tau_2\in L^\infty_\diamond(\T)$, subtract the corresponding equations and test with $\bu_{\tau_1}-\bu_{\tau_2}$ to obtain
\[
 \norm{\bu_{\tau_1}-\bu_{\tau_2}}{L^2}
 \le\frac{\Lambda L_\mu H_0}{c_0^2}
       \norm{\tau_1-\tau_2}{L^\infty}.
\]
Since $\|\Jop\bv\|_\infty\le C_\infty\|\bv\|_2$, the preceding Lipschitz constant is multiplied by $C_\infty$. By definition, this product is exactly $q_0$. Hence $\tau\mapsto\Jop\bu_\tau$ is a contraction on the complete Banach space $L^\infty_\diamond$. At its fixed point, $\vartheta=\Jop\bu\in H^2_\diamond$ and $\PP(\vartheta\bg)=\Bop^*\Jop\bu=\Kop\bu$. The variational equation therefore gives $\PP(\mu(\vartheta)\bM_0\bu+\vartheta\bg-\bh)=0$. The orthogonal complement of $\V$ consists of gradients of mean-zero $H^1$ functions, which recovers the unique pressure in \eqref{eq:nonlinear-instantaneous}. Conversely, each solution of \eqref{eq:nonlinear-instantaneous} determines this fixed point, which proves uniqueness. For two different forcings, by the same subtraction and $\norm{\vartheta_1-\vartheta_2}{L^\infty} \le C_\infty\norm{\bu_1-\bu_2}{L^2}$, we obtain
\[
c_0\norm{\bu_1-\bu_2}{L^2}
 \le\norm{\bh_1-\bh_2}{L^2}
+c_0q_0\norm{\bu_1-\bu_2}{L^2}.
\]
Absorption proves \eqref{eq:nonlinear-L2-stability}.
\end{proof}

The transport equation requires higher Sobolev regularity of the instantaneous velocity. We first estimate the frozen problem and then use the two-derivative gain of the heat equation to treat the nonlinear coefficient.

\begin{lemma}[Frozen Sobolev estimates]\label{lem:frozen-sobolev}
Let $k\ge1$ be an integer. The Leray projection commutes with spatial derivatives. The operator $\Kop$ satisfies, for solenoidal $\bv\in H^{k-1}(\T;\R^3)$,
\begin{equation}\label{eq:K-commutator}
 \|[\partial^\beta,\Kop]\bv\|_2
 \le C_{k,g}\|\bv\|_{H^{k-1}},\quad |\beta|=k.
\end{equation}
Let $\bh\in H^k(\T;\R^3)$ and let $\bu\in\V$ be the variational solution of $\PP(\bA\bu)+\Kop\bu=\PP\bh$. Assume $\bA=m\bM_0$ is uniformly positive and $m\in H^{k+1}(\T)$. Then $\bu\in H^k$ and
\begin{equation}\label{eq:frozen-Hk}
 \|\bu\|_{H^k}\le C_{\bA}\bigl(\|\bh\|_{H^k}+\|\bu\|_2\bigr).
\end{equation}
The constant depends on $k$, the fixed data, $\|m\|_{H^{k+1}}$ and the positive lower bound of $m$. The estimates hold uniformly for smooth approximations with common bounds in these norms.
\end{lemma}
\begin{proof}
The Leray projection is a Fourier multiplier on the torus and includes constant vector fields. The mean-zero inverse of $-\kappa_f\Delta$ also commutes with derivatives. Expanding $\Kop \bv=\PP\bigl(\bg\Lop^{-1}(\bv\cdot \bg)\bigr)$ by Leibniz' rule shows that each term in $[\partial^\beta,\Kop]\bv$ has at least one derivative on $\bg$ and at most $k-1$ derivatives on $\bv$. Smoothness of $\bg$ and the two-derivative gain of $\Lop^{-1}$ give \eqref{eq:K-commutator}. The scalar sources in this expansion have zero mean, since spatial derivatives of $\bg$ are gradients of periodic potentials and spatial derivatives of a solenoidal field remain solenoidal.

We mollify $m$ and $\bh$ periodically. At fixed mollification scale, solve the coercive variational equation on the space of solenoidal Fourier polynomials with frequencies $|k|\le M$, including constants. Its orthogonal projection commutes with spatial derivatives. Thus $(-1)^{|\beta|}\partial^{2\beta}\bu_M$ is an admissible finite-dimensional test. Integration by parts gives the differentiated energy identity, with no prior regularity assumption on the infinite-dimensional solution. The ensuing estimates are independent of $M$. Coercivity controls the highest derivative. At first order, the coefficient term is bounded by $C\|\nabla\bA\|_6\|\bu\|_3$. The interpolation inequality $\|\bu\|_3\le C\|\bu\|_2^{1/2}(\|\nabla \bu\|_2+\|\bu\|_2)^{1/2}$ and Young's inequality absorb the unknown gradient. At higher levels, the integer-order Sobolev product estimates give
\begin{equation*}
 \|[\partial^\beta,\bA]\bu\|_2
 \le C_k\|\bA\|_{H^{k+1}}\|\bu\|_{H^{k-1}},\quad k\ge2.
\end{equation*}
For $k=2$, Leibniz' rule and Sobolev embedding give, for $|\beta|=2$,
\[
 \|[\partial^\beta,\bA]\bu\|_2
 \le C\bigl(\|\nabla\bA\|_\infty\|\nabla\bu\|_2
       +\|\nabla^2\bA\|_6\|\bu\|_3\bigr)
 \le C\|\bA\|_{H^3}\|\bu\|_{H^1}.
\] For $k\ge3$, the highest coefficient derivative contributes $\|\partial^\beta\bA\|_2\|\bu\|_\infty\le C\|\bA\|_{H^{k+1}}\|\bu\|_{H^{k-1}}$, since $k-1\ge2$. If $1\le|\gamma|\le k-1$, then $\partial^\gamma\bA\in H^{k+1-|\gamma|}\hookrightarrow L^\infty$ and $\|\partial^{\beta-\gamma}\bu\|_2\le\|\bu\|_{H^{k-1}}$. These bounds cover every term in the commutator. Induction using \eqref{eq:K-commutator} proves \eqref{eq:frozen-Hk}. At fixed mollification scale, we let $M\to\infty$. Coercivity and density of the Fourier spaces identify the limit with the variational solution. Weak compactness retains the displayed Sobolev bounds. Periodic mollification preserves the positive lower bound and the relevant Sobolev bounds for $m$. Letting its scale tend to zero, using coefficient convergence in $L^\infty$ and uniqueness in $L^2$, proves the asserted regularity. The multiplication and composition estimates are the periodic integer-order forms of \cite[Corollary~2.86 and Theorem~2.87]{BCD}, with the constant part of a composed function separated.
\end{proof}

The two-derivative gain of the heat operator closes the frozen estimates for the nonlinear coefficient $\mu(\vartheta)$. The same gain controls differences of solutions and gives the regularity of the forcing-to-velocity map.

\begin{proposition}[Sobolev dependence on the forcing]\label{prop:nonlinear-sobolev}
Let $s\ge3$ be an integer. For forcings satisfying $\norm \bh{H^s}\le R$ and $\norm \bh{L^2}\le H_0$, the solutions of \cref{prop:nonlinear-frozen} satisfy
\[
 \norm \bu{H^s}+\norm\vartheta{H^{s+2}}+\norm p{H^{s+1}}\le C_R.
\]
For two forcings in this set and an integer $0\le m\le s$,
\begin{equation}\label{eq:nonlinear-Hm-stability}
 \norm{\bu_1-\bu_2}{H^m}
 +\norm{\vartheta_1-\vartheta_2}{H^{m+2}}
 \le C_R\norm{\bh_1-\bh_2}{H^m}.
\end{equation}
The velocity solution map is locally $C^1$ in $H^{s-1}$ near each such forcing. The constants may depend on finitely many derivatives of $\mu$ on $[-C_\infty H_0/c_0,C_\infty H_0/c_0]$.
\end{proposition}

\begin{proof}
The Sobolev bootstrap starts from $\bA=\mu(\vartheta)\bM_0$. Initially $\vartheta=\Jop\bu\in H^2$, so $\bA\in H^2$. The first-order estimate of \cref{lem:frozen-sobolev} gives $\bu\in H^1$ and then $\vartheta\in H^3$. Inductively, $\bu\in H^{m-1}$ implies $\bA\in H^{m+1}$ and \eqref{eq:frozen-Hk} gives $\bu\in H^m$ for $2\le m\le s$. The coefficient composition bounds depend only on the attained temperature interval and the preceding Sobolev level. The same lemma supplies the uniform smooth-approximation estimates, so this bootstrap applies to the original variational solution. The heat equation gives $\vartheta\in H^{s+2}$ and the pressure gradient identity, with zero mean, gives $p\in H^{s+1}$.

For the difference $\bv=\bu_1-\bu_2$, set $\bA_i=\mu(\vartheta_i)\bM_0$ for $i=1,2$ and put $b_{12}=\int_0^1\mu'(\vartheta_2+t(\vartheta_1-\vartheta_2))\dd t$. Then
\[
 \PP(\bA_1\bv)+\Kop\bv
=\PP(\bh_1-\bh_2)-\PP\bigl(b_{12}\bM_0\bu_2\,\Jop\bv\bigr).
\]
The $L^2$ stability estimate \eqref{eq:nonlinear-L2-stability} gives $\|\bv\|_2\le C_R\|\bh_1-\bh_2\|_2$. The coefficient $b_{12}\bM_0\bu_2$ is bounded in $H^s$. Since $\Jop$ gains two derivatives, the last term is bounded in $H^2$ by $C_R\|\bv\|_2$. Applying \cref{lem:frozen-sobolev} first with $k=1$ and then with $k=2$ gives $\|\bv\|_{H^k}\le C_R\|\bh_1-\bh_2\|_{H^k}$ at these two orders. For higher orders we use $\norm{b_{12}\bM_0\bu_2\Jop\bv}{H^m}\le C_R\norm \bv{H^{m-2}}$, $2\le m\le s$, together with the frozen estimates, to obtain \eqref{eq:nonlinear-Hm-stability}.

For the differentiability statement, set $m=s-1$ and equip $\mathcal X^m=H^m(\T;\R^3)\cap\V$ and $\mathcal Y^m=H^m(\T;\R^3)$ with their $H^m$ norms. Consider the $C^1$ map
\[
 \Phi:\mathcal X^m\times\mathcal Y^m\longrightarrow\mathcal X^m,\quad
 \Phi(\bu,\bh)=\PP\!\left(\mu(\Jop\bu)\bM_0\bu\right)+\Kop\bu-\PP\bh.
\]
At a solution, the Fr\'echet derivative of $\Phi$ with respect to $\bu$ is
\[
D_{\bu}\Phi\,\bv=\PP\bigl(\mu(\vartheta)\bM_0\bv +\mu'(\vartheta)(\Jop\bv)\bM_0\bu\bigr)+\Kop\bv.
\]
Its quadratic form is bounded below by $c_0(1-q_0)\|\bv\|_2^2$. Set $\mathscr A=D_{\bu}\Phi$, $\bA=\mu(\vartheta)\bM_0$ and $\bb=\mu'(\vartheta)\bM_0\bu\in H^s$. Nonsymmetric Lax--Milgram gives a unique $L^2$ solution of $\mathscr A\bv=\br$ and
\[
 \|\bv\|_2\le[c_0(1-q_0)]^{-1}\|\br\|_2.
\]
For $\br\in\mathcal X^m$, write its equation as $\PP(\bA\bv)+\Kop\bv=\br-\PP(\bb\,\Jop\bv)$. The two-derivative gain of $\Jop$ and Sobolev multiplication give
\[
 \|\bb\,\Jop\bv\|_{H^k}\le C_R
 \begin{cases}
    \|\bv\|_2,&k=1,2,\\
    \|\bv\|_{H^{k-2}},&3\le k\le s-1.
 \end{cases}
\]
Successive application of \cref{lem:frozen-sobolev} for $k=1,\ldots,m$ yields $\|\bv\|_{H^m}\le C_R\|\br\|_{H^m}$. Thus $\mathscr A$ is an isomorphism on $\mathcal X^m$. Multiplication, smooth composition and the bounded map $\Jop:H^m\cap\V\longrightarrow H^{m+2}$ make $\Phi$ continuously Fr\'echet differentiable between the stated spaces. The implicit function theorem gives the local velocity map $\Sop:\mathcal Y^m\longrightarrow\mathcal X^m$ with
\[
 D\Sop(\bh)\br=\mathscr A^{-1}\PP\br,\qquad \|D\Sop(\bh)\br\|_{H^m}\le C_R\|\br\|_{H^m}.
\]
Here the Fourier multiplier $\PP$ is bounded on $H^m$, so the derivative acts on arbitrary $\br\in\mathcal Y^m$. The margin $q_0\le1/2$ leaves a neighborhood on which the linearized coercivity stays positive.
\end{proof}

\subsection{Evolution problem and local well-posedness}\label{sec:nonconstant-state}
We use the instantaneous velocity map to iterate the transport equation.
\begin{proof}[Proof of \cref{thm:local}]
We denote by $\Sop(r)$ the instantaneous velocity corresponding to $\bh=r\bfv$ and initialize the iteration with $\rho^0(t)=\rho_0$. The successive densities solve
\begin{equation*}
 \partial_t\rho^{n+1}+\Sop(\rho^n)\cdot\nabla\rho^{n+1}=0,
 \quad \rho^{n+1}(0)=\rho_0.
\end{equation*}
For smooth approximations the characteristics exist because $H^s\hookrightarrow W^{1,\infty}$ for $s\ge3$. The solenoidal flow preserves $\rho_*\le\rho^{n+1}\le\rho^*$, so $\|\rho^n\bfv\|_2\le H_0$ at every step.

Fix $R>2\|\rho_0\|_{H^s}$. On the set $\{r\in H^s: \|r\|_{H^s}\le R, \ \rho_*\le r\le\rho^*\}$, multiplication by the fixed forcing gives $\|r\bfv\|_{H^s}\le C_{\bfv}R$. We apply \cref{prop:nonlinear-sobolev} with forcing radius $\widetilde R=C_{\bfv}R$ to bound $\Sop(r)$ in $H^s$. The density bounds ensure $\|r\bfv\|_2\le H_0$ on this set. The transport commutator satisfies
\[
 \|[\partial^\beta,\bv\cdot\nabla]r\|_2
 \le C_s\bigl(\|\nabla \bv\|_\infty\|r\|_{H^s}
+\|\bv\|_{H^s}\|\nabla r\|_\infty\bigr),
 \quad |\beta|\le s.
\]
For smooth approximations, testing the differentiated transport equation with $\partial^\beta\rho^{n+1}$ eliminates the principal term by solenoidality. Sobolev embedding gives
\begin{equation*}
 \frac{d}{dt}\|\rho^{n+1}\|_{H^s}
 \le C_R\|\rho^{n+1}\|_{H^s},
\end{equation*}
provided that the previous iterate remains in the ball. The characteristic argument for $L^2$ transport given next passes this estimate to the original data. Choose $T_*>0$ such that $e^{C_RT_*}\|\rho_0\|_{H^s}\le R/2$. This provides a common $H^s$ bound for all iterates on $[0,T_*]$.

For $n\ge1$, subtract two successive equations and put $d^{n+1}=\rho^{n+1}-\rho^n$. Then
\[
 \partial_td^{n+1}+\Sop(\rho^n)\cdot\nabla d^{n+1}
=-\bigl(\Sop(\rho^n)-\Sop(\rho^{n-1})\bigr)\cdot\nabla\rho^n,
 \quad d^{n+1}(0)=0.
\]
Since $H^{s-1}$ is an algebra, the $H^{s-1}$ stability in \cref{prop:nonlinear-sobolev} yields
\begin{equation*}
 \|d^{n+1}\|_{C([0,T_*];H^{s-1})}
 \le C_RT_*e^{C_RT_*}
                  \|d^n\|_{C([0,T_*];H^{s-1})}.
\end{equation*}
Reducing $T_*$ so that the coefficient is at most $1/2$, we obtain convergence of $\rho^n$ in $C_tH^{s-1}$ to a solution $\rho$, with a uniform $L^\infty_tH^s$ bound. The velocities converge in $C_tH^{s-1}$ as well. The corresponding difference estimate for two solutions proves uniqueness.

We establish continuity at the highest derivative level through the characteristic flow $X(t;a,\bx)$ of $\bu\in L^1_tW^{1,\infty}_x$. This flow preserves Lebesgue measure. For $z_0\in L^2$ and $f_z\in L^1_tL^2_x$, the formula
\[
 z(t,X(t;0,\bx))=z_0(\bx)
       +\int_0^t f_z(a,X(a;0,\bx))\dd a
\]
gives the unique $L^2$ transport solution and implies
\[
 \frac12\|z(t)\|_2^2-\frac12\|z_0\|_2^2
       =\int_0^t\int_\T f_z z\dx\dd a.
\]
Approximation of the scalar data along the fixed flow and measure preservation prove the identity and $z\in C_tL^2$. For $|\beta|\le s$, the distribution $r_\beta=\partial^\beta\rho$ satisfies
\[
 \partial_t r_\beta+\bu\cdot\nabla r_\beta
=-[\partial^\beta,\bu\cdot\nabla]\rho.
\]
The source belongs to $L^1_tL^2_x$ by the displayed commutator bound and $\bu\in L^1_tW^{1,\infty}_x$ is solenoidal. For $|\beta|\le s-1$, the initial trace follows directly from the $C_tH^{s-1}$ convergence. At order $|\beta|=s$, weak compactness and continuity in $H^{s-1}$ extend the common $H^s$ bound to every time. For $\varphi\in C^\infty(\T)$,
\[
 \langle r_\beta(t),\varphi\rangle
 =(-1)^{|\beta|}\int_\T\rho(t)\partial^\beta\varphi\dx
 \to(-1)^{|\beta|}\int_\T\rho_0\partial^\beta\varphi\dx
 =\langle\partial^\beta\rho_0,\varphi\rangle
 \quad\text{as }t\to0.
\]
Density of smooth functions in $L^2$ and the uniform derivative bound identify the weak $L^2$ initial trace as $\partial^\beta\rho_0$. The characteristic representation with this initial datum and the commutator source then gives $r_\beta\in C_tL^2_x$. Thus $\rho\in C_tH^s$. The $H^s$ local Lipschitz estimate of the instantaneous solution map gives $\bu\in C_tH^s$ and $\vartheta\in C_tH^{s+2}$. All bounds are on the local interval $[0,T_*]$.

Finally $\rho_t=-\bu\cdot\nabla\rho\in C_tH^{s-1}$. The $C^1$ solution map obtained by the implicit function theorem yields
\[
D_{\bu}\Phi\,\bu_t=\PP(\bfv\rho_t),\quad
 \bu_t\in C_tH^{s-1},\quad
 \vartheta_t=\Jop\bu_t\in C_tH^{s+1}.
\]
The pressure gradient identity and mean-zero normalization yield $p\in C_tH^{s+1}$. The instantaneous estimate gives $\|\bu(t)\|_2\le H_0/c_0$ and transport preserves the density bounds.
\end{proof}

\medskip
The instantaneous equations also give an energy identity and an explicit
stationary state with variable viscosity. Every instantaneous solution satisfies
\begin{equation*}
 \int_{\T}\mu(\vartheta)\bM_0\bu\cdot \bu\dx
       +\kappa_f\int_{\T}|\nabla\vartheta|^2\dx
 =\int_{\T} \bh\cdot \bu\dx.
\end{equation*}
This follows by testing the Darcy equation with $\bu$ and the heat equation with $\vartheta$. It explains the $L^2$ velocity bound used in \eqref{eq:lift-bound}. If $\bh=0$, the only instantaneous velocity and temperature are $\bu=0$ and $\vartheta=0$, irrespective of the contraction and absorption conditions.

The admissible class contains states with a spatially varying viscosity. For example, take $\rho_*=\rho^*=1$, choose $\mu_0>0$ and let
\[
 \mu(r)=\mu_0(1+\delta\tanh r),\quad 0<\delta<1,\quad
F(\bx)=\sin(2\pi x_1),\quad \bg=2\pi\cos(2\pi x_1)\be_1.
\]
For a nonzero amplitude $b$, prescribe
\begin{equation*}
 \begin{gathered}
 \rho\equiv1,\quad \bu=b\be_1,\quad
 \vartheta=\frac{b}{2\pi\kappa_f}\cos(2\pi x_1),\quad p=0,\\
 \bfv(\bx)=b\mu(\vartheta(\bx))\bM_0\be_1+\vartheta(\bx)\bg(\bx).
 \end{gathered}
\end{equation*}
These fields solve \eqref{eq:limit} and are stationary. The force is prescribed by this stationary state, whose temperature
and viscosity are nonconstant. Moreover $\|\bfv\|_2=O(|b|)$, so the contraction and absorption conditions hold for sufficiently small nonzero $b$ with $\mu$ fixed.

\medskip
Two further reductions recover familiar limiting cases. For constant viscosity, the instantaneous problem takes the linear coercive form
\begin{equation*}
\left\{
\begin{aligned}
 &\PP(\mu_0\bM_0\bu)+\Kop\bu=\PP(\rho\bfv),\\
& \vartheta=\Jop\bu.
\end{aligned}
\right.
\end{equation*}
Here $\mu_0>0$ is constant. Hence $L_\mu=0$, both coupling parameters $q_0$ and $\chi_0$ vanish. The viscosity residual $R_{\mu,\eps}$ is identically zero. Constant viscosity is compatible with nonzero $\Kop$ and nonconstant $\vartheta$, so the nonnegative operator $\Kop$ can still transmit thermal feedback. The term $\PP(\bM_0\bu)$ is retained because a constant anisotropic matrix need not map solenoidal fields into solenoidal fields.

If in addition $\bg=0$ (equivalently, $F$ is constant), the mean-zero elliptic heat equations give $\vartheta_\eps=\vartheta=0$. The effective equations then reduce exactly to
\begin{equation}\label{eq:isothermal-consistency}
\left\{
\begin{aligned}
 \partial_t\rho+\Div(\rho\bu)&=0,\quad \Div\bu=0,\\
 \mu_0\bM_0\bu+\nabla p&=\rho\bfv,\quad \av p=0.
\end{aligned}
\right.
\end{equation}
The system \eqref{eq:isothermal-consistency} has the density-dependent incompressible Darcy structure of \cite{BOP}. At the microscopic level, the thermal force disappears and the viscous term reduces to $-\mu_0\sig^2\Delta \bu_\eps$, with the same $\sig^4$ inertial scaling. The transport construction still yields a local-in-time strong solution when the two coupling parameters vanish.

The bounded-domain analysis retains the same interior resistance mechanism, with an additional correction for the mismatch between microscopic no-slip and macroscopic impermeability.

\section{Bounded-domain analytical tools and macroscopic well-posedness}\label{sec:bounded}
The bounded-domain result requires geometric and elliptic tools associated with the outer wall. Throughout this section, $\Oe$, $\He$ and $\sig$ refer to the geometry in \eqref{eq:bounded-geometry}. The Neumann heat estimate and conormal pressure estimate replace the periodic elliptic estimates, with zero-mean compatibility supplied by the conservative thermal force. Together with the geometric estimates, they construct the local macroscopic solution. The boundary-corrected lifting and the final relative-energy argument are deferred to \cref{sec:bounded-quantitative}.

\subsection{Geometric, thermal and conormal estimates}
The cell fields are constructed on the complete interior cubes. Their estimates extend to boundary strips by counting the cells that meet each strip. The cellwise zero mean of the divergence defect permits a local correction. A Neumann estimate controls the temperature error.

\begin{lemma}[Bounded-domain geometric tools]\label{lem:bounded-tools}
There are $\bW_\eps=(\bW_\eps^1,\bW_\eps^2,\bW_\eps^3)$ and $\bQ_\eps=(Q_\eps^1,Q_\eps^2,Q_\eps^3)$ for \eqref{eq:bounded-geometry} such that, after zero extension through the holes,
\begin{equation*}
 \|\bW_\eps\|_{L^\infty(\Omega)}\le C,\quad
 \|\bW_\eps-\bI\|_{L^2(\Omega)}\le C\eps^{\alpha-1},\quad
 \sig\bigl(\|\nabla \bW_\eps\|_{L^2(\Oe)}
                +\|\bQ_\eps\|_{L^2(\Oe)}\bigr)\le C,
\end{equation*}
and
\begin{equation}\label{eq:bounded-cell-residual}
 \left|\left\langle\sig^2(-\Delta \bW_\eps^i+\nabla Q_\eps^i)
          -\bM_0\be_i,\bzeta\right\rangle\right|
 \le C\eps\|\bzeta\|_{H^1(\Oe)}
 \quad(\bzeta\in H^1_0(\Oe;\R^3)).
\end{equation}
The columns are solenoidal and vanish on the holes, but they do not belong to $\V_{\eps,\Omega}$ since $\bW_\eps=\bI$ near $\partial\Omega$. For $C\eps\le\ell\le\ell_0$, they can be chosen to satisfy
\begin{equation}\label{eq:bounded-local-cell}
 \|\bW_\eps-\bI\|_{L^2(S_\ell)}\le C\ell^{1/2}\eps^{\alpha-1},
 \quad
 \|\nabla \bW_\eps\|_{L^2(S_\ell)}\le C\ell^{1/2}\sig^{-1}.
\end{equation}
Moreover,
\begin{equation}\label{eq:bounded-poincare}
 \|z\|_2\le C\sig\|\nabla z\|_2
       \quad(z\in H^1_0(\Oe)),\quad
 \|z\|_2\le C_{P,\Omega}\sig\|\sym z\|_2
       \quad(z\in\V_{\eps,\Omega}).
\end{equation}
All constants are independent of $\eps$ and $\ell$ in these ranges.
\end{lemma}
\begin{proof}
The same exterior Stokes fields as in \cref{lem:geometric-tools} are used in each complete cell. We fix the matching annulus between radii $\eps/8$ and $\eps/4$, set $\bW_\eps=\bI$ and $\bQ_\eps=0$ outside the matching balls and then normalize the assembled pressure globally. This fixed choice of matching radii yields the same resistance matrix $\bM_0$ and the same estimates. In particular $\bW_\eps=\bI$ on $S_{\eps/4}$. The unnormalized pressure is zero there as well. The collar width is proportional to $\eps$. The negative-norm estimate is the bounded-domain form of \cite[Proposition~2.7]{HLO}. Its tests have zero trace both on the holes and on $\partial\Omega$. The local bounds in \eqref{eq:bounded-local-cell} follow by counting cells that intersect $S_\ell$. There are at most $C\ell\eps^{-3}$ of them for $\ell\ge C\eps$ and each cell contributes at most $Ca_\eps^2\eps$ to $\|\bW_\eps-\bI\|_2^2$ and $Ca_\eps$ to $\|\nabla \bW_\eps\|_2^2$.

For the first inequality in \eqref{eq:bounded-poincare}, sum the cell estimate \eqref{eq:cell-poincare} over the complete cubes. Overlaps of their closed faces have measure zero. On $\Omega\setminus\mathcal Q_\eps$, contained in $S_{C\eps}$, the zero trace on $\partial\Omega$ and integration along inward normals give
\[
 \|z\|_{L^2(S_{C\eps})}\le C\eps
       \|\nabla z\|_{L^2(S_{C\eps})},
\]
where $z$ is extended by zero through the holes. Since $\eps\le\sig$ for small $\eps$, this completes the estimate. For solenoidal fields, $2\|\sym z\|_2^2=\|\nabla z\|_2^2$ follows from zero extension and integration by parts. The Poincar\'e estimate also controls the boundary strip formed by incomplete cells. The missing resistance on that strip contributes at most $C\eps^{3/2}\|\nabla\bzeta\|_2$, which is no larger than the right-hand side of \eqref{eq:bounded-cell-residual}.

\end{proof}

For the bounded-domain cell fields, we use the divergence defect $\mathfrak d_\eps(\bv)=\Div(\bW_\eps\bv)$ from \cref{sec:lift}. Multiplication by $\bW_\eps$ preserves the hole trace. For a field solenoidal in $\Omega$, the defect has zero mean on each complete cell, which permits a local correction with a constant independent of the hole-to-cell ratio.

\begin{lemma}[Uniform cellwise divergence repair]
\label{lem:bounded-local-repair}
Let $\bv\in H^1(\Omega;\R^3)$ satisfy $\Div\bv=0$ in $\Omega$. There is a linear operator $\mathfrak B_\eps$ such that
\begin{equation}\label{eq:bounded-local-bog}
 \begin{gathered}
 \mathfrak B_\eps(\bv)\in H^1_0(\Oe;\R^3),\quad
 \Div\mathfrak B_\eps(\bv)=\mathfrak d_\eps(\bv),\quad \|\nabla\mathfrak B_\eps(\bv)\|_2
       +\eps^{-1}\|\mathfrak B_\eps(\bv)\|_2
          \le C\|\mathfrak d_\eps(\bv)\|_2.
 \end{gathered}
\end{equation}
The constant is independent of $\eps$ and of $a_\eps/\eps$ for all sufficiently small $\eps$. The operator is defined on the solenoidal subspace of $H^1(\Omega;\R^3)$.
\end{lemma}
\begin{proof}
The uniformity of the constant follows from scaling fixed reference domains. More precisely, the standard Bogovski\u\i{} operator on $Q=(-1/2,1/2)^3$ has a constant depending only on $Q$. After the affine change $\bx=\eps(k+y)$, the estimate $\|\nabla \mathcal B_Q b\|_{L^2}\le C\|b\|_{L^2}$ is invariant under this rescaling. Likewise, $E_k=B(\eps k,a_\eps/2)\setminus T_{\eps,k}$ is exactly the image, under translation and dilation by $a_\eps$, of the fixed Lipschitz domain $B(0,1/2)\setminus\overline{T_0}$. Hence its zero-trace Bogovski\u\i{} constant is independent of both $a_\eps$ and $\eps$. All cells are translates of these two fixed reference geometries, so the same constant is valid for all $k$ and sufficiently small $\eps$.

On a complete cell, let $\bn_k$ denote its outward unit normal. Since $\bW_\eps=\bI$ near the faces and $\bW_\eps=0$ on the hole,
\[
 \int_{Q_{\eps,k}\setminus T_{\eps,k}}\mathfrak d_\eps(\bv)\dx
=\int_{\partial Q_{\eps,k}}\bv\cdot\bn_k\,\,\mathrm dS =\int_{Q_{\eps,k}}\Div \bv\dx=0.
\]
Outside the complete cells the defect vanishes. Each cell therefore admits its own zero-trace divergence solution with an $L^2$ gradient bound independent of the ratio $a_\eps/\eps$.

To construct a cellwise repair with this uniform bound, solve the divergence problem on the whole cube after extending the source by zero into the hole. The cube Bogovski\u\i{} solution $\bz_k$ satisfies $\|\nabla \bz_k\|_2\le C\|\mathfrak d_\eps(\bv)\|_{2,Q_{\eps,k}}$. Remove its trace on the hole by the small-ball correction used in \cref{app:periodic-geometry}. The flux through the hole is zero since $\Div \bz_k=0$ there. More explicitly, let $B_k=B(\eps k,a_\eps/2)$, $E_k=B_k\setminus T_{\eps,k}$ and let $\chi_k$ be a smooth cutoff equal to one near the hole and supported in $B_k$, with $|\nabla\chi_k|\le C/a_\eps$. Write $\mathcal B_{E_k}$ for the zero-trace Bogovski\u\i{} operator on $E_k$, obtained by rescaling the fixed reference annulus as in \cref{app:periodic-geometry}. The field $\bc_k=\chi_k\bz_k-\mathcal B_{E_k}\Div(\chi_k\bz_k)$ on $E_k$, extended by $\bz_k$ into the hole and by zero outside the small ball, is globally solenoidal. Hence $\bz_k-\bc_k$ has the original divergence, zero hole trace and zero cell-face trace. On a ball of radius $Ca_\eps$, the correction cost is bounded by
\[
C\bigl(\|\nabla \bz_k\|_2+ a_\eps^{-1}\|\bz_k\|_{L^2(B(\eps k,Ca_\eps))}\bigr)
 \le C\bigl(\|\nabla \bz_k\|_2+\|\bz_k\|_6\bigr)
 \le C\|\mathfrak d_\eps(\bv)\|_{2,Q_{\eps,k}}.
\]
Here the last Sobolev inequality is scale invariant for $\bz_k\in H^1_0(Q_{\eps,k})$ in dimension three. Since $\chi_k=0$ on $\partial B_k$ and $\chi_k=1$ on $\partial T_{\eps,k}$, the source of the annular Bogovski\u\i{} correction satisfies
\[
 \int_{E_k}\Div(\chi_k\bz_k)\dx
 =\int_{\partial T_{\eps,k}}\bz_k\cdot\bn_{E_k}\,\mathrm dS
 =-\int_{T_{\eps,k}}\Div\bz_k\dx=0,
\]
where $\bn_{E_k}$ is the outward normal of the fluid annulus,
pointing into the hole on its inner boundary. The zero-extended source
vanishes in the hole, so $\Div\bz_k=0$ there. Thus
$\mathcal B_{E_k}$ is applicable. The traces of $\chi_k\bz_k$ and of the zero-trace Bogovski\u\i{} term show that $\bc_k$ agrees with $\bz_k$ on the hole boundary and vanishes on $\partial B_k$. Consequently $\bz_k-\bc_k$ has zero trace both on the hole and on every cell face. Its zero extension through the hole belongs to $H^1_0(Q_{\eps,k})$, so the cell Poincar\'e inequality gives
$\|\bz_k-\bc_k\|_2\le C\eps\|\nabla(\bz_k-\bc_k)\|_2$.
The cube Bogovski\u\i{} map, multiplication by the fixed cutoff and the annular Bogovski\u\i{} map are linear in the cell source and hence in $\bv$. Summing over the disjoint cell interiors therefore defines a linear global operator. The zero cell-face traces create no jump terms, and the sum belongs to $H^1_0(\Oe)$. Squaring and summing the cell estimates proves both bounds in \eqref{eq:bounded-local-bog}. The construction is the localized mechanism behind \cite[Lemma~2.8]{HLO}. For general sources, the $q=2$ case of \cite[Theorem~2.3]{DFL}, equivalently the scale recorded in \cite[Proposition~2.5]{HLO}, gives the complementary global bound $C(1+\sig^{-1})$.
\end{proof}

A scalar Neumann estimate complements the geometric bounds. As on the torus, measurable ellipticity is sufficient, while the variational formulation also records the conormal flux generated by a divergence source.

\begin{lemma}[Interface-independent Neumann estimate]
\label{lem:bounded-scalar}
Let $a$ be a measurable scalar function on $\Omega$ satisfying $\kappa_*\le a\le\kappa^*$. For $b\in L^2(\Omega)$ with $\int_\Omega b=0$ and $\bG\in L^q(\Omega;\R^3)$, $q>3$, let $\psi\in H^1_\diamond(\Omega)$ satisfy
\begin{equation}\label{eq:bounded-scalar-weak}
 \int_\Omega a\nabla\psi\cdot\nabla\phi\dx
   =\int_\Omega b\phi\dx-\int_\Omega\bG\cdot\nabla\phi\dx
          \quad(\phi\in H^1(\Omega)).
\end{equation}
Then
\begin{equation}\label{eq:bounded-scalar-est}
 \|\psi\|_{H^1(\Omega)}+\|\psi\|_\infty
       \le C_{\Omega,q}(\|b\|_2+\|\bG\|_q).
\end{equation}
The constant depends on the fixed domain, ellipticity and $q$, not on the interfaces of $a$.
\end{lemma}
Equation \eqref{eq:bounded-scalar-weak} represents the natural conormal condition $(a\nabla\psi+\bG)\cdot\bn=0$ in variational form. For merely measurable $a$, this condition is not interpreted as a pointwise trace.
\begin{proof}
The mean-zero Poincar\'e inequality gives existence and $\|\psi\|_2\le C_1M$, where $M=\|b\|_2+\|\bG\|_q$. If $M=0$, coercivity gives $\psi=0$. Otherwise set $k_0=\sqrt{2/|\Omega|}\,C_1M$. Chebyshev's inequality gives $|\{\psi>k_0\}|\le|\Omega|/2$. If $\varphi_k=(\psi-k)_+$ and $k\ge k_0$, the fixed-domain Sobolev--Poincar\'e inequality for functions vanishing on a set of measure at least $|\Omega|/2$ gives $\|\varphi_k\|_6\le C_\Omega\|\nabla\varphi_k\|_2$. One can use $\varphi_k$ as a test because \eqref{eq:bounded-scalar-weak} is valid for all $H^1$ tests. The estimate and iteration in \cref{lem:scalar} now apply with $\gamma=\min\{1/3,1/2-1/q\}>1/6$. Repeating for $-\psi$ gives \eqref{eq:bounded-scalar-est}. 
\end{proof}

With the constants chosen in the bounded-domain part of \cref{sec:main-results}, apply the lemma to the temperature error $z_\eps$. Subtracting the heat equations gives
\begin{equation}\label{eq:bounded-heat-difference}
 \int_\Omega\bar\kappa_\eps\nabla z_\eps\cdot\nabla\psi\dx
 =\int_\Omega(\widetilde \bu_\eps-\bu)\cdot \bg\,\psi\dx
  -\int_\Omega(\bar\kappa_\eps-\kappa_f)
                     \nabla\vartheta\cdot\nabla\psi\dx.
\end{equation}
Both velocities are in $\V_\Omega$, so the first source has mean zero. The second source is supported on the holes, away from $\partial\Omega$. Consequently
\begin{equation}\label{eq:bounded-temp-error}
 \|z_\eps(t)\|_\infty
 \le C_{g,\Omega}\|\widetilde \bu_\eps(t)-\bu(t)\|_2
    +C_{\Omega,q}|\kappa_s-\kappa_f|h_\eps^{1/q}
                           \|\nabla\vartheta(t)\|_\infty.
\end{equation}
The matched insulation conditions give the same thermal error structure as on the torus, with the Neumann compatibility encoded in \eqref{eq:bounded-heat-difference}.

\subsection{Local well-posedness of the bounded Darcy--heat system}\label{sec:bounded-macro}
Let $\PP_\Omega$ be the orthogonal projection onto $\V_\Omega$. Using $\Lop_\Omega$, $\Bop_\Omega$ and $\Jop_\Omega$ from the bounded-domain part of \cref{sec:main-results}, set
\[
 \Kop_\Omega=\Bop_\Omega^*\Jop_\Omega,\quad
 \Bop_\Omega^*\psi=\PP_\Omega(\psi\bg).
\]
The operator $\Kop_\Omega$ is self-adjoint and nonnegative. The $L^2$ fixed-point argument from the torus therefore applies on $\V_\Omega$. To obtain higher regularity, we first estimate the pressure in the associated conormal problem. This replaces the Fourier argument used in \cref{lem:frozen-sobolev}.

The variable-coefficient pressure estimate rests on the following constant-coefficient conormal regularity result.

\begin{lemma}[Constant-coefficient conormal regularity]
\label{lem:constant-conormal}
Let $\bA_0$ be a fixed symmetric positive definite matrix, let $r\ge1$ be an integer and assume $\partial\Omega\in C^{r,1}$. Suppose $\bF\in H^r(\Omega;\R^3)$ and $f\in H^{r-1}(\Omega)$ satisfy $\int_\Omega f\dx=0$. The mean-zero variational solution of
\[
 \int_\Omega\bA_0\nabla q\cdot\nabla\phi\dx
=\int_\Omega\bF\cdot\nabla\phi\dx-\int_\Omega f\phi\dx
 \qquad(\phi\in H^1(\Omega))
\]
belongs to $H^{r+1}(\Omega)$ and satisfies
\begin{equation}\label{eq:constant-conormal-regularity}
 \|q\|_{H^{r+1}}
 \le C_{r,\Omega,\bA_0}
       \bigl(\|\bF\|_{H^r}+\|f\|_{H^{r-1}}+\|q\|_{H^1}\bigr).
\end{equation}
The boundary condition is $\bn\cdot\bA_0\nabla q=\bn\cdot\bF$ in the conormal sense. In particular, if $j\ge0$ is an integer, $b\in H^j_\diamond(\Omega)$ and $\partial\Omega\in C^{j+1,1}$, then the mean-zero solution of $-\kappa_f\Delta\theta=b$, $\partial_n\theta=0$, satisfies
\[
 \|\theta\|_{H^{j+2}(\Omega)}\le C_{j,\Omega,\kappa_f}\|b\|_{H^j(\Omega)}.
\]
\end{lemma}
\begin{proof}
Lax--Milgram on $H^1_\diamond(\Omega)$ gives the variational solution. Testing with constants gives the compatibility condition. We next prove the boundary regularity by difference quotients.

A $C^{r,1}$ boundary chart and its inverse have bounded weak derivatives through order $r+1$. Flattening such a chart transforms the bilinear form into one with symmetric uniformly elliptic coefficients $\widehat A\in W^{r,\infty}$ on a half-ball. In this proof hats denote the quantities in these local coordinates. The transformed vector source is in $H^r$ and the scalar source is in $H^{r-1}$. All changes of variables and multiplication estimates have constants controlled by the fixed atlas. The localized weak formulation has unrestricted traces on the flat boundary, as is appropriate for the conormal problem.

At the first level $r=1$, let $D_h^jz=(z(\,\cdot+h\be_j)-z)/h$ for a tangential direction $j=1,2$. Choose a smooth cutoff $\zeta$ supported in the chart and take $\phi=-D_{-h}^j(\zeta^2D_h^j\widehat q)$ in the local weak equation. For sufficiently small $h$, this test belongs to $H^1$ and is supported away from the curved edge. Discrete integration by parts produces the coercive term $\int\zeta^2\widehat A(\,\cdot+h\be_j)\nabla D_h^j\widehat q\cdot\nabla D_h^j\widehat q$. The difference-quotient bound on $H^1$ gives
\[
 \left|\int\widehat f\,D_{-h}^j(\zeta^2D_h^j\widehat q)\right|
 \le C\|\widehat f\|_2
       \bigl(\|\zeta\nabla D_h^j\widehat q\|_2+\|\widehat q\|_{H^1}\bigr).
\]
The vector-source term is bounded using $\|D_h^j\widehat{\bF}\|_2\le\|\partial_j\widehat{\bF}\|_2$. The coefficient commutator is bounded by
\[
 C\|\nabla\widehat A\|_\infty\|\widehat q\|_{H^1}
       \|\zeta\nabla D_h^j\widehat q\|_2,
\]
because $\|D_h^j\widehat A\|_\infty\le\|\partial_j\widehat A\|_\infty$. The bound $\|D_h^j\widehat q\|_2\le\|\partial_j\widehat q\|_2$ controls the cutoff terms. Coercivity and Young's inequality give
\[
 \|\zeta\nabla D_h^j\widehat q\|_2
 \le C\bigl(\|\widehat{\bF}\|_{H^1}+\|\widehat f\|_2+\|\widehat q\|_{H^1}\bigr)
\]
uniformly in $h$. Weak compactness as $h\to0$ controls all second derivatives with a tangential direction. In distributions,
\[
 \widehat A_{33}\partial_{33}\widehat q
 =\Div\widehat{\bF}+\widehat f
  -\sum_{(i,j)\ne(3,3)}\widehat A_{ij}\partial_{ij}\widehat q
  -\sum_{i,j}(\partial_i\widehat A_{ij})\partial_j\widehat q.
\]
The right-hand side belongs to $L^2$ and $\widehat A_{33}$ is bounded below, which recovers $\partial_{33}\widehat q$ and proves the local $H^2$ estimate.

For the induction step suppose $q\in H^r$. The same tests applied to $r$ tangential difference quotients give, on nested half-balls $U'\subset\subset U$ relative to the flat boundary,
\[
 \sum_{|\gamma|=r,\,\gamma_3=0}
 \|\nabla\partial^\gamma\widehat q\|_{L^2(U')}
 \le C\bigl(\|\widehat{\bF}\|_{H^r(U)}
+\|\widehat f\|_{H^{r-1}(U)} +\|\widehat q\|_{H^r(U)}\bigr).
\]
Here every coefficient commutator has at least one derivative on $\widehat A$ and at most $r$ derivatives on $\widehat q$. For a tangential multi-index $\gamma$ of length $r$, choose $j$ with $\gamma_j>0$. The scalar-source term satisfies
\[
 |\langle\partial^\gamma\widehat f,\phi\rangle|=|\langle\partial^{\gamma-\be_j}\widehat f,\partial_j\phi\rangle|\le C\|\widehat f\|_{H^{r-1}}\|\phi\|_{H^1}.
\]
Discrete integration by parts gives the same bound for difference quotients, using only $r-1$ derivatives of the scalar source. The cutoff terms involve derivatives of $\widehat q$ of order at most $r$. The estimates are uniform in the difference-quotient step, so weak compactness gives all derivatives of order $r+1$ with at most one normal derivative.

We recover the other derivatives by induction on their number of normal derivatives. Fix $N\in\{2,\ldots,r+1\}$ and suppose all derivatives of order $r+1$ with fewer than $N$ normal derivatives have been controlled. For $|\beta|=r-1$ and $\beta_3=N-2$, differentiate the distributional equation to obtain
\[
\begin{aligned}
 \widehat A_{33}\partial^{\beta+2\be_3}\widehat q
 &=\partial^\beta(\Div\widehat{\bF}+\widehat f)
   \!-\!\!\!\sum_{(i,j)\ne(3,3)}\widehat A_{ij}
                      \partial^{\beta+\be_i+\be_j}\widehat q\\
 &\quad-\sum_{i,j}[\partial^\beta,\widehat A_{ij}]
                                      \partial_{ij}\widehat q
   -\sum_{i,j}\partial^\beta
                    \bigl((\partial_i\widehat A_{ij})\partial_j\widehat q\bigr).
\end{aligned}
\]
The principal terms in the first sum contain fewer than $N$ normal derivatives. In the remaining sums at least one derivative falls on a coefficient and at most $r$ derivatives fall on $\widehat q$, so their $L^2$ norms are bounded by $C\|\widehat A\|_{W^{r,\infty}}\|\widehat q\|_{H^r}$. The source is in $L^2$ because $\widehat{\bF}\in H^r$ and $\widehat f\in H^{r-1}$. Since $\widehat A_{33}$ is bounded below, this identity controls every derivative of order $r+1$ with $N$ normal derivatives. Induction over $N$ completes the local estimate.

Interior estimates follow from ordinary difference quotients. A finite partition of unity and the induction hypothesis yield \eqref{eq:constant-conormal-regularity}. For the Neumann heat estimate take $\bA_0=\bI$, $\bF=0$, $f=-b/\kappa_f$ and $r=j+1$. The energy estimate controls the remaining $H^1$ norm. This gives the heat estimate under $C^{j+1,1}$ boundary regularity.
\end{proof}

The pressure estimate requires an $H^2$ bound for a variable principal matrix.
\begin{lemma}[First conormal regularity step]\label{lem:lipschitz-conormal}
Assume $\partial\Omega\in C^{1,1}$. Let $\bA\in W^{1,\infty}(\Omega;\R^{3\times3})$ be symmetric and uniformly elliptic. If $\bF\in H^1(\Omega;\R^3)$ and $f\in L^2(\Omega)$ with $\int_\Omega f\dx=0$, then the mean-zero variational solution of
\[
 \int_\Omega \bA\nabla q\cdot\nabla\phi\dx
 =\int_\Omega\bF\cdot\nabla\phi\dx-\int_\Omega f\phi\dx
 \qquad(\phi\in H^1(\Omega))
\]
belongs to $H^2(\Omega)$ and satisfies
\[
 \|q\|_{H^2}
 \le C\bigl(\|\bF\|_{H^1}+\|f\|_2+\|q\|_{H^1}\bigr),
\]
where $C$ depends only on the ellipticity constants, $\|\bA\|_{W^{1,\infty}}$, and the fixed $C^{1,1}$ atlas of $\Omega$.
\end{lemma}
\begin{proof}
The $r=1$ argument in the proof of \cref{lem:constant-conormal} uses only the ellipticity and $W^{1,\infty}$ norm of the flattened matrix. A $C^{1,1}$ chart transforms the present matrix into one with these properties. The same test $-D_{-h}^j(\zeta^2D_h^jq)$ controls tangential second derivatives with $\bF\in H^1$ and $f\in L^2$. The displayed distributional equation in that proof recovers the second normal derivative. Interior difference quotients and a finite partition of unity give the claimed estimate, with precisely the stated coefficient dependence.
\end{proof}

In the nonlinear Darcy problem the pressure equation carries the positive scalar factor $a=\mu(\vartheta)^{-1}$. The next estimate incorporates this coefficient under positivity and Sobolev bounds for arbitrary $\nabla a$ satisfying the stated regularity assumptions.

\begin{lemma}[Conormal pressure estimates with a positive scalar coefficient]
\label{lem:bounded-conormal}
Let $\bA_0$ be a fixed symmetric positive definite matrix and assume $\partial\Omega$ is of class $C^{1,1}$. Suppose $a\in H^2(\Omega)$ and $0<a_*\le a\le a^*$ almost everywhere in $\Omega$. For a vector field $\bH\in H^1(\Omega;\R^3)$, let $p\in H^1_\diamond(\Omega)$ be the variational solution
\begin{equation}\label{eq:conormal-weak}
 \int_\Omega a\bA_0\nabla p\cdot\nabla\phi\dx
   =\int_\Omega a\bA_0\bH\cdot\nabla\phi\dx
       \quad(\phi\in H^1(\Omega)).
\end{equation}
Then $\bv=a\bA_0(\bH-\nabla p)\in H^1(\Omega)\cap\V_\Omega$ and
\begin{equation}\label{eq:conormal-base-estimate}
 \|p\|_{H^2}+\|\bv\|_{H^1}\le C\|\bH\|_{H^1}.
\end{equation}
The constant depends only on $\Omega,\bA_0,a_*,a^*$ and an upper bound for $\|a\|_{H^2}$, uniformly along the smooth approximations. More generally, if $m\ge2$ is an integer, $\partial\Omega$ is of class $C^{m,1}$, $a\in H^{m+1}(\Omega)$ and $\bH\in H^m(\Omega)$, then
\begin{equation}\label{eq:conormal-higher-estimate}
 \|p\|_{H^{m+1}}+\|\bv\|_{H^m}\le C_m\|\bH\|_{H^m},
\end{equation}
where $C_m$ may additionally depend on $\|a\|_{H^{m+1}}$.
\end{lemma}
The proof uses the positive lower bound for $a$ together with its stated Sobolev bounds. The size of $\nabla a$ enters only through the constants in these bounds.
\begin{proof}
Coercivity on mean-zero scalars first gives $\|p\|_{H^1}\le C\|\bH\|_2$. To justify the first regularity step, approximate $\log a$ in $H^2$ by smooth functions and exponentiate. Since $H^2(\Omega)\hookrightarrow C^0(\overline\Omega)$, the resulting $a_j \to a$ in $H^2$ uniformly, with a common positive lower bound and upper bound. Approximate $\bH$ in $H^1$ by smooth $\bH_j$ and solve \eqref{eq:conormal-weak} for $p_j$. These solutions have a uniform $H^1$ bound.

For fixed $j$, \cref{lem:lipschitz-conormal} applied to $a_j\bA_0$ gives $p_j\in H^2(\Omega)$. We derive an estimate uniform in $j$. For a smooth scalar $\phi$, the test $\phi/a_j$ is admissible
in \eqref{eq:conormal-weak}. Expanding its gradient gives
\begin{equation}\label{eq:conormal-divided-weak}
 \int_\Omega\bA_0\nabla p_j\cdot\nabla\phi\dx
 =\int_\Omega\bA_0\bH_j\cdot\nabla\phi\dx
   -\int_\Omega f_j\phi\dx,\qquad
 f_j=\nabla\log a_j\cdot\bA_0(\bH_j-\nabla p_j).
\end{equation}
Since $p_j\in H^2$, we have $f_j\in L^2$. Density extends this identity
to every $\phi\in H^1(\Omega)$. It is the variational form of the
constant-principal-part conormal problem
\begin{equation}\label{eq:bounded-pressure-constant-principal}
 \left\{
\begin{aligned}
 \Div(\bA_0\nabla p_j)&=\Div(\bA_0\bH_j)
       +\nabla\log a_j\cdot\bA_0(\bH_j-\nabla p_j),
     &&\text{in }\Omega,\\
 \bn\cdot\bA_0\nabla p_j&=\bn\cdot\bA_0\bH_j,
     &&\text{on }\partial\Omega.
\end{aligned}
\right.
\end{equation}
Taking $\phi=1$ in \eqref{eq:conormal-divided-weak} gives the compatibility identity
\begin{equation*}
 \int_\Omega\nabla\log a_j\cdot \bA_0(\bH_j-\nabla p_j)\dx=0.
\end{equation*}
Thus the additional interior term has zero integral and the volume integral of the right-hand side of \eqref{eq:bounded-pressure-constant-principal} equals the prescribed conormal flux. Applying \cref{lem:constant-conormal} with $r=1$ and using \eqref{eq:constant-conormal-regularity} gives
\[
 \|p_j\|_{H^2}\le C\left(\|\bH_j\|_{H^1}+\|p_j\|_{H^1}
+\|\nabla\log a_j\|_6 (\|\bH_j\|_3+\|\nabla p_j\|_3)\right).
\]
With $b_j=\|\nabla\log a_j\|_6$, interpolation followed by Young's inequality gives, for any fixed $\delta>0$,
\[
b_j\|\nabla p_j\|_3
 \le Cb_j\|p_j\|_{H^1}^{1/2}\|p_j\|_{H^2}^{1/2}
 \le\delta\|p_j\|_{H^2}+C_\delta b_j^2\|p_j\|_{H^1}.
\]
After $\delta>0$ is chosen sufficiently small, the highest norm is absorbed. The $H^1$ energy bound and the uniform $H^2$ coefficient bound give an $H^2$ estimate for $p_j$ with a constant independent of $j$. Subtracting the variational problems shows $p_j\to p$ in $H^1$, since $a_j\to a$ uniformly and $\bH_j\to\bH$ in $L^2$. Weak compactness in $H^2$ therefore identifies the weak $H^2$ limit with the original variational solution $p$.  The formula $\bv=a\bA_0(\bH-\nabla p)$ and Sobolev multiplication give \eqref{eq:conormal-base-estimate}, and \eqref{eq:conormal-weak} gives $\bv\in\V_\Omega$.

For the higher estimates we use induction. Assume $p\in H^m$ and $a\in H^{m+1}$ for some $m\ge2$. Since $a\ge a_*>0$, the composition estimate in \cite[Theorem~2.87]{BCD}, localized to the fixed domain, gives
\[
 \|\nabla\log a\|_{H^m}
 \le C(a_*,\|a\|_{H^{m+1}}).
\]
Moreover $\bH-\nabla p\in H^{m-1}$. In three dimensions $H^m(\Omega)$ is a Banach algebra for $m\ge2$ and multiplication by an $H^m$ function is bounded on $H^{m-1}$. Hence
\[
f_m:=\nabla\log a\cdot\bA_0(\bH-\nabla p)
       \in H^{m-1}(\Omega),
 \qquad
 \|f_m\|_{H^{m-1}}
 \le C_m\bigl(\|\bH\|_{H^m}+\|p\|_{H^m}\bigr).
\]
Sobolev composition gives $a^{-1}\in H^{m+1}(\Omega)$.
Testing \eqref{eq:conormal-weak} with $\phi/a$ for smooth $\phi$ and
then using density yields
\[
 \int_\Omega\bA_0\nabla p\cdot\nabla\phi\dx
 =\int_\Omega\bA_0\bH\cdot\nabla\phi\dx
       -\int_\Omega f_m\phi\dx
       \quad(\phi\in H^1(\Omega)).
\]
In particular $\int_\Omega f_m\dx=0$. Thus both the interior equation
and its natural conormal condition hold in the variational sense. Since
$\bA_0\bH\in H^m(\Omega;\R^3)$ and the boundary is $C^{m,1}$,
\cref{lem:constant-conormal} applies with $r=m$ and yields
$p\in H^{m+1}$ and
\[
 \|p\|_{H^{m+1}}
 \le C_m\bigl(\|\bH\|_{H^m}+\|p\|_{H^m}\bigr).
\]
The induction hypothesis absorbs the lower-order pressure norm. The formula $\bv=a\bA_0(\bH-\nabla p)$ and the same Sobolev product estimates give $\bv\in H^m$. Iteration proves \eqref{eq:conormal-higher-estimate}. The constants depend on $\Omega$, $\bA_0$, $a_*$, $a^*$ and an upper bound for $\|a\|_{H^{m+1}}$. The smooth approximation argument preserves these dependencies at each finite order.
\end{proof}

The estimate in \cref{prop:bounded-instantaneous} raises the $L^2$ fixed point to $H^s$. For $s\ge4$ this gives $\bu\in W^{2,\infty}$, so the tangent flow preserves $\Omega$ and closes the $H^s$ transport estimate for the density. We first record the instantaneous solution map and its stability.

\begin{proposition}[Instantaneous bounded-domain problem]
\label{prop:bounded-instantaneous}
If $q_{0,\Omega}\le1/2$ and $\|\bh\|_{L^2(\Omega)}\le H_{0,\Omega}$, there is a unique $(\bu,\vartheta,p)\in\V_\Omega\times H^2_\diamond(\Omega) \times H^1_\diamond(\Omega)$ satisfying
\[
\left\{
\begin{aligned}
 &\mu(\vartheta)\bM_0\bu+\nabla p=\bh-\vartheta\bg,\\
 &\vartheta=\Jop_\Omega\bu,\quad \bu\in\V_\Omega,\ \
 \int_\Omega p\dx=\int_\Omega\vartheta\dx=0.
\end{aligned}
\right.
\]
It obeys $\|\bu\|_2\le H_{0,\Omega}/c_0$ and
\begin{equation}\label{eq:bounded-instantaneous-L2}
 \|\bu_1-\bu_2\|_2\le\frac{1}{c_0(1-q_{0,\Omega})}\|\bh_1-\bh_2\|_2.
\end{equation}
Let $s\ge4$ be an integer and, for the following higher-order statements, assume additionally that $\partial\Omega$ is of class $C^{s+1,1}$. On sets where $\|\bh\|_{H^s}\le R$,
\begin{equation}\label{eq:bounded-macro-Hs}
 \|\bu\|_{H^s}+\|\vartheta\|_{H^{s+2}}+\|p\|_{H^{s+1}}\le C_R.
\end{equation}
For two such forcings and every integer $0\le m\le s$,
\begin{equation}\label{eq:bounded-macro-difference}
 \begin{aligned}
 \|\bu_1-\bu_2\|_{H^m}+\|\vartheta_1-\vartheta_2\|_{H^{m+2}}+\|p_1-p_2\|_{H^{m+1}}\le C_R\|\bh_1-\bh_2\|_{H^m}.
 \end{aligned}
\end{equation}
The velocity solution map is locally $C^1$ from $H^{s-1}(\Omega;\R^3)$ to $H^{s-1}(\Omega;\R^3)\cap\V_\Omega$ near each such forcing. Constants may depend on finitely many derivatives of $\mu$ on the attained temperature interval.
\end{proposition}

\begin{proof}
For each $\tau\in L^\infty_\diamond(\Omega)$, let $\bu_\tau\in\V_\Omega$ be the unique solution of
\[
 \int_\Omega\mu(\tau)\bM_0\bu_\tau\cdot \bv\dx
+\langle \Kop_\Omega\bu_\tau,\bv\rangle =\int_\Omega \bh\cdot \bv\dx\quad(\bv\in\V_\Omega).
\]
Coercivity gives $\|\bu_\tau\|_2\le H_{0,\Omega}/c_0$. For $\tau_1,\tau_2\in L^\infty_\diamond(\Omega)$, subtraction of the frozen equations and testing with $\bu_{\tau_1}-\bu_{\tau_2}$ give contraction of $\tau\mapsto\Jop_\Omega\bu_\tau$ in $L^\infty$ with factor $q_{0,\Omega}$. It maps into the mean-zero ball of radius $C_{\infty,\Omega}H_{0,\Omega}/c_0$, which is complete in $L^\infty$. Banach's theorem gives the fixed point. Every solution obeys the same energy bound and hence belongs to that ball, proving uniqueness in the asserted class. Subtraction of the equations gives \eqref{eq:bounded-instantaneous-L2}. The orthogonal complement of $\V_\Omega$ is the closed space of gradients of $H^1$ functions, so pressure is recovered uniquely after fixing its mean.

To raise the spatial regularity, express the velocity through the conormal pressure problem. Put $\bA_0=\bM_0^{-1}$, $a=\mu(\vartheta)^{-1}$ and $\bH=\bh-\vartheta \bg$. Then $\bu=a\bA_0(\bH-\nabla p)$, where
\begin{equation}\label{eq:bounded-pressure}
\left\{
\begin{aligned}
 \Div(a\bA_0\nabla p)&=\Div(a\bA_0\bH),
     &&\text{in }\Omega,\\
 \bn\cdot a\bA_0\nabla p&=\bn\cdot a\bA_0\bH,
     &&\text{on }\partial\Omega,\\
 \int_\Omega p\dx&=0.
\end{aligned}
\right.
\end{equation}
This problem is understood in the variational sense of \eqref{eq:conormal-weak}. Initially $\vartheta=\Jop_\Omega\bu\in H^2$, hence $a=\mu(\vartheta)^{-1}\in H^2$. For $\bh\in H^1$, \cref{lem:bounded-conormal} gives $p\in H^2$ and $\bu\in H^1$. The Neumann heat equation then gives $\vartheta\in H^3$. At order $m\ge2$, the preceding level $\bu\in H^{m-1}$ implies $\vartheta,a\in H^{m+1}$ and $\bH\in H^m$ whenever $\bh\in H^m$. Since $m\le s$, the assumed $C^{s+1,1}$ boundary is more than sufficient for \eqref{eq:conormal-higher-estimate}, whose pressure step itself requires only $C^{m,1}$. This gives $\bu\in H^m$ and $p\in H^{m+1}$ and then $\vartheta\in H^{m+2}$. At the last heat step, apply the Neumann heat estimate in \cref{lem:constant-conormal} with $j=s$ to $b=\bu\cdot\bg\in H^s_\diamond(\Omega)$. Its required boundary class is precisely $C^{s+1,1}$, the class assumed here. This finite induction proves \eqref{eq:bounded-macro-Hs}. The variational change of test in \cref{lem:bounded-conormal}
retains the natural conormal condition throughout this bootstrap.

We next estimate the difference of two solutions at each Sobolev order. Let $\bv=\bu_1-\bu_2$, $\br=\bh_1-\bh_2$ and $\pi=p_1-p_2$. Set
\[
 \beta_{12}=\int_0^1\mu'\bigl(\vartheta_2+t(\vartheta_1-\vartheta_2)\bigr)\dd t,
 \quad \bc_{12}=\beta_{12}\bM_0\bu_2.
\]
The uniform $H^s$ bounds imply $\|\bc_{12}\|_{H^s}\le C_R$. The difference problem is
\begin{equation}\label{eq:bounded-difference-equation}
 \mu(\vartheta_1)\bM_0\bv+\nabla\pi
       =\br-(\Jop_\Omega\bv)\bg-\bc_{12}\Jop_\Omega\bv,
 \quad \bv\in\V_\Omega.
\end{equation}
At order zero, $|\beta_{12}|\le L_\mu$ and $\|\bu_2\|_2\le H_{0,\Omega}/c_0$, so
\[
 \left|\int_\Omega (\Jop_\Omega\bv)\,\bc_{12}\cdot\bv\dx\right|
 \le \frac{\Lambda L_\mu H_{0,\Omega}}{c_0}C_{\infty,\Omega}\|\bv\|_2^2
=c_0q_{0,\Omega}\|\bv\|_2^2.
\]
The $\bg\,\Jop_\Omega\bv$ term has nonnegative quadratic form. This proves the displayed $L^2$ bound and, by pressure recovery and the heat equation, the $m=0$ case of \eqref{eq:bounded-macro-difference}.

For $m=1,2$, multiplication and the two-derivative gain of $\Jop_\Omega$ give $\|\bc_{12}\Jop_\Omega\bv\|_{H^m}\le C_R\|\bv\|_2$. For $2\le m\le s$ they give
\[
 \|\bc_{12}\Jop_\Omega\bv\|_{H^m}
+\|\bg\,\Jop_\Omega\bv\|_{H^m}
       \le C_R\|\bv\|_{H^{m-2}}.
\]
The conormal estimate in \cref{lem:bounded-conormal} applies to \eqref{eq:bounded-difference-equation} with the fixed coefficient $a_1=\mu(\vartheta_1)^{-1}$. The right-hand side consists of $\br$ and quantities bounded by already controlled lower orders. Starting from \eqref{eq:bounded-instantaneous-L2} and inducting proves \eqref{eq:bounded-macro-difference}, including the pressure and thermal norms.

The differentiability of the solution map follows from the linearized equation in the same Sobolev spaces. For $m=s-1$, write $\mathcal X_\Omega^m=H^m(\Omega;\R^3)\cap\V_\Omega$ and $\mathcal Y_\Omega^m=H^m(\Omega;\R^3)$ with their $H^m$ norms. The Leray projection is bounded $\mathcal Y_\Omega^m\longrightarrow\mathcal X_\Omega^m$. Indeed, $\PP_\Omega\bv=\bv-\nabla q$, where $q$ satisfies \cref{lem:constant-conormal} with $\bA_0=\bI$, $\bF=\bv$ and $f=0$. Taking $r=m$ gives $q\in H^{m+1}$ and the required bound. The boundary regularity $C^{m,1}$ follows from $\partial\Omega\in C^{s+1,1}$ and $m\le s$. The map
\[
 \Phi_\Omega:\mathcal X_\Omega^m\times\mathcal Y_\Omega^m\longrightarrow\mathcal X_\Omega^m,
 \quad
 \Phi_\Omega(\bu,\bh)=\PP_\Omega\bigl(\mu(\Jop_\Omega\bu)\bM_0\bu\bigr)
+\Kop_\Omega\bu-\PP_\Omega \bh
\]
is $C^1$ by Sobolev multiplication and smooth composition. At a solution its velocity derivative is
\begin{equation*}
 \mathcal A_{\bu} \bv
   =\PP_\Omega\bigl(\mu(\vartheta)\bM_0\bv
            +\mu'(\vartheta)(\Jop_\Omega\bv)\bM_0\bu\bigr)+\Kop_\Omega\bv.
\end{equation*}
For $\bv\in\V_\Omega$, there holds
\[
 \langle\mathcal A_{\bu}\bv,\bv\rangle
       \ge c_0(1-q_{0,\Omega})\|\bv\|_2^2.
\]
Indeed the possibly signed term is bounded in absolute value by $\Lambda L_\mu\|\bu\|_2\|\Jop_\Omega\bv\|_\infty\|\bv\|_2$. The bilinear form is bounded on $\V_\Omega$. The nonsymmetric Lax--Milgram theorem therefore supplies its $L^2$ inverse. For $\br\in\mathcal X_\Omega^m$, its solution satisfies a pressure equation of the form \eqref{eq:bounded-difference-equation}, with $\bc_{12}$ replaced by $\mu'(\vartheta)\bM_0\bu\in H^s$. The same conormal induction proves
\begin{equation*}
 \|\mathcal A_{\bu}^{-1}\br\|_{H^m}\le C_R\|\br\|_{H^m}.
\end{equation*}
Thus $\mathcal A_{\bu}$ is an isomorphism on $\mathcal X_\Omega^m$. The implicit function theorem in these specified spaces yields the local $C^1$ solution map $\Sop_\Omega$ and
\[
D\Sop_\Omega(\bh)\br=\mathcal A_{\bu}^{-1}\PP_\Omega\br,
 \quad \br\in\mathcal Y_\Omega^m.
\]
The margin $q_{0,\Omega}\le1/2$ leaves a neighborhood of each base forcing on which the coercivity remains positive. The heat solution map has derivative $\Jop_\Omega D\Sop_\Omega(\bh)$ and the pressure derivative follows from the gradient equation with zero mean. These derivatives give the time regularity required by the transport construction.
\end{proof}

The instantaneous solution map now closes the density transport iteration and gives the local solution asserted in \cref{thm:bounded-local}.

\begin{proof}[Proof of \cref{thm:bounded-local}]
We denote by $\Sop_\Omega(\bh)$ the forcing-to-velocity map constructed in \cref{prop:bounded-instantaneous}. In the transport iteration use $\bv^n=\Sop_\Omega(\rho^n \bfv)$ and solve $\partial_t\rho^{n+1}+\bv^n\cdot\nabla\rho^{n+1}=0$ with $\rho^{n+1}(0)=\rho_0$, starting from $\rho^0(t)=\rho_0$. On a density ball of radius $R$, the bound $\|\rho^n\bfv\|_{H^s}\le C_{\bfv}R$ allows us to apply \cref{prop:bounded-instantaneous} with forcing radius $C_{\bfv}R$. The transport iteration then follows the proof of \cref{thm:local} with $\Sop_\Omega$ evaluated at the force. The velocities are Lipschitz and tangent to $\partial\Omega$, so their characteristic flows preserve $\overline\Omega$ and Lebesgue measure. At the highest derivative level, the transport estimates follow from the characteristic representation for tangent solenoidal fields. Indeed, if a Lipschitz tangent solenoidal field $\bv$ has flow $X(t;s,\bx)$, then this flow preserves $\Omega$ and Lebesgue measure. For $z_0\in L^2(\Omega)$ and $f_z\in L^1(0,T;L^2(\Omega))$, the formula
\[
z(t,X(t;0,\bx)) =z_0(\bx)+\int_0^t f_z(a,X(a;0,\bx))\dd a
\]
defines the unique $L^2$ transport solution, belongs to $C_tL^2$ and satisfies
\[
 \frac12\|z(t)\|_2^2-\frac12\|z_0\|_2^2
=\int_0^t\int_\Omega f_z z\dx\dd a.
\]
These assertions follow first for smooth scalar data along the fixed flow and then by $L^2$ approximation and measure preservation. Applied to the differentiated transport equation with $z=\partial^\beta\rho$ and $f_z=-[\partial^\beta,\bv\cdot\nabla]\rho$, this gives the $H^s$ estimate. At the highest order, testing $\partial^\beta\rho(t)$ against $C_c^\infty(\Omega)$ and using $\rho(t)\to\rho_0$ in $H^{s-1}$ identifies the weak initial trace as $\partial^\beta\rho_0$, by the density argument in the proof of \cref{thm:local}. The characteristic representation with this datum then gives $\partial^\beta\rho\in C_tL^2$ and hence $\rho\in C_tH^s$. The Sobolev product and difference estimates therefore give the same short-time iteration and contraction in $C_tH^{s-1}$ as in \cref{thm:local}. The locally $C^1$ instantaneous map gives the claimed time regularity of $\bu,\vartheta$ and pressure follows from \eqref{eq:bounded-pressure}.
\end{proof}

\section{Bounded-domain correctors and quantitative homogenization}\label{sec:bounded-quantitative}
With the bounded macroscopic solution and the common elliptic estimates available, it remains to adapt the periodic comparison field to the boundary $\partial\Omega$. This section constructs an admissible boundary-corrected lifting and then repeats the relative-energy argument with the additional boundary-layer errors made explicit.

\subsection{The boundary-corrected Stokes lifting}
\label{sec:bounded-cutoff}\label{sec:bounded-lift-time}
The embedding $H^s(\Omega)\hookrightarrow W^{2,\infty}(\Omega)$ for $s\ge4$ gives the spatial regularity required for the boundary correction. To match the microscopic outer boundary condition, a solenoidal correction supported in a thin strip removes the tangential trace of the Darcy velocity.

\begin{lemma}[Boundary trace removal]\label{lem:bounded-cutoff}
For $\bv\in W^{2,\infty}(\Omega;\R^3)\cap\V_\Omega$ and $0<\ell<\ell_0$, there is a linear boundary correction $\bb^\ell[\bv]$ supported in $S_\ell$ such that
\begin{equation}\label{eq:bounded-cutoff-properties}
 \begin{gathered}
 \Div \bb^\ell[\bv]=0,\quad \bb^\ell[\bv]|_{\partial\Omega}=\bv,
 \quad \bv^\ell:=\bv-\bb^\ell[\bv]\in H^1_0(\Omega)\cap\V_\Omega,\\
 \|\bb^\ell[\bv]\|_\infty+\ell\|\nabla \bb^\ell[\bv]\|_\infty\le C_{\bv},
 \quad
 \|\bb^\ell[\bv]\|_2\le C_{\bv}\ell^{1/2},\quad
 \|\nabla \bb^\ell[\bv]\|_2\le C_{\bv}\ell^{-1/2}.
 \end{gathered}
\end{equation}
Here $C_{\bv}\le C_\Omega\|\bv\|_{W^{2,\infty}}$.
\end{lemma}
\begin{proof}
Let $d(\bx)=\operatorname{dist}(\bx,\partial\Omega)$ and let $\pi(\bx)$ be the nearest boundary point in a fixed tubular neighborhood. Then $\nabla d=-\bn\circ\pi$. In that neighborhood define
\[
 \boldsymbol A_{\bv}(\bx)
=d(\bx)\,\bn(\pi(\bx))\times\bv(\pi(\bx)).
\]
We multiply $\boldsymbol A_{\bv}$ by a fixed cutoff that equals one in a smaller neighborhood and extend the product by zero to $\Omega$. The $C^{3,\beta}$ boundary and the trace of $\bv\in W^{2,\infty}$ give
\[
 \|\boldsymbol A_{\bv}\|_{W^{2,\infty}(\Omega)}
 \le C_\Omega\|\bv\|_{W^{2,\infty}(\Omega)},
 \qquad \boldsymbol A_{\bv}|_{\partial\Omega}=0.
\]
The construction is linear. On the boundary, $\operatorname{curl}\boldsymbol A_{\bv} =-\bn\times(\bn\times\bv)=\bv$ since $\bv\cdot\bn=0$. All other product-rule terms contain the factor $d$. This directly gives the bounded linear boundary potential, in the spirit of
\cite[Lemma~3.8]{HLO}.

Reduce $\ell_0$ to lie in the smaller tubular neighborhood. Choose a scalar cutoff $\xi_\ell(d)$ which equals one for $d\le\ell/2$ and zero for $d\ge\ell$, with derivatives bounded by $C\ell^{-1}$ and $C\ell^{-2}$. Set
\[
 \bb^\ell[\bv]
=\operatorname{curl}(\xi_\ell(d)\boldsymbol A_{\bv}).
\]
Its divergence is zero and its boundary trace is $\bv$. Since $|\boldsymbol A_{\bv}|\le C_{\bv}d$, $|\nabla\boldsymbol A_{\bv}|+|\nabla^2\boldsymbol A_{\bv}|\le C_{\bv}$, we have $|\bb^\ell[\bv]|\le C_{\bv}$ and $|\nabla\bb^\ell[\bv]|\le C_{\bv}\ell^{-1}$ on $S_\ell$. The measure bound $|S_\ell|\le C_\Omega\ell$ proves \eqref{eq:bounded-cutoff-properties}.
\end{proof}

For $C\eps\le\ell<\ell_0$, define
\begin{equation}\label{eq:bounded-comparison}
 \mathcal C_{\eps,\ell}\bv
      =\bW_\eps(\bv-\bb^\ell[\bv])
              -\mathfrak B_\eps(\bv-\bb^\ell[\bv]).
\end{equation}
Here $\bb^\ell[\bv]$ is the boundary correction from \cref{lem:bounded-cutoff}, while $\mathfrak B_\eps$ is the cellwise divergence repair from \cref{lem:bounded-local-repair}. After these two corrections, the resulting field belongs to $\V_{\eps,\Omega}$. The localized estimates give
\begin{equation*}
 \|\mathfrak d_\eps(\bv-\bb^\ell[\bv])\|_2
       \le C_{\bv}\tau_\eps(1+\ell^{-1/2}).
\end{equation*}
The first estimate in \eqref{eq:bounded-local-cell} and the cutoff bounds give
\[
 \begin{aligned}
 \|\mathfrak d_\eps(\bv)\|_2
 &\le C\tau_\eps\|\nabla\bv\|_\infty,\\
 \|\mathfrak d_\eps(\bb^\ell[\bv])\|_2
 &\le\|\bW_\eps-\bI\|_{L^2(S_\ell)}
       \|\nabla\bb^\ell[\bv]\|_\infty
 \le C_{\bv}\tau_\eps\ell^{-1/2}.
 \end{aligned}
\]
The cells intersecting $S_\ell$ lie in $S_{\ell+C\eps}$, whose volume is $O(\ell)$ since $\ell\ge C\eps$. Combining \eqref{eq:bounded-local-bog}, \eqref{eq:bounded-local-cell} and \eqref{eq:bounded-cutoff-properties} gives
\begin{align}
 \|\widetilde{\mathcal C_{\eps,\ell}\bv}-\bv\|_2
     &\le C_{\bv}(\tau_\eps+\ell^{1/2}),
            \label{eq:bounded-comparison-L2}\\
 \sig\|\nabla(\mathcal C_{\eps,\ell}\bv-\bW_\eps\bv)\|_2
     &\le C_{\bv}\bigl(\ell^{1/2}+\sig\ell^{-1/2}\bigr).
            \label{eq:bounded-comparison-energy}
\end{align}
The underlying gradient bound is
\[
 \|\nabla(\bW_\eps \bb^\ell[\bv])\|_2
+\|\nabla\mathfrak B_\eps(\bv-\bb^\ell[\bv])\|_2
 \le C_{\bv}\left(\ell^{1/2}\sig^{-1}+\ell^{-1/2}
+\tau_\eps(1+\ell^{-1/2})\right).
\]
The $L^2$ repair contributes at most $C_{\bv}\eps\tau_\eps(1+\ell^{-1/2})\le C_{\bv}\ell^{1/2}$, since $\eps\tau_\eps=a_\eps\le\eps\le C\ell$ and $\ell<1$. The factor $\eps$ in \eqref{eq:bounded-local-bog} keeps the $L^2$ divergence correction within the boundary-strip error.

The field $\mathcal C_{\eps,\ell}\bv$ is compared with the exact lifting $\bw_\eps[m,\bv]$ defined in \eqref{eq:bounded-lift}, for $m\in W^{1,\infty}(\Omega)$ with $0<m_*\le m\le m^*$. The untruncated field $\bW_\eps\bv$ has zero normal trace on the wall and may retain a tangential trace. We first estimate its stress against zero-trace tests and then bound the boundary correction.

\begin{lemma}[Untruncated weighted residual with a tangent trace]
\label{lem:bounded-untruncated}
Let $m\in W^{1,\infty}(\Omega)$ satisfy $0<m_*\le m\le m^*<\infty$ and let $\bU=\bW_\eps\bv$, with $\bv\in W^{1,\infty}(\Omega;\R^3)\cap\V_\Omega$. For $\bphi\in\V_{\eps,\Omega}$,
\begin{equation}\label{eq:bounded-untruncated-residual}
 \left|2\sig^2\int_{\Oe}m\sym \bU:\sym\bphi\dx
              -\int_{\Oe}m\bM_0\bv\cdot\bphi\dx\right|
 \le C\left(\frac\eps\sig+\sig\right)
                                 \sig\|\sym\bphi\|_2.
\end{equation}
The constant depends only on the fixed geometry, $m_*,m^*$, $\|m\|_{W^{1,\infty}}$ and $\|\bv\|_{W^{1,\infty}}$.
\end{lemma}
\begin{proof}
For a smooth zero-trace test $\bphi$, the transpose-term calculation in \cref{lem:weighted-residual} produces the additional boundary integral
\[
 \sig^2\int_{\partial\Omega}
m\sum_{i,j}n_iU_j\partial_j\phi_i\,\,\mathrm dS.
\]
The tangential derivatives of $\bphi$ vanish on the wall, so $\partial_j\phi_i=n_j\partial_n\phi_i$. The displayed term is therefore $\sig^2\int_{\partial\Omega}m(\bU\cdot\bn) (\bn\cdot\partial_n\bphi)\,\,\mathrm dS=0$, since $\bW_\eps=\bI$ near the wall and $\bv\cdot\bn=0$. On the hole boundaries $\bU=0$. Thus \eqref{eq:weighted-transpose} holds for smooth zero-trace tests. For a general $\bphi\in\V_{\eps,\Omega}$, the spatial approximation in \cref{lem:finite-boundary-density} and continuity of the volume pairings in $H^1$ extend this identity to $\bphi$.

For the remaining calculation set $\mathcal R_{\eps,\Omega}^i=\sig^2(-\Delta\bW_\eps^i+\nabla Q_\eps^i)-\bM_0\be_i$, put $\omega_i=mv_i$ and let $Z=\sig\|\sym\bphi\|_2$. The bounded-domain Poincar\'e inequality gives $\|\bphi\|_2\le CZ$ and $\|\nabla\bphi\|_2\le C\sig^{-1}Z$. Since $\omega_i\bphi\in H^1_0(\Oe)$, \eqref{eq:bounded-cell-residual} and the corrector bounds give
\begin{align*}
 |\langle\mathcal R_{\eps,\Omega}^i,\omega_i\bphi\rangle|
 &\le C\eps\|\omega_i\bphi\|_{H^1}
 \le C\frac{\eps}{\sig}Z,\\
 \sig^2\left|\int_{\Oe}\nabla\bW_\eps^i:(\bphi\otimes\nabla\omega_i)\dx\right|
 &\le C\sig^2\|\nabla\bW_\eps^i\|_2\|\nabla\omega_i\|_\infty\|\bphi\|_2
 \le C\sig Z,\\
 \sig^2\left|\int_{\Oe}Q_\eps^i\nabla\omega_i\cdot\bphi\dx\right|
 &\le C\sig^2\|Q_\eps^i\|_2\|\nabla\omega_i\|_\infty\|\bphi\|_2
 \le C\sig Z.
\end{align*}
The product contribution satisfies $\sig^2|\int m\bW_\eps\nabla\bv:\nabla\bphi|\le C\sig Z$. The coefficient-transpose term obeys the same bound because $\|\bW_\eps\bv\|_2\le C$. Combining these bounds proves \eqref{eq:bounded-untruncated-residual}.
\end{proof}

The untruncated stress estimate, together with the boundary-correction cost, controls the exact lifting. The choice of strip thickness is made next and yields the bounded-domain approximation rate.

\begin{proposition}[Bounded-domain frozen Stokes rate]
\label{prop:bounded-stokes-rate}
Let $\bv\in W^{2,\infty}(\Omega;\R^3)\cap\V_\Omega$ and let $m\in W^{1,\infty}(\Omega)$ satisfy $0<m_*\le m\le m^*<\infty$. Then
\begin{equation}\label{eq:bounded-stokes-rate}
 \|\widetilde{\bw}_\eps[m,\bv]-\bv\|_{L^2(\Omega)}
     \le C\left(\frac\eps\sig+\sig^{1/2}\right).
\end{equation}
More precisely, for $C\eps\le\ell<\ell_0$,
\begin{equation}\label{eq:bounded-stokes-unoptimized}
 \|\widetilde{\bw}_\eps[m,\bv]-\bv\|_2
     \le C\left(\frac\eps\sig+\sig+
                       \ell^{1/2}+\sig\ell^{-1/2}\right).
\end{equation}
The constants depend only on the fixed geometry, coefficient bounds and the stated norms of $m,\bv$, uniformly for the admissible $\eps$ and $\ell$.
\end{proposition}
\begin{proof}
The comparison field is admissible by \eqref{eq:bounded-comparison}. Separate its stress into that of $\bU=\bW_\eps\bv$ and that of $\mathcal C_{\eps,\ell}\bv-\bU$. The first has residual \eqref{eq:bounded-untruncated-residual}, and the second satisfies
\[
2\sig^2\left|\int_{\Oe}m
       \sym(\mathcal C_{\eps,\ell}\bv-\bU):\sym\bphi\dx\right|
 \le C\bigl(\ell^{1/2}+\sig\ell^{-1/2}\bigr)
                                      \sig\|\sym\bphi\|_2
\]
by \eqref{eq:bounded-comparison-energy}. Subtract this comparison relation from \eqref{eq:bounded-lift} and test with $\bphi=\bw_\eps[m,\bv]-\mathcal C_{\eps,\ell}\bv$. Coercivity and \eqref{eq:bounded-poincare} give
\[
 \|\bw_\eps[m,\bv]-\mathcal C_{\eps,\ell}\bv\|_2
 \le C\left(\frac\eps\sig+\sig+\ell^{1/2}+\sig\ell^{-1/2}\right).
\]
The triangle inequality and \eqref{eq:bounded-comparison-L2} then imply \eqref{eq:bounded-stokes-unoptimized}, since $\tau_\eps=(\eps/\sig)^2\le\eps/\sig$ for sufficiently small $\eps$.

The boundary contribution to the squared estimate is $O(\ell+\sig^2/\ell)$. Choose $\ell=\sig$, which is admissible for all sufficiently small $\eps$, since $\eps/\sig=\eps^{(\alpha-1)/2}\to0$ and $\sig\to0$. Substitution yields \eqref{eq:bounded-stokes-rate}.
\end{proof}

\begin{lemma}[Stokes approximation with zero full boundary trace]
\label{lem:bounded-zero-trace}
Under the coefficient and geometry hypotheses of \cref{prop:bounded-stokes-rate}, let $\bv\in W^{2,\infty}(\Omega;\R^3)\cap\V_\Omega$ satisfy $\bv|_{\partial\Omega}=0$. Then
\[
 \|\widetilde{\bw}_\eps[m,\bv]-\bv\|_2
 \le C\left(\frac\eps\sig+\sig\right),
\]
where $C$ is independent of $\eps$ and depends only on the fixed geometry, coefficient bounds and the stated norms of $m,\bv$.
\end{lemma}
\begin{proof}
Set $\mathcal C_\eps\bv=\bW_\eps\bv-\mathfrak B_\eps(\bv)$. The cell fields equal the identity near $\partial\Omega$, so the full zero trace of $\bv$ gives the outer zero trace of $\mathcal C_\eps\bv$. Its hole trace is zero, and \cref{lem:bounded-local-repair} cancels its divergence. Hence $\mathcal C_\eps\bv\in\V_{\eps,\Omega}$. The geometric estimates in \cref{lem:bounded-tools} and the local repair estimate give
\[
 \|\widetilde{\mathcal C_\eps\bv}-\bv\|_2\le C\tau_\eps,
 \qquad \sig\|\nabla\mathfrak B_\eps(\bv)\|_2\le C\tau_\eps.
\]
The untruncated residual in \cref{lem:bounded-untruncated}, together with the repair term, implies
\[
 \left|2\sig^2\int_{\Oe}m\sym(\mathcal C_\eps\bv):\sym\bphi\dx
       -\int_{\Oe}m\bM_0\bv\cdot\bphi\dx\right|
 \le C\left(\frac\eps\sig+\sig+\tau_\eps\right)\sig\|\sym\bphi\|_2
\]
for $\bphi\in\V_{\eps,\Omega}$. Subtract this relation from \eqref{eq:bounded-lift} and take $\bphi=\bw_\eps[m,\bv]-\mathcal C_\eps\bv$. Coercivity and \eqref{eq:bounded-poincare} give
\[
 \|\bw_\eps[m,\bv]-\mathcal C_\eps\bv\|_2
 \le C\left(\frac\eps\sig+\sig+\tau_\eps\right).
\]
The triangle inequality combines this bound with $\|\widetilde{\mathcal C_\eps\bv}-\bv\|_2\le C\tau_\eps$. Since $\tau_\eps=(\eps/\sig)^2\le\eps/\sig$ for sufficiently small $\eps$, this proves the asserted estimate.
\end{proof}

The approximation estimate concerns the spatial error. For the relative-energy argument we also need bounds on the exact lifting and its time derivative, which follow directly from its variational equation.

Consider an admissible reference solution $(\rho,\bu,\vartheta,p)$ in the sense of \cref{def:bounded-reference}. Its coupled energy identity gives $\sup_t\|\bu(t)\|_2\le H_{0,\Omega}/c_0$. Put $m=\mu(\vartheta)$ and $\bw_\eps(t)=\bw_\eps[m(t),\bu(t)]$. Testing \eqref{eq:bounded-lift} with $\bw_\eps$ and using \eqref{eq:bounded-poincare} yields
\begin{equation}\label{eq:bounded-lift-bounds}
\left\{
\begin{aligned}
& \sig\|\sym \bw_\eps(t)\|_2\le W_{0,\Omega},\quad  \|\partial_t\bw_\eps\|_2+\sig\|\nabla\partial_t\bw_\eps\|_2 \le C_T, \\
&\|\bw_\eps\|_2+\sig\|\nabla \bw_\eps\|_2+\sig^{1/2}\|\bw_\eps\|_3\le C_T.
\end{aligned}
\right.
\end{equation}
The spatial part of the last line follows since zero extension to $H^1_0(\Omega)$ and Sobolev embedding give $\|\bw_\eps\|_6\le C\|\nabla\bw_\eps\|_2\le C\sig^{-1}$, while the bounded-domain Poincar\'e inequality and interpolation yield $\|\bw_\eps\|_3\le C\sig^{-1/2}$. The time-derivative estimate is obtained separately by time difference quotients in the variational lifting, exactly as in \eqref{eq:lift-time}. It uses $\partial_t\bu\in L^\infty_tL^2_x$ and $\partial_t\vartheta\in L^\infty_tL^\infty_x$. For fixed $\eps$, the operator inverse argument gives
\[
 \bw_\eps\in W^{1,\infty}(0,T;\V_{\eps,\Omega})
       \subset C([0,T];\V_{\eps,\Omega}),
\]
which is sufficient for the comparison-test argument in \cref{lem:finite-boundary-density}. The constants in \eqref{eq:bounded-lift-bounds} are controlled by the fixed data and $\mathcal N_{T,\Omega}$. For the spatial error, \cref{prop:bounded-stokes-rate} gives, uniformly for almost every $t\in(0,T)$,
\[
 \|\widetilde{\bw}_\eps(t)-\bu(t)\|_2^2
 \le C_T\left(\frac{\eps^2}{\sig^2}+\sig\right)
 =C_T\left(\eps^{\alpha-1}+\eps^{(3-\alpha)/2}\right).
\]
Integrating in time gives
\begin{equation}\label{eq:bounded-eta}
 \eta_{\eps,\Omega}^2\le C_T\left(\eps^{\alpha-1}+\eps^{(3-\alpha)/2}\right).
\end{equation}

\subsection{Relative energy and proof of the bounded-domain limit}
\label{sec:bounded-homogenization}
We combine \eqref{eq:bounded-lift-bounds} and \eqref{eq:bounded-temp-error} in the relative-energy comparison, retaining the transport and momentum boundary terms.

\begin{proof}[Proof of \cref{thm:bounded-homogenization}]
Use the error fields $\be_\eps$, $\br_\eps$, $d_\eps$ and $z_\eps$ introduced in the periodic argument, now with ambient domain $\mathcal O=\Omega$ and zero extension of $\be_\eps$. The microscopic velocity and the comparison velocity, $\bu_\eps$ and $\bw_\eps$, have zero trace on all of $\partial\Oe$. The macroscopic $\bu$ need not have zero tangential trace. Define
\[
 \E_{\eps,\Omega}(t)=\frac{\sig^4}{2}\int_{\Oe}\rho_\eps|\be_\eps|^2\dx
+\frac12\|d_\eps\|_2^2.
\]
At fixed $\eps$, \cref{lem:finite-boundary-density} is applied on $D=\Oe$. The regularity in \cref{sec:bounded-lift-time} makes $\bw_\eps$ an admissible momentum comparison. The same lemma makes its zero-extended square an admissible scalar transport test. As in the proof of \cref{lem:finite-boundary-density}, scalar mollification is performed on $\mathbb T_L^3$, with $L$ chosen to contain $\overline\Omega$ in the interior of one coordinate cube, and the tests are restricted to $\Omega$. Thus the comparison tests require only the finite-boundary approximation furnished by \cref{lem:finite-boundary-density} on the $C^{3,\beta}$ wall.

We first compare momentum and density. The momentum test with $\bw_\eps$ and the continuity test with $|\widetilde{\bw}_\eps|^2/2$ give \eqref{eq:kinetic-prelift} with bounded-domain integrals and the same pre-lifting cross-stress and inertial terms. Splitting $\mu(\vartheta_\eps)$ and then applying the bounded lifting equation therefore gives \eqref{eq:kinetic-error} with bounded-domain integrals. The pressure pairing vanishes because
\[
 \int_{\Oe}\nabla p\cdot \be_\eps\dx
=\int_\Omega\nabla p\cdot\widetilde \be_\eps\dx=0.
\]
The argument uses only the velocity trace at the outer boundary. The extended density satisfies
\[
 \partial_td_\eps+\widetilde \bu_\eps\cdot\nabla d_\eps
=-(\be_\eps+\br_\eps)\cdot\nabla\rho.
\]
The squared-error identity follows by expanding
$\|\widehat\rho_\eps-\rho\|_2^2$, using renormalization with $b(r)=r^2$
and testing the microscopic continuity equation with the reference density,
as in the periodic proof. Here $\rho\in L^\infty_tW^{1,\infty}_x$ and
$\partial_t\rho\in L^\infty_tL^2_x$ suffice for scalar-test approximation
on the fixed domain. The transport velocities have zero normal flux, and
$\int_\Omega\rho(\widetilde\bu_\eps-\bu)\cdot\nabla\rho\dx=0$
by the definition of $\V_\Omega$. Thus no density boundary term remains.
For any fixed small $\delta>0$, this gives
\[
 \frac12\frac{d}{dt}\|d_\eps\|_2^2
       \le\delta X_\eps^2+C_\delta\|d_\eps\|_2^2
+C_\delta\|\br_\eps\|_2^2.
\]

The temperature equation supplies the thermal dissipation. Testing \eqref{eq:bounded-heat-difference} with $z_\eps$ gives
\[
 \int_\Omega\bar\kappa_\eps|\nabla z_\eps|^2\dx
=\int_\Omega(\be_\eps+\br_\eps)\cdot \bg\,z_\eps\dx -\int_\Omega(\bar\kappa_\eps-\kappa_f)
                           \nabla\vartheta\cdot\nabla z_\eps\dx.
\]
The term with $\be_\eps$ cancels the thermal work in the relative kinetic-energy inequality. The insulated boundary condition makes the heat-flux boundary integral vanish. Mean-zero Poincar\'e controls $\|z_\eps\|_{H^1}$ by its gradient. The other two terms are bounded by a small part of thermal dissipation plus $C(\|\br_\eps\|_2^2+|\kappa_s-\kappa_f|^2h_\eps)$.

For the viscosity residual in \eqref{eq:viscosity-residual}, estimates \eqref{eq:bounded-temp-error} and \eqref{eq:bounded-lift-bounds} give
\begin{align*}
 |R_{\mu,\eps}|
  &\le\mu_*\chi_{0,\Omega}X_\eps^2
       +2L_\mu W_{0,\Omega}X_\eps
           \left(C_{g,\Omega}\|\br_\eps\|_2
                +C_{T,\Omega,q}|\kappa_s-\kappa_f|h_\eps^{1/q}\right)\\
  &\le\mu_*X_\eps^2
       +C_{T,\Omega,q}\left(\|\br_\eps\|_2^2
                       +|\kappa_s-\kappa_f|^2h_\eps^{2/q}\right).
\end{align*}
The last step uses $\chi_{0,\Omega}\le1/2$. As in the periodic case, the coefficient part consumes at most $(\mu_*/2)X_\eps^2$, while Young's inequality consumes at most another $(\mu_*/2)X_\eps^2$. At least $\mu_*X_\eps^2$ therefore remains from the lower bound $2\mu_*X_\eps^2$.

It remains to estimate the inertial remainders. Zero extension gives the fixed-domain bound $\|\be_\eps\|_6\le C\|\nabla\be_\eps\|_2$. Consequently, \eqref{eq:bounded-poincare} yields $\|\be_\eps\|_4^2\le C\sig^{1/2}D_\eps^2$ for $D_\eps=\|\nabla\be_\eps\|_2$. The lifting estimates in \eqref{eq:bounded-lift-bounds} give $\|\bw_\eps\|_2\le C$, $\|\nabla\bw_\eps\|_2\le C\sig^{-1}$ and $\|\partial_t\bw_\eps\|_2\le C$. Interpolation then yields $\|\bw_\eps\|_4\le C\sig^{-3/4}$. Thus
\begin{align*}
 &\sig^4\left|\int_{\Oe}\rho_\eps\be_\eps\cdot(\be_\eps\cdot\nabla)\bw_\eps\dx\right|
 \le C_T\sig^4\|\be_\eps\|_4^2\|\nabla\bw_\eps\|_2
 \le C_T\sig^{7/2}D_\eps^2,\\
 &\sig^4\left|\int_{\Oe}\rho_\eps\be_\eps\cdot(\bw_\eps\cdot\nabla)\bw_\eps\dx\right|
 \le C_T\sig^4\|\be_\eps\|_4\|\bw_\eps\|_4\|\nabla\bw_\eps\|_2
 \le C_T\sig^{5/2}D_\eps\le\delta\sig^2D_\eps^2+C_{T,\delta}\sig^3,\\
& \sig^4\left|\int_{\Oe}\rho_\eps\be_\eps\cdot\partial_t\bw_\eps\dx\right|
 \le C_T\sig^4\|\be_\eps\|_2\|\partial_t\bw_\eps\|_2
 \le C_T\sig^5D_\eps\le\delta\sig^2D_\eps^2+C_{T,\delta}\sig^8.
\end{align*}
Since $\be_\eps\in H^1_0(\Oe;\R^3)$ is solenoidal, integration by parts gives $D_\eps^2=2\|\sym\be_\eps\|_2^2$. Choose the Young parameters $\delta$ so that their total viscous contribution uses at most $(\mu_*/4)X_\eps^2$. Since $\sig\to0$ as $\eps\to0$, we can then choose $\eps$ small enough that $2C_T\sig^{3/2}\le\mu_*/4$, and hence
\[
 C_T\sig^{7/2}D_\eps^2\le\frac{\mu_*}{4}X_\eps^2.
\]
At least $(\mu_*/2)X_\eps^2$ remains after these absorptions. Therefore
\begin{align*}
 \E_{\eps,\Omega}(\tau)+c\int_0^\tau X_\eps^2\dt
             +c\int_0^\tau\|z_\eps\|_{H^1}^2\dt
  &\le \E_{\eps,\Omega}(0)+C\int_0^\tau \E_{\eps,\Omega}(t)\dt\\
  &\quad+C_{T,\Omega,q}\left(\eta_{\eps,\Omega}^2
           +|\kappa_s-\kappa_f|^2h_\eps^{2/q}+\sig^3\right).
\end{align*}
To express the initial term through $\mathcal D_{\eps,\Omega}^0$, the lifting bounds give $\|\bw_\eps(0)-\bu(0)\|_{L^2(\Oe)}\le C_T$. Since $\rho_\eps^0\le\rho^*$, the decomposition
\[
 \bu_\eps^0-\bw_\eps(0)
 =\bu_\eps^0-\bu(0)+\bu(0)-\bw_\eps(0)
\]
and $|\boldsymbol a+\boldsymbol b|^2\le2|\boldsymbol a|^2+2|\boldsymbol b|^2$ imply
\[
 \E_{\eps,\Omega}(0)\le2\mathcal D_{\eps,\Omega}^0+C_T\sig^4,
 \qquad
 \mathcal D_{\eps,\Omega}^0\le2\E_{\eps,\Omega}(0)+C_T\sig^4.
\]
The density terms coincide in these two quantities. Gronwall and \eqref{eq:bounded-temp-error} now give \eqref{eq:bounded-stability}, with the cost $C_T\sig^4$ absorbed by $C_T\sig^3$.

Finally take $q=4$ and use \eqref{eq:bounded-count} and \eqref{eq:bounded-eta}. Then $h_\eps^{1/2}\le C\eps^{3(\alpha-1)/2} \le C\eps^{\alpha-1}$ and $\sig^3\le\sig=\eps^{(3-\alpha)/2}$. This proves \eqref{eq:bounded-final-rate}. Lipschitz continuity of $\mu$ gives the viscosity convergence. All smallness thresholds and constants refer to the fixed macroscopic strong-solution interval.
\end{proof}

\begin{proof}[Proof of \cref{cor:bounded-rates}]
The first assertion follows from \eqref{eq:bounded-final-rate} by taking the minimum of the three squared-error exponents and then taking square roots for the individual norms. If $\bu|_{\partial\Omega}=0$, \cref{lem:bounded-zero-trace} gives $\eta_{\eps,\Omega}^2\le C_T(\eps^{\alpha-1}+\eps^{3-\alpha})$. In \eqref{eq:bounded-stability} with $q=4$, $h_\eps^{1/2}\le C\eps^{3(\alpha-1)/2}\le C\eps^{\alpha-1}$ and $\sig^3\le\sig^2=\eps^{3-\alpha}$ for small $\eps$. Substitution gives the trace-compatible squared-error estimate.
\end{proof}

\appendix
\setcounter{section}{0}
\renewcommand{\thesection}{A.\arabic{section}}
\crefalias{section}{appendix}
\section{Periodic cell correctors and zero-hole-trace duality}
\label{app:periodic-geometry}
In this appendix $\Oe$ denotes the periodic geometry in \eqref{eq:geometry}. Let $D=\Oe$. We record the periodic construction and the testing conventions used in \cref{lem:geometric-tools}. The ambient domains allowed in \cite{HLO} include $\T$. Their resistance matrix $M_0$ is denoted by $\bM_0$ here in \eqref{eq:resistance}.

We begin with the cell construction at the independent hole and cell-spacing scales. For this appendix abbreviate $l=\cellsize$ and $a=a_\eps$. The $N_\eps^3=l^{-3}$ translates $Q_{\eps,k}=lk+l[-1/2,1/2)^3$, indexed modulo $N_\eps$, tile $\T$. Use the rescaled exterior Stokes field in $B(lk,l/8)\setminus T_{\eps,k}$, a homogeneous Stokes matching field on $B(lk,l/4)\setminus\overline{B(lk,l/8)}$ and $\be_i$ outside $B(lk,l/4)$. The two annular velocity traces have compatible flux because the exterior field and $\be_i$ are solenoidal. Extend velocity by zero through the hole. For each matching problem choose the annular pressure to have zero mean on that annulus before assembly. The exterior pressure has its decay normalization and the pressure outside the matching balls is initially zero. Assemble these pieces and then subtract one global constant from each assembled pressure column, as in \cref{lem:geometric-tools}.

All annuli become the same fixed annulus after rescaling by $l$. The exterior far-field bounds and annular Stokes estimates are uniform for $0<a/l\le\tau_0$, with a fixed sufficiently small $0<\tau_0<1$. They give
\[
 \begin{aligned}
& \|\bW_\eps-\bI\|_2^2\le C l^{-3}(a^2l+a^3)\le C(a/l)^2=C\tau_\eps^2,\\
& \|\nabla \bW_\eps\|_2^2+\|\bQ_\eps\|_2^2\le C l^{-3}a=C\sig^{-2}.
 \end{aligned}
\]
The estimate includes the zero-extension error on the holes, whose contribution is $O(l^{-3}a^3)$. This is the two-scale form of the construction in \cite[Proposition~2.7 and Appendix~A]{HLO}.

The cell construction gives a residual supported on the matching surfaces. Its total mass determines the normalized resistance. We specify the traction signs using the nonsymmetric Stokes stress $\mathsf S(\bw,q)=\nabla\bw-q\bI$, for which $-\Div\mathsf S=-\Delta\bw+\nabla q$ under our gradient convention. Fix one column $i$ and one cell $k$. Let $S_1=\partial B(lk,l/8)$ and $S_2=\partial B(lk,l/4)$, with $\bnr=(\bx-lk)/|\bx-lk|$ on both spheres. Thus $\bnr$ points from the inner region to the outer region. On $S_1$ it is opposite to the outward normal of the matching annulus. Write $\mathsf S_E$ for the rescaled exterior stress and $\mathsf S_A$ for the homogeneous matching-annulus stress, with its pressure constant fixed before assembly. Outside $S_2$ the unnormalized stress is zero. Integration by parts in the three regions defines a vector-valued surface measure $\nu_k^i$ with densities $t_{k,1}^i=(\mathsf S_E-\mathsf S_A)\bnr$ on $S_1$ and $t_{k,2}^i=\mathsf S_A \bnr$ on $S_2$. Its pairing with a vector test is
\begin{equation*}
 \langle\nu_k^i,\bzeta\rangle
  =\int_{S_1}(\mathsf S_E-\mathsf S_A)\bnr\cdot\bzeta\,\,\mathrm dS
       +\int_{S_2}\mathsf S_A \bnr\cdot\bzeta\,\,\mathrm dS.
\end{equation*}
There is also a hole-boundary functional, which vanishes for $\bzeta\in\Xe$ because its hole trace is zero. Since $\Div\mathsf S_A=0$ in the matching annulus,
\[
 \int_{S_2}\mathsf S_A \bnr\,\,\mathrm dS
-\int_{S_1}\mathsf S_A \bnr\,\,\mathrm dS=0.
\]
The normal of the exterior fluid on the hole, pointing into the solid, is denoted by $\bnh$. Exterior force balance now gives
\begin{equation}\label{eq:cell-traction-mass}
 \int_{S_1\cup S_2}\dd\nu_k^i
   =\int_{S_1}\mathsf S_E \bnr\,\,\mathrm dS
   =-\int_{\partial T_{\eps,k}}\mathsf S_E \bnh\,\,\mathrm dS
   =a\bM_0\be_i.
\end{equation}
For the last sign, test the reference exterior equation with $\bV^j-\be_j$. Its trace is $-\be_j$ on the hole and it decays at infinity. Hence
\[
(\bM_0)_{ij} =-\be_j\cdot\int_{\partial T_0}\mathsf S(\bV^i,Q^i)\bnh\,\,\mathrm dS.
\]
Scaling the surface traction contributes a factor $a$ and symmetry of $\bM_0$ identifies the vector in \eqref{eq:cell-traction-mass}. The exterior decay and the fixed-annulus estimates also give $\|t_{k,1}^i\|_{L^\infty(S_1)}+\|t_{k,2}^i\|_{L^\infty(S_2)}\le Ca/l^2$. A subsequent common pressure normalization changes neither jump.

For each complete cell, set $\bar\bzeta_k=l^{-3}\int_{Q_{\eps,k}}\widetilde\bzeta\dx$. The zero extension belongs to $H^1$ on the full cell. Scaled trace and Poincar\'e inequalities give
\[
 \|\widetilde\bzeta-\bar\bzeta_k\|_{L^2(S_1\cup S_2)}
        \le Cl^{1/2}\|\nabla\widetilde\bzeta\|_{L^2(Q_{\eps,k})}.
\]
Since $\sig^2a=l^3$, cancellation of the cell-average parts gives
\begin{align*}
 \left|\left\langle
   \sig^2(-\Delta \bW_\eps^i+\nabla Q_\eps^i)-\bM_0\be_i,
       \bzeta\right\rangle\right|&\le\sum_k\sig^2
           |\langle\nu_k^i,\widetilde\bzeta-\bar\bzeta_k\rangle|\\
 &\le C\sum_k l^{5/2}
                 \|\nabla\widetilde\bzeta\|_{L^2(Q_{\eps,k})}
       \le Cl\|\nabla\bzeta\|_{L^2(D)}.
\end{align*}
This proves \eqref{eq:cell-residual}, the two-scale version of the traction calculation in \cite[Appendix~A]{HLO}. Periodic face terms cancel. The residual belongs to $\Xe^*$ and is tested against $\Xe$. A global pressure constant leaves the pairing unchanged because $\int_D\Div\bzeta=0$.

To complete the geometric tools, we construct a periodic inverse of the divergence with the stated scale. For $b\in L^2_\diamond(D)$ extend it by zero and set $\bv=\nabla\Delta_{\T}^{-1}\widetilde b$, where $\Delta_{\T}^{-1}$ is the inverse of the periodic Laplacian on mean-zero scalars. Then $\Div \bv=\widetilde b$ and $\|\bv\|_{H^1(\T)}\le C\|b\|_2$. Choose once and for all a smooth cutoff $\chi$ on the reference ball, equal to one near $\overline{T_0}$ and supported in $B(0,1/2)$ and set $\chi_k(\bx)=\chi((\bx-\cellsize k)/a)$. Thus $B_k=B(\cellsize k,a/2)$ and $|\nabla\chi_k|\le C/a$. Each cell uses a translate and rescaling of the same cutoff and reference Bogovski\u\i{} operator, with indices taken modulo $N_\eps$. The pieces are assembled with these fixed choices. On $E_k=B_k\setminus T_{\eps,k}$ set $h_k=\Div(\chi_k\bv)$. Its integral is zero since the outer trace vanishes and the flux through the hole is zero by $\Div \bv=0$ there. Let $\mathcal B_{E_k}$ be the zero-trace Bogovski\u\i{} operator on this scaled copy of the fixed connected domain $B(0,1/2)\setminus\overline{T_0}$. Define $\bc_k=\chi_k\bv-\mathcal B_{E_k}h_k$ on $E_k$, equal to $\bv$ on the hole and zero outside $B_k$. The traces match across both interfaces, so $\bc_k\in H^1(\T)$ is globally solenoidal since its distributional divergence has no interface measure. Hence $\bv-\bc_k$ has zero trace on that hole, including for cells that meet the periodic identification. Writing $\bar \bv_k=\cellsize^{-3}\int_{Q_{\eps,k}}\bv\dx$, scaled Sobolev--Poincar\'e and the local divergence estimate give
\begin{equation*}
 \|\nabla \bc_k\|_{L^2(B_k)}^2
 \le C\bigl(\|\nabla \bv\|_{L^2(Q_{\eps,k})}^2+a|\bar \bv_k|^2\bigr).
\end{equation*}
Indeed $a^{-1}\|\bv-\bar \bv_k\|_{L^2(B_k)} \le C\|\bv-\bar \bv_k\|_{L^6(Q_{\eps,k})} \le C\|\nabla \bv\|_{L^2(Q_{\eps,k})}$. The constant part costs $a^{1/2}|\bar \bv_k|$. This is the local mechanism used in the construction of the perforated-domain divergence inverse in \cite[Section~3]{DFL}. The disjoint supports and $\sum_k a|\bar \bv_k|^2\le a\cellsize^{-3}\|\bv\|_2^2$ show that the linear operator $\mathcal B_\eps b=(\bv-\sum_k\bc_k)|_D$ satisfies
\[
 \Div\mathcal B_\eps b=b,\quad
 \|\nabla\mathcal B_\eps b\|_2
 \le C\bigl(\|\nabla \bv\|_2+(a\cellsize^{-3})^{1/2}\|\bv\|_2\bigr)
 \le C(1+\sig^{-1})\|b\|_2.
\]
The corrections are supported inside the periodically identified cells. For general $L^2$ data, approximate $b$ by smooth mean-zero functions on $D$. At fixed $\eps$, the estimate gives convergence of the constructed fields in $H^1(D)$. Continuity of divergence and of the trace operator preserves both $\Div\mathcal B_\eps b=b$ and the zero hole traces. The uniform bound comes from the common reference construction. To verify the Poincar\'e scale directly, let $\bar z_k=l^{-3}\int_{Q_{\eps,k}}\widetilde z\dx$. Since $\widetilde z=0$ on the hole, its $L^6$ norm there gives the first inequality in the following chain, while the scale-invariant Sobolev--Poincar\'e inequality on the full cube gives the second.
\[
 a^{1/2}|T_0|^{1/6}|\bar z_k|=\|\widetilde z-\bar z_k\|_{L^6(T_{\eps,k})}
 \le\|\widetilde z-\bar z_k\|_{L^6(Q_{\eps,k})}
 \le C\|\nabla\widetilde z\|_{L^2(Q_{\eps,k})}.
\]
The ordinary cell Poincar\'e inequality then yields
\begin{equation}\label{eq:cell-poincare}
 \|\widetilde z\|_{L^2(Q_{\eps,k})}^2
 \le C\left(l^2+\frac{l^3}{a}\right)
               \|\nabla\widetilde z\|_{L^2(Q_{\eps,k})}^2
 \le C\sig^2\|\nabla\widetilde z\|_{L^2(Q_{\eps,k})}^2,
\end{equation}
because $a<l$ and $l^3/a=\sig^2$. Summing over cells proves \eqref{eq:general-poincare}. This is the $L^2$ cell argument underlying \cite[Lemma~3.4.1]{AllaireII} and \cite[Lemma~2.3]{HLO}. For solenoidal $z$, periodic integration by parts gives $2\|\sym z\|_2^2=\|\nabla z\|_2^2$. This proves the periodic assertions of \cref{lem:geometric-tools}, with the stated scale-dependent divergence bound.

\section{Fixed-parameter microscopic existence in the periodic case}\label{sec:weak-existence}
The periodic Galerkin construction is carried out at fixed $\eps$. Before stating the existence proposition, we record two fixed-domain compactness facts that identify the time traces and the nonlinear kinetic energy. Throughout this section, the constants $C_{\eps,T}$ may depend on the fixed perforated domain but remain independent of the Galerkin dimension. The final energy estimate recovers the weighted bounds that are uniform in $\eps$.

\begin{lemma}[Strong continuity of the density-weighted mass operator]
\label{lem:mass-operator}
Fix $D=\Oe$, $V=\Ve$, $H=\Hvel$ and let $\PP_D:L^2(D;\R^3)\longrightarrow H$ be the orthogonal projection. Suppose $\rho\in C([0,T];L^q(D))$ for $1\le q<\infty$, with $\rho_*\le\rho(t,\bx)\le\rho^*$. For $\bv\in H$ define
\begin{equation*}
 \mathscr A_\rho(t)\bv=\PP_D(\rho(t)\bv).
\end{equation*}
For $t\in[0,T]$, the operator $\mathscr A_\rho(t)$ is a bounded self-adjoint isomorphism of $H$, with
\begin{equation}\label{eq:mass-coercivity}
 \|\mathscr A_\rho(t)\|_{\mathcal L(H,H)}\le\rho^*,\quad
 (\mathscr A_\rho(t)\bv,\bv)_H\ge\rho_*\|\bv\|_H^2,\quad
 \|\mathscr A_\rho(t)^{-1}\|_{\mathcal L(H,H)}\le\rho_*^{-1}.
\end{equation}
The families $\mathscr A_\rho(t)$ and $\mathscr A_\rho(t)^{-1}$ are continuous in the strong operator topology on $H$. Moreover, let $\bmom\in C_{\mathrm{weak}}([0,T];H)$ and define $\bu(t)=\mathscr A_\rho(t)^{-1}\bmom(t)$. Then $\bu\in C_{\mathrm{weak}}([0,T];H)$ and $\rho \bu\in C_{\mathrm{weak}}([0,T];L^2(D))$.
\end{lemma}
\begin{proof}
For $\bv,\bz\in H$ one has $(\mathscr A_\rho(t)\bv,\bz)_H=\int_D\rho(t)\bv\cdot\bz\dx$. The bounds and invertibility in \eqref{eq:mass-coercivity} follow from positivity and Lax--Milgram. Fix $s\in[0,T]$. To prove strong continuity, first let $\bv\in V$. Since $V\hookrightarrow L^6(D)$,
\[
 \|(\mathscr A_\rho(t)-\mathscr A_\rho(s))\bv\|_H
 \le\|\rho(t)-\rho(s)\|_{L^3(D)}\|\bv\|_{L^6(D)}\to0\quad\text{as }t\to s.
\]
For general $\bv\in H$, choose $\bv_j\in V$ with $\bv_j\to \bv$ in $H$ as $j\to\infty$. The replacement error is bounded by $2\rho^*\|\bv-\bv_j\|_H$, uniformly in $s,t$. Letting $t\to s$ and then $j\to\infty$ proves strong operator continuity on $H$. For $\bh\in H$, the inverse identity gives
\begin{equation*}
 \begin{split}
 \|\mathscr A_\rho(t)^{-1}\bh-\mathscr A_\rho(s)^{-1}\bh\|_H
 &\le\rho_*^{-1}\|
 (\mathscr A_\rho(s)-\mathscr A_\rho(t))
                   \mathscr A_\rho(s)^{-1}\bh\|_H\to0\quad\text{as }t\to s.
 \end{split}
\end{equation*}
Weak continuity of $\bmom$ on $[0,T]$ gives $\sup_t\|\bmom(t)\|_H<\infty$. Self-adjointness, strong continuity of the inverse and the identity $(\bu(t),\bv)_H=(\bmom(t),\mathscr A_\rho(t)^{-1}\bv)_H$ imply weak continuity of $\bu$. The inverse bound in \eqref{eq:mass-coercivity} also gives $\sup_t\|\bu(t)\|_H<\infty$.
For a fixed bounded vector field $\bpsi$, continuity of $\rho$ in $L^2(D)$ gives $\rho(t)\bpsi\to\rho(s)\bpsi$ in $L^2(D)$ as $t\to s$. Approximation in $L^2(D)$ and the uniform density bound extend this convergence to every $\bpsi\in L^2(D;\R^3)$. Weak continuity of $\bu$ and strong continuity of $\PP_D(\rho(t)\bpsi)$ now imply continuity of $\int_D\rho(t)\bu(t)\cdot\bpsi\dx=(\bu(t),\PP_D(\rho(t)\bpsi))_H$, proving the assertion for the full momentum.
\end{proof}

The mass-operator lemma supplies the correct weak time traces. To identify the nonlinear kinetic energy in the Galerkin limit, we also need strong compactness of the velocity under density weights. The next lemma provides this fixed-domain step.

\begin{lemma}[Weighted kinetic compactness on a fixed perforated domain]
\label{lem:weighted-kinetic-compactness}
Let $D$ be fixed. Let $H$ be a closed subspace of $L^2(D;\R^3)$ with the inherited inner product and let the Hilbert space $V$ be densely and continuously embedded in $H$. Identify $H$ with a subspace of $V^*$ by the $L^2$ pairing. Let $\PP_D$ be the $L^2$-orthogonal projection onto $H$. For finite-dimensional subspaces $\mathcal X_n\subset V$, let $\PP_n$ be the $L^2$-orthogonal projection onto $\mathcal X_n$. Suppose $0<\rho_*\le\rho_n,\rho\le\rho^*<\infty$, $\rho_n\to\rho$ strongly in $L^2((0,T)\times D)$ and $\bu_n\in L^2(0,T;\mathcal X_n)$ satisfies $\bu_n\rightharpoonup\bu$ in $L^2(0,T;V)$. Set $\bmom_n=\PP_n(\rho_n\bu_n)$ and assume
\[
 \bmom_n\to\PP_D(\rho\bu)
       \quad\text{strongly in }L^2(0,T;V^*)\quad\text{as }n\to\infty.
\]
Then
\begin{equation}\label{eq:weighted-kinetic-lemma-limit}
 \int_0^T\!\!\int_D\rho_n|\bu_n|^2\dx\dt
 \to
 \int_0^T\!\!\int_D\rho|\bu|^2\dx\dt\quad\text{as }n\to\infty,
\end{equation}
and
\[
 \bu_n\to\bu
 \quad\text{strongly in }L^2((0,T)\times D)\quad\text{as }n\to\infty.
\]
\end{lemma}
\begin{proof}
Since $\bu_n\in\operatorname{Ran}\PP_n$, we have
\[
 \int_0^T\!\!\int_D\rho_n|\bu_n|^2\dx\dt
=\int_0^T\langle\bmom_n,\bu_n\rangle\dt.
\]
The difference between this pairing and $\int_0^T\langle\PP_D(\rho\bu),\bu\rangle\dt$ can be written as
\[
 \int_0^T\langle\bmom_n-\PP_D(\rho\bu),\bu_n\rangle\dt
+\int_0^T\langle\PP_D(\rho\bu),\bu_n-\bu\rangle\dt,
\]
the first term tends to zero by strong convergence in $L^2_tV^*$ and the boundedness of $\bu_n$ in $L^2_tV$, while the second tends to zero by weak convergence in $L^2_tV$. Since $\bu\in H$, $\langle\PP_D(\rho\bu),\bu\rangle=\int_D\rho|\bu|^2\dx$, which proves \eqref{eq:weighted-kinetic-lemma-limit}.

For $K>0$, truncation of the fixed field $\bu\in L^2((0,T)\times D)$ gives
\[
 \| (\rho_n-\rho)\bu\|_{L^2((0,T)\times D)}^2
 \le K^2\|\rho_n-\rho\|_{L^2((0,T)\times D)}^2
      +4(\rho^*)^2\int_{\{|\bu|>K\}}|\bu|^2\dx\dt.
\]
First let $n\to\infty$ at fixed $K$, using the strong density convergence, and then let $K\to\infty$. The $L^2$ integrability of $\bu$ makes the last integral tend to zero, so $(\rho_n-\rho)\bu\to0$ strongly in $L^2((0,T)\times D)$ as $n\to\infty$. Pairing $\rho_n\bu\to\rho\bu$ with the weakly convergent sequence $\bu_n$ and with the fixed field $\bu$ yields, as $n\to\infty$,
\[
 \int_0^T\!\!\int_D\rho_n\bu_n\cdot\bu\dx\dt
   \to
 \int_0^T\!\!\int_D\rho|\bu|^2\dx\dt,
 \qquad
 \int_0^T\!\!\int_D\rho_n|\bu|^2\dx\dt
   \to
 \int_0^T\!\!\int_D\rho|\bu|^2\dx\dt.
\]
Expanding the weighted square and using \eqref{eq:weighted-kinetic-lemma-limit} yields $\int_0^T\!\int_D\rho_n|\bu_n-\bu|^2\dx\dt\to0$ as $n\to\infty$. The lower bound $\rho_n\ge\rho_*$ gives the strong $L^2$ convergence.
\end{proof}

The two preceding lemmas provide the time traces and the strong compactness required for the nonlinear kinetic term. We can therefore state the fixed-parameter existence result in the form used by the homogenization argument.

\begin{proposition}[Microscopic existence for a fixed parameter]\label{prop:micro-existence}
Fix $0<\eps<\eps_0$ as in \eqref{eq:continuous-parameter} and let $T>0$. For any admissible initial pair $(\rho_\eps^0,\bu_\eps^0)$ there exists a finite-energy weak solution in the sense of \cref{def:weak} on $(0,T)$. If the initial data satisfy \eqref{eq:initial-weighted-bound}, then the resulting solutions satisfy the weighted bounds in \eqref{eq:uniform} with a constant $C_T$ independent of $\eps$. The constant $C_T$ depends on the supremum in \eqref{eq:initial-weighted-bound}, the common density bounds, the force and the fixed geometric and constitutive data.
\end{proposition}
The proof proceeds by a Galerkin approximation on the fixed fluid domain, followed by compactness for the transported density and the projected momentum.

\begin{proof}[Proof of \cref{prop:micro-existence}]
For this proof, we set $D=\Oe$, $V=\Ve$ and $H=\Hvel$, with $\eps$ fixed. We write $(\rho,\bu,\vartheta)$ for the Galerkin limit on this microscopic domain. Let $\PP_D$ be the $L^2(D)$ projection onto $H$. The Stokes operator associated with the form $\int_D\nabla \bv:\nabla \bz\dx$ on $V$ has a smooth orthonormal eigenbasis $(\ba_j)_{j\ge1}$ of $H$. Denote by $\PP_n$ the $L^2(D)$ projection onto $\mathcal X_n=\operatorname{span}\{\ba_1,\ldots,\ba_n\}$, so that $\PP_n\to \PP_D$ strongly on $L^2(D)$.

We mollify the positive extension of the initial density periodically on $\T$ and restrict the smooth functions $\rho_n^0$ to $D$. Then $\rho_n^0\to\rho_\eps^0$ in every finite $L^q(D)$ and $\rho_*\le\rho_n^0\le\rho^*$. The Lipschitz norms of these functions are taken on $\T$. We set $\bu_n^0=\PP_n\bu_\eps^0$.

These approximations also converge at the level of initial momentum and kinetic energy. For this fixed $\eps$, abbreviate $r_n=\rho_n^0$, $r=\rho_\eps^0$, $z_n=\PP_n\bu_\eps^0$ and $z=\bu_\eps^0\in H\subset L^2(D)$. Then $z_n\to z$ in $L^2(D)$, $r_n\to r$ in $L^2(D)$ and $\rho_*\le r_n,r\le\rho^*$. The decomposition
\[
r_nz_n-rz=r_n(z_n-z)+(r_n-r)z
\]
controls the first term by $\rho^*\|z_n-z\|_2$. For any $R>0$,
\[
 \|(r_n-r)z\|_2
 \le R\|r_n-r\|_2
+2\rho^*\|z\mathbf1_{\{|z|>R\}}\|_2.
\]
Letting $n\to\infty$ at fixed $R$ and then $R\to\infty$ gives
\begin{equation}\label{eq:initial-momentum-product}
 \rho_n^0\bu_n^0\to\rho_\eps^0\bu_\eps^0
 \quad\text{strongly in }L^2(D;\R^3).
\end{equation}
The same truncation argument applies to $(\sqrt{r_n}-\sqrt r)z$, since the square-root coefficients are uniformly bounded and converge in $L^2(D)$. Indeed, $|\sqrt{r_n}-\sqrt r|\le |r_n-r|/(2\sqrt{\rho_*})$. Consequently, we have
\begin{equation}\label{eq:initial-kinetic-product}
 \begin{gathered}
 \sqrt{\rho_n^0}\,\bu_n^0\to
           \sqrt{\rho_\eps^0}\,\bu_\eps^0\quad\text{in }L^2(D),\quad
 \int_D\rho_n^0|\bu_n^0|^2\dx\to
           \int_D\rho_\eps^0|\bu_\eps^0|^2\dx.
 \end{gathered}
\end{equation}
Since $\PP_n$ is an $L^2$ contraction and $\PP_n\to \PP_D$ strongly on $L^2(D)$, \eqref{eq:initial-momentum-product} also gives $\PP_n(\rho_n^0\bu_n^0)\to \PP_D(\rho_\eps^0\bu_\eps^0)$ strongly in $H$. These limits are taken as $n\to\infty$ on the fixed domain $D$.

The approximate density and temperature are now coupled to the Galerkin velocity. For $\bu_n(t)\in\mathcal X_n$, solve the transport equation $\partial_t\rho_n+\bu_n\cdot\nabla\rho_n=0$ with this initial density. The finite-dimensional velocity is Lipschitz in space and has zero boundary trace. Its flow preserves the density bounds. For each time define the mean-zero temperature $\vartheta_n$ by
\[
-\Div(\bar\kappa_\eps\nabla\vartheta_n)=\widetilde \bu_n\cdot \bg
 \quad\hbox{on }\T.
\]
The source has mean zero. The temperature depends linearly and continuously on $\bu_n$, with values in $H^1\cap L^\infty$ by \cref{lem:scalar}.

For each $\ba_j\in\mathcal X_n$ impose
\begin{equation}\label{eq:galerkin}
\begin{aligned}
 &\sig^4\frac{d}{dt}\int_D\rho_n\bu_n\cdot \ba_j\dx
 -\sig^4\int_D\rho_n\bu_n\otimes \bu_n:\nabla \ba_j\dx\\
 &\quad+2\sig^2\int_D\mu(\vartheta_n)\sym \bu_n:\sym \ba_j\dx
 =\int_D(\rho_n\bfv-\vartheta_n\bg)\cdot \ba_j\dx.
\end{aligned}
\end{equation}
The coupled flow--ODE system is locally solvable by contraction. Fix $n$ and $\eps$. Write $\bu_{\bc}(t)=\sum_{j=1}^nc_j(t)\ba_j$ and let $\bc^0$ be the coefficient vector of $\bu_n^0$. We work on
\[
 \mathscr K_{R,\tau}=\{\bc\in C([0,\tau];\R^n):
       \bc(0)=\bc^0,\ \|\bc-\bc^0\|_{C_t\R^n}\le1\},\quad
R=|\bc^0|+1,
\]
equipped with the supremum norm. Finite-dimensional norm equivalence gives
\[
 \|\bu_{\bc}(t)\|_{W^{1,\infty}(D)}\le K_{\eps,n}R,
 \quad
 \|\bu_{\bc}(t)-\bu_{\bd}(t)\|_{W^{1,\infty}(D)}
      \le K_{\eps,n}|\bc(t)-\bd(t)|.
\]
The zero boundary trace makes the characteristic flows stay in $D$. Periodicity and solenoidality make them volume preserving. Let $\Phi_{\bc}(t;s,\bx)$ solve $\partial_t\Phi_{\bc}(t;s,\bx)=\bu_{\bc}(t,\Phi_{\bc}(t;s,\bx))$ with $\Phi_{\bc}(s;s,\bx)=\bx$. Gronwall's inequality gives
\begin{equation*}
 \sup_{s,t\in[0,\tau]}
 \|\Phi_{\bc}(t;s,\cdot)-\Phi_{\bd}(t;s,\cdot)\|_\infty
 \le K_{\eps,n}\tau e^{K_{\eps,n}R\tau}\|\bc-\bd\|_{C_t\R^n}.
\end{equation*}
The transported density associated with $\bc$ is $\rho_{\bc}(t,\bx)=\rho_n^0(\Phi_{\bc}(0;t,\bx))$. In particular,
\begin{equation*}
 \|\rho_{\bc}-\rho_{\bd}\|_{C_tL^\infty(D)}
 \le\|\nabla\rho_n^0\|_\infty K_{\eps,n}\tau e^{K_{\eps,n}R\tau}
                 \|\bc-\bd\|_{C_t\R^n}.
\end{equation*}
The density bounds remain $\rho_*,\rho^*$. Let $\vartheta_{\bc}$ denote the mean-zero solution of the heat equation with source $\widetilde \bu_{\bc}\cdot \bg$. By \cref{lem:scalar} and orthonormality,
\begin{equation*}
 \|\vartheta_{\bc}(t)-\vartheta_{\bd}(t)\|_{H^1(\T)}
 +\|\vartheta_{\bc}(t)-\vartheta_{\bd}(t)\|_\infty
 \le C\|\bg\|_\infty|\bc(t)-\bd(t)|.
\end{equation*}
Define the mass matrix $\bM[\bc]$ and force vector $\bG[\bc]$ by
\begin{align*}
 \bM[\bc]_{ij}(t)&=\int_D\rho_{\bc}(t)\ba_j\cdot \ba_i\dx,
                       \\
 \bG[\bc]_i(t)&=\int_D(\rho_{\bc}\bfv-\vartheta_{\bc}\bg)\cdot \ba_i\dx
     -2\sig^2\int_D\mu(\vartheta_{\bc})\sym \bu_{\bc}:\sym \ba_i\dx-\sig^4\int_D\rho_{\bc}(\bu_{\bc}\cdot\nabla)\bu_{\bc}\cdot \ba_i\dx.
\end{align*}
For all $\bxi\in\R^n$, $\rho_*|\bxi|^2\le\bxi\cdot\bM[\bc](t)\bxi\le\rho^*|\bxi|^2$. The inverse matrices are therefore bounded by $\rho_*^{-1}$, independently of the path, time, $n$ and $\eps$. Using the identity for the difference of inverses, the preceding path estimates and the Lipschitz continuity of $\mu$, we obtain constants $K,L$ such that, on this ball and for $\tau\le1$,
\[
 \|\sig^{-4}\bM[\bc]^{-1}\bG[\bc]\|_{C_t\R^n}\le K,
 \quad
 \|\sig^{-4}\bM[\bc]^{-1}\bG[\bc]
-\sig^{-4}\bM[\bd]^{-1}\bG[\bd]\|_{C_t\R^n}
 \le L\|\bc-\bd\|_{C_t\R^n}.
\]
Here $K,L$ may depend on $n,\eps,R,\rho_n^0,\bfv,\bg$. Define
\begin{equation*}
 (\mathscr F\bc)(t)=\bc^0+\sig^{-4}\int_0^t
             \bM[\bc](s)^{-1}\bG[\bc](s)\dd s.
\end{equation*}
It maps $\mathscr K_{R,\tau}$ into itself if $\tau K\le1$ and $\|\mathscr F\bc-\mathscr F\bd\|_{C_t\R^n} \le\tau L\|\bc-\bd\|_{C_t\R^n}$. Choose also $\tau L<1$. Banach's theorem gives a $C^1$ fixed point. Using the transport equation for $\rho_{\bc}$, the fixed-point ODE is equivalent to \eqref{eq:galerkin}. It therefore solves the coupled Galerkin momentum and density equations.

Multiplication by the velocity coefficients, together with the transport equation for $\rho_n$, gives
\begin{equation}\label{eq:galerkin-energy}
\begin{aligned}
 &\frac{\sig^4}{2}\int_D\rho_n(\tau)|\bu_n(\tau)|^2\dx
 +2\sig^2\int_0^\tau\!\!\int_D\mu(\vartheta_n)|\sym \bu_n|^2\dx\dt
 \\
 &\quad+\int_0^\tau\!\!\int_{\T}
                  \bar\kappa_\eps|\nabla\vartheta_n|^2\dx\dt
 =\frac{\sig^4}{2}\int_D\rho_n^0|\bu_n^0|^2\dx
       +\int_0^\tau\!\!\int_D\rho_n\bfv\cdot \bu_n\dx\dt.
\end{aligned}
\end{equation}
The thermal identity used here is $\int_D\vartheta_n\bg\cdot \bu_n= \int_{\T}\bar\kappa_\eps|\nabla\vartheta_n|^2$. Consequently, on every finite interval,
\begin{equation}\label{eq:fixed-eps-bounds}
 \|\bu_n\|_{L^\infty_tL^2(D)}+\|\bu_n\|_{L^2_tV}
 +\|\vartheta_n\|_{L^2_tH^1(\T)}
 +\|\vartheta_n\|_{L^2_tL^\infty(\T)}\le C_{\eps,T}.
\end{equation}
At fixed $n,\eps$, the energy bound controls the velocity coefficients. Moreover, $\|\nabla\rho_n(t)\|_\infty \le\|\nabla\rho_n^0\|_\infty \exp(\int_0^t\|\nabla \bu_n(s)\|_\infty\dd s)$ is finite on every finite interval by finite-dimensional norm equivalence. The local construction can therefore be continued up to any finite time $T$.

The energy bound allows passage to the limit at fixed $\eps$. We first establish strong compactness of the density. Extend $\rho_n$ by $\rho_*$ on the holes and $\bu_n$ by zero. The extended density satisfies
\[
 \partial_t\widehat\rho_n
=-\Div(\widehat\rho_n\widetilde \bu_n),\quad
 \|\partial_t\widehat\rho_n\|_{L^\infty_tH^{-1}(\T)}\le C_{\eps,T}.
\]
The bounded $L^2$ ball is compact in $H^{-1}(\T)$. Equicontinuity and compactness therefore give, along a subsequence,
\[
 \widehat\rho_n\to\widehat\rho
       \quad\hbox{in }C([0,T];H^{-1}(\T)),\quad
 \bu_n\rightharpoonup \bu\quad\hbox{in }L^2_tV.
\]
The same extraction gives $\bu_n\rightharpoonup^*\bu$ in $L^\infty(0,T;H)$. Since every $\widehat\rho_n$ equals $\rho_*$ on the fixed holes, its distributional limit has that value there. The fluid density is defined by $\rho=\widehat\rho|_D$. Thus $\widehat\rho=\rho\mathbf1_D+\rho_*\mathbf1_{\He}$. The continuity equation passes to the limit already at this stage. Indeed, for a smooth scalar test $\zeta$ that vanishes for $t$ sufficiently close to $T$,
\[
 \left|\int_0^T
 \langle\widehat\rho_n-\widehat\rho,
                    \widetilde \bu_n\cdot\nabla\zeta\rangle\dt\right|
 \le C_\zeta\|\widehat\rho_n-\widehat\rho\|_{C_tH^{-1}}
                   \|\widetilde \bu_n\|_{L^1_tH^1}\to0.
\]
The remaining product passes by weak convergence of $\bu_n$. Periodic spatial mollification and the commutator estimate in \cite[Lemma~II.1(ii)]{DiPernaLions} give the renormalized limit equation, since $\widetilde \bu\in L^2_tH^1(\T)$ is solenoidal and $\widehat\rho$ is bounded. In particular, $\widehat\rho\in C_{\mathrm{weak}}([0,T];L^2(\T))$, it attains the initial density and preserves every finite $L^q$ norm. For $q=2$ this yields
\[
 \|\widehat\rho(t)\|_2=\|\widehat\rho(0)\|_2
 \quad(0\le t\le T).
\]
Weak $L^2$ continuity together with constancy of the $L^2$ norm implies strong continuity $\widehat\rho\in C([0,T];L^2(\T))$. Indeed, for any sequence $t_k\to t$ as $k\to\infty$, weak continuity gives $\widehat\rho(t_k)\rightharpoonup\widehat\rho(t)$ in $L^2$, while the norms agree. Uniform convexity then gives strong convergence. Consequently the range $\{\widehat\rho(t):0\le t\le T\}$ is compact in $L^2(\T)$.

We next prove strong density convergence uniformly in time. Since the transported approximations preserve their $L^2$ norms and $\widehat\rho_n^0\to\widehat\rho(0)$ strongly in $L^2$, we have
\[
 \sup_{t\in[0,T]}
 \bigl|\|\widehat\rho_n(t)\|_2^2-\|\widehat\rho(t)\|_2^2\bigr|\to0.
\]
Given $\delta>0$, compactness of the range of $\widehat\rho$ provides finitely many smooth functions $g_1,\ldots,g_N$ such that each $\widehat\rho(t)$ is within $\delta$ in $L^2$ of some $g_j$. For each fixed $g_j$, the convergence in $C_tH^{-1}$ gives
\[
 \sup_{t\in[0,T]}
 \bigl|\langle\widehat\rho_n(t)-\widehat\rho(t),g_j\rangle\bigr|\to0.
\]
The uniform $L^2$ bounds then imply
\[
 \sup_{t\in[0,T]}
 \bigl|\langle\widehat\rho_n(t)-\widehat\rho(t),\widehat\rho(t)\rangle\bigr|
 \le o_n(1)+C\delta.
\]
Expanding the squared $L^2$ distance and then letting $n\to\infty$ at fixed $\delta$, followed by $\delta\to0$, gives convergence in $C([0,T];L^2(\T))$. The common $L^\infty$ bound and interpolation (and, for $q<2$, the finite measure of $\T$) therefore give
\begin{equation}\label{eq:fixed-density-strong}
 \widehat\rho_n\to\widehat\rho
       \quad\hbox{in }C([0,T];L^q(\T))\quad(1\le q<\infty).
\end{equation}
The velocity is at this point known only to converge weakly.

With strong density convergence established, compactness of the projected momentum yields strong convergence of the velocity. The fixed-domain energy bound \eqref{eq:fixed-eps-bounds} and interpolation first give, uniformly in $n$,
\begin{equation}\label{eq:galerkin-flux-integrability}
 \|\bu_n\|_{L^{8/3}_tL^4(D)}\le C_{\eps,T},\quad
 \|\rho_n\bu_n\otimes \bu_n\|_{L^{4/3}_tL^2(D)}\le C_{\eps,T}.
\end{equation}
The projected momentum is $\bmom_n=\PP_n(\rho_n\bu_n)$. For $\bphi\in V$, the Galerkin equation is tested with $\PP_n\bphi$. Contractivity of $\PP_n$ on the form domain gives
\[
|\langle\partial_t\bmom_n,\bphi\rangle|
 \le C_\eps\bigl(\|\bu_n\|_4^2+\|\sym \bu_n\|_2
+\|\rho_n\bfv-\vartheta_n\bg\|_2\bigr)\|\bphi\|_V.
\]
The factors $\sig^{-2}$ and $\sig^{-4}$ from the stress and force are included in $C_\eps$. Consequently,
\begin{equation}\label{eq:projected-momentum}
 \|\bmom_n\|_{L^\infty_tH}
       +\|\partial_t\bmom_n\|_{L^{4/3}(0,T;V^*)}\le C_{\eps,T}.
\end{equation}
The derivative bound in \eqref{eq:projected-momentum} gives
\[
 \|\bmom_n(t)-\bmom_n(s)\|_{V^*}\le C_{\eps,T}|t-s|^{1/4}.
\]
By the compact embedding $H\hookrightarrow\hookrightarrow V^*$ and the Arzel\`a--Ascoli theorem, a subsequence converges in $C([0,T];V^*)$. In particular,
\[
 \bmom_n\to \bmom\quad\hbox{strongly in }L^2(0,T;V^*).
\]
For a fixed $\bphi\in L^2_tV$, the spectral convergence gives $\PP_n\bphi\to\bphi$ in $L^2_tH$ and $\int_0^T\langle \bmom_n,\bphi\rangle\dt=\int_0^T\!\int_D\rho_n\bu_n\cdot \PP_n\bphi\dx\dt$. For the fixed test, the relevant product convergence is
\[
 \|(\rho_n-\rho)\bphi\|_{L^2_tL^2(D)}
 \le\|\rho_n-\rho\|_{C_tL^3(D)}
                         \|\bphi\|_{L^2_tL^6(D)}\to0.
\]
The spectral projection error is controlled by the common $L^2_tH$ bound for $\rho_n\bu_n$ and $\PP_n\bphi\to\bphi$ in $L^2_tH$. Weak convergence of $\bu_n$ now identifies $\bmom=\PP_D(\rho\bu)$ on $V$. The hypotheses of \cref{lem:weighted-kinetic-compactness} are now satisfied. Hence
\[
 \int_0^T\!\!\int_D\rho_n|\bu_n|^2\dx\dt
 \to
 \int_0^T\!\!\int_D\rho|\bu|^2\dx\dt.
\]
Moreover,
\[
 \bu_n\to\bu
 \quad\text{strongly in }L^2((0,T)\times D).
\]

The limit temperature $\vartheta(t)\in H^1_\diamond(\T)$ is defined as the unique solution of
\[
-\Div(\bar\kappa_\eps\nabla\vartheta(t)) =\widetilde\bu(t)\cdot\bg.
\]
The source is compatible and \cref{lem:scalar} gives the required time-integrated bounds. Here $\vartheta$ is the fixed-$\eps$ microscopic temperature associated with the coefficient $\bar\kappa_\eps$. Strong velocity convergence identifies this temperature and then the viscous stress. Applying \cref{lem:scalar} to the difference of the heat equations gives
\begin{equation}\label{eq:fixed-temperature-strong}
 \|\vartheta_n-\vartheta\|_{L^2_tH^1_x}
 +\|\vartheta_n-\vartheta\|_{L^2_tL^\infty_x}
 \le C\|\widetilde \bu_n-\widetilde \bu\|_{L^2_{t,x}}\to0.
\end{equation}
Thus $\mu(\vartheta_n)\to\mu(\vartheta)$ strongly in $L^2_tL^\infty_x$, while the coefficients are bounded by $\mu^*$. To identify the stress for an arbitrary $L^2$ tensor test $\Psi$, take an almost-everywhere convergent subsequence. Since the viscosities are bounded by $\mu^*$, dominated convergence gives $(\mu(\vartheta_n)-\mu(\vartheta))\Psi\to0$ in $L^2$ as $n\to\infty$ along this subsequence. Together with $\sym \bu_n\rightharpoonup\sym \bu$ this proves weak $L^2$ convergence of the stresses. For convection, we have
\[
 \begin{aligned}
 &\|\rho_n\bu_n\otimes \bu_n-\rho \bu\otimes \bu\|_{L^1_{t,x}}\le\rho^*\|\bu_n\otimes \bu_n-\bu\otimes \bu\|_{L^1_{t,x}}
        +\int_0^T\!\!\int_D|\rho_n-\rho|\,|\bu|^2\dx\dt
        \to0.
 \end{aligned}
\]
The first term uses strong $L^2$ velocity convergence and the second uses boundedness and almost-everywhere convergence of density. The thermal force passes to the limit by \eqref{eq:fixed-temperature-strong}.

We next extend the limiting spectral identities to every test in the weak formulation. Let $\bphi$ be a test from \cref{def:weak} and set $\bphi_n=\PP_n\bphi$. Use $\bphi_n$ in the integrated Galerkin equations. Strong spectral convergence on the form domain and commutation with the time derivative give
\begin{equation*}
 \begin{gathered}
 \bphi_n\to\bphi\quad\hbox{in }L^4_tV\cap L^2_tV\cap C_tH,\quad \partial_t\bphi_n\to\partial_t\bphi\quad\hbox{in }L^1_tH.
 \end{gathered}
\end{equation*}
The projection error in convection satisfies
\begin{equation*}
 \left|\int_0^T\!\!\int_D
     \rho_n\bu_n\otimes \bu_n:\nabla(\bphi_n-\bphi)\dx\dt\right|
 \le C_{\eps,T}\|\bphi_n-\bphi\|_{L^4_tV}\to0.
\end{equation*}
The stress projection error is bounded by $C_{\eps,T}\|\bphi_n-\bphi\|_{L^2_tV}$, the time projection error by $C_{\eps,T}\|\partial_t\bphi_n-\partial_t\bphi\|_{L^1_tH}$ and the force projection error by $C_{\eps,T}\|\bphi_n-\bphi\|_{L^2_tH}$. For the initial term, strong convergence in \eqref{eq:initial-momentum-product} and $\bphi_n(0)\to\bphi(0)$ in $H$ give the required limit. After removing the projection errors, all terms are paired with the fixed test $\bphi$. We use $\rho_n\bu_n\otimes \bu_n\to\rho \bu\otimes \bu$ in $L^1_{t,x}$ and the previously established stress and thermal-force convergences. Hence \eqref{eq:weak-momentum} holds for every test in \cref{def:weak}.

For pressure reconstruction, introduce the vector distribution on $(0,T)\times D$
\[
 \begin{aligned}
 \mathcal F_\eps={}&\rho \bfv-\vartheta \bg
  -\sig^4\bigl[\partial_t(\rho \bu)+\Div(\rho \bu\otimes \bu)\bigr]+\Div\bigl(2\sig^2\mu(\vartheta)\sym \bu\bigr).
 \end{aligned}
\]
The energy and convective bounds imply $\mathcal F_\eps\in\mathcal D'(0,T;\Xe^*)$ and the test-density argument extends its annihilation property to time-smooth tests valued in $\Ve$. Let $\mathcal B_D:L^2_\diamond(D)\longrightarrow\Xe$ be the periodic zero-hole-trace divergence right inverse constructed in \cref{app:periodic-geometry}, now used on the fixed perforated domain $D=\Oe$. For a time-smooth mean-zero scalar test $\psi$, define
\[
 \langle p_\eps,\psi\rangle
=-\langle\mathcal F_\eps,\mathcal B_D\psi\rangle.
\]
Every $\bzeta\in\Xe$ satisfies $\bzeta-\mathcal B_D(\Div\bzeta)\in\Ve$, hence $\langle\mathcal F_\eps,\bzeta\rangle =-\langle p_\eps,\Div\bzeta\rangle$. This constructs $p_\eps\in\mathcal D'(0,T;L^2_\diamond(D))$ and $\nabla p_\eps=\mathcal F_\eps$. The zero mean fixes the additive time distribution. Connectedness of $D$ gives uniqueness in that gauge. This is a direct fixed-domain de Rham reconstruction through the chosen divergence right inverse. It imposes no pressure boundary condition and provides no estimate uniform in $\eps$.

The initial energies converge by \eqref{eq:initial-kinetic-product}. Along a further subsequence, $\bu_n\to \bu$ in $L^2(D)$ for almost every $t\in(0,T)$. Together with \eqref{eq:fixed-density-strong} and weak lower semicontinuity of the dissipation, this passes \eqref{eq:galerkin-energy} to \eqref{eq:energy}. For the viscous term we use $\sqrt{\mu(\vartheta_n)}\sym \bu_n\rightharpoonup \sqrt{\mu(\vartheta)}\sym \bu$ in $L^2$.

Finally, we recover the initial velocity and full momentum traces from the projected momentum. Write $\bmom=\PP_D(\rho \bu)$. The equation just obtained for every defining test, together with \eqref{eq:galerkin-flux-integrability} for the limit, implies
\[
 \bmom\in L^\infty(0,T;H),\quad
 \partial_t\bmom\in L^{4/3}(0,T;V^*).
\]
For $\bphi\in V$, the time derivative gives
\[
|\langle \bmom(t)-\bmom(s),\bphi\rangle|
 \le |t-s|^{1/4}
     \|\partial_t\bmom\|_{L^{4/3}(0,T;V^*)}\|\bphi\|_V.
\]
These scalar pairings have continuous representatives. By density of $V$ in $H$ and the $L^\infty_tH$ bound, we obtain $\bmom\in C_{\mathrm{weak}}([0,T];H)$. The initial term in \eqref{eq:weak-momentum} gives
\begin{equation*}
 \bmom(0)=\PP_D(\rho_\eps^0\bu_\eps^0).
\end{equation*}
The mass-operator lemma, applied with the continuous density furnished by \eqref{eq:fixed-density-strong}, gives the required time representative. Since $\bmom(t)=\mathscr A_\rho(t)\bu(t)$ for almost every $t\in(0,T)$, the representative $\bu(t)=\mathscr A_\rho(t)^{-1}\bmom(t)$ is weakly continuous in $H$ and
\[
 \bu(0)=\mathscr A_\rho(0)^{-1}\PP_D(\rho_\eps^0\bu_\eps^0)
=\bu_\eps^0.
\]
By the same lemma, $\rho \bu$ is weakly continuous in $L^2(D)$ and has initial trace $\rho_\eps^0\bu_\eps^0$. Thus \eqref{eq:initial-momentum-product} identifies the projected initial momentum and the strongly continuous inverse mass operator recovers the velocity trace and then the full momentum trace. The density and temperature conditions in \cref{def:weak} follow from the preceding limits.

We have proved existence at fixed $\eps$. If \eqref{eq:initial-weighted-bound} holds, we estimate the force work in \eqref{eq:energy} by the improved Poincar\'e inequality, as in \cref{sec:heat}. This gives \eqref{eq:uniform} with constants independent of $\eps$. The gradient bound remains weighted by $\sig$. The corresponding unweighted bound is of order $\sig^{-1}$.
\end{proof}

\section{Fixed-parameter microscopic existence in the bounded case}
\label{app:bounded-existence}
The periodic fixed-domain construction extends to the bounded geometry after the transport equation is embedded in a larger torus and the temperature equation is solved with the insulated Neumann condition on $\partial\Omega$. All compactness arguments are again performed at fixed $\eps$. The resulting statement is as follows.

\begin{proposition}[Bounded-domain microscopic existence]
\label{prop:bounded-micro}
For each fixed sufficiently small $\eps>0$, every admissible initial pair with $0<\rho_*\le\rho_\eps^0\le\rho^*$ and $\bu_\eps^0\in\mathscr H_{\eps,\Omega}$ of finite kinetic energy has a finite-energy weak solution in the sense of \cref{def:bounded-weak} on every finite interval. If the initial weighted kinetic energies are uniformly bounded, then
\begin{equation}\label{eq:bounded-uniform}
 \sig^2\|\sqrt{\rho_\eps}\bu_\eps\|_{L^\infty_tL^2_x}
 +\|\widetilde \bu_\eps\|_{L^2_tL^2_x}
 +\sig\|\nabla \bu_\eps\|_{L^2_tL^2_x}
 +\|\vartheta_\eps\|_{L^2_tH^1_x}
 +\|\vartheta_\eps\|_{L^2_tL^\infty_x}\le C_T.
\end{equation}
\end{proposition}
The proof follows the periodic Galerkin scheme, with the two boundary conditions treated separately in the transport and thermal equations.

\begin{proof}[Proof of \cref{prop:bounded-micro}]

Use the auxiliary torus $\mathbb T_L^3$ from the proof of \cref{lem:finite-boundary-density}, with $L$ chosen so that $\overline\Omega\subset(-L/2,L/2)^3$. At fixed $\eps$, write $D=\Oe$, $V=\V_{\eps,\Omega}$ and $H=\mathscr H_{\eps,\Omega}$. As in the preceding appendix, $(\rho,\bu,\vartheta)$ denotes the fixed-$\eps$ Galerkin limit throughout this proof. Extend velocities by zero outside $D$ and densities by the same positive constant $\rho_*$ on $\mathbb T_L^3\setminus D$. Every $\bv\in V$ extends to a solenoidal $H^1$ field on the larger torus. These extensions satisfy the global transport equation distributionally across the interface and allow renormalization and scalar test approximation. We solve the heat equation on $\Omega$ with its Neumann boundary condition.

On this fixed domain, the Galerkin construction couples the transport equation to the Neumann heat problem. The coercive Stokes form $\int_D\nabla\bv:\nabla\bz$ defines a positive self-adjoint operator on $H$. Since $V\hookrightarrow\hookrightarrow H$, its inverse is compact. The spectral theorem therefore provides an $L^2$-orthonormal basis $(\ba_j)_{j\ge1}\subset V$ and eigenvalues $\lambda_j\to\infty$. The eigenvalue identity holds first against all tests in $V$. The distribution $-\Delta\ba_j-\lambda_j\ba_j\in H^{-1}(D;\R^3)$ annihilates every solenoidal $H^1_0$ test. Using the fixed-domain Bogovski\u\i{} right inverse of divergence, equivalently the $L^2$ inf--sup condition on the fixed Lipschitz domain $D$, we recover a unique mean-zero pressure $\pi_j\in L^2_\diamond(D)$ such that
\[
 -\Delta\ba_j+\nabla\pi_j=\lambda_j\ba_j,
 \qquad \Div\ba_j=0,
 \qquad \ba_j|_{\partial D}=0.
\]
Thus each eigenpair is a weak Dirichlet Stokes solution. We apply \cite[Theorem~3.2]{Breit2024}, which covers dimensions $n=2,3$, twice on the fixed perforated domain to obtain the spatial Lipschitz regularity required by the finite-dimensional transport flow. For $n=3$, $p=2$ and $s=2$, we take $\varrho=4$ and any $\theta\in(3/2,2)$ in its Besov-domain hypothesis. The $C^{3,\beta}$ boundary satisfies this local hypothesis, and the zero boundary datum satisfies the compatibility condition. Since $\lambda_j\ba_j\in L^2(D)$, the theorem gives $\ba_j\in H^2(D)$ and the normalized pressure $\pi_j\in H^1(D)$.

The forcing $\lambda_j\ba_j$ now belongs to $H^1(D)$. Applying the same theorem with $p=2$, $s=3$, $\varrho=2$ and any $\theta\in(5/2,3)$ gives $\ba_j\in H^3(D)\hookrightarrow W^{1,\infty}(D)$. Rotating and shrinking each boundary chart gives the local smallness of the Lipschitz constant required by the theorem. The regularity constants may depend on $D$ and $j$, since they are used only for the finite-dimensional flow.

We use this basis as in \cref{sec:weak-existence}. Spectral convergence in the form domain permits passage to general tests without higher-order boundary compatibility. We obtain the initial Galerkin densities by mollifying the positive density extension on $\mathbb T_L^3$ and restricting it to $D$. The initial momentum and kinetic-energy products converge by the truncation argument in \eqref{eq:initial-momentum-product}--\eqref{eq:initial-kinetic-product}. The finite-dimensional velocity fields are Lipschitz and have zero trace, so their flows preserve the boundary and the positive density bounds.

For each Galerkin velocity $\bu_n$, solve
\[
 \int_\Omega\bar\kappa_\eps\nabla\vartheta_n\cdot\nabla\psi\dx
=\int_\Omega(\widetilde \bu_n\cdot \bg)\psi\dx
       \quad(\psi\in H^1(\Omega)),\quad
  \int_\Omega\vartheta_n\dx=0.
\]
Compatibility is \eqref{eq:bounded-compatibility} and \cref{lem:bounded-scalar} gives linear continuity of this solution map from $H$, with its $L^2(D)$ norm, to $H^1_\diamond(\Omega)\cap L^\infty(\Omega)$. All Galerkin inputs belong to $H$ and therefore have compatible Neumann sources, so the same flow--ODE fixed-point argument applies. Testing with $\bu_n$ and $\vartheta_n$ yields
\begin{align*}
 &\frac{\sig^4}{2}\int_D\rho_n(\tau)|\bu_n(\tau)|^2\dx+2\sig^2\int_0^\tau\!\!\int_D\mu(\vartheta_n)|\sym \bu_n|^2\dx\dt+\int_0^\tau\!\!\int_\Omega\bar\kappa_\eps|\nabla\vartheta_n|^2\dx\dt\\
   &=\frac{\sig^4}{2}\int_D\rho_n^0|\bu_n^0|^2\dx
                    +\int_0^\tau\!\!\int_D\rho_n\bfv\cdot \bu_n\dx\dt.
\end{align*}
The boundary contribution vanishes. At fixed $n,\eps$, this energy prevents finite-time blowup of the Galerkin coefficients and the Lipschitz flow controls the spatial Lipschitz norm of the transported density on every finite interval.

The resulting energy bound gives the compactness needed to pass to the limit. The density compactness proof in \cref{sec:weak-existence} now takes place on $\mathbb T_L^3$. Uniform boundedness and equicontinuity in $H^{-1}$ first identify its transport equation with the weak limit of the velocity. The renormalization argument then gives weak $L^2$ continuity and conservation of the $L^2$ norm. Uniform convexity yields strong $L^2$ continuity of the limit. The compact-range argument used in the proof of \eqref{eq:fixed-density-strong} consequently gives strong density convergence in every finite $C_tL^q$ space. If $\bmom_n=\PP_n(\rho_n\bu_n)$, the same estimates give
\[
 \|\bmom_n\|_{L^\infty_tH}
+\|\partial_t\bmom_n\|_{L^{4/3}_tV^*}\le C_{\eps,T}.
\]
By the compact embedding $H\hookrightarrow\hookrightarrow V^*$ and the preceding time-derivative bound, after extraction $\bmom_n\to \bmom$ strongly in $L^2(0,T;V^*)$. Let $\widehat\rho$ be the limit on the larger torus and set $\rho=\widehat\rho|_D$. Its extension remains $\rho_*$ outside $D$. Strong density convergence and weak velocity convergence identify $\bmom=\PP_D(\rho\bu)$. The hypotheses of \cref{lem:weighted-kinetic-compactness} are therefore satisfied and we obtain
\[
 \bu_n\to\bu
 \quad\text{strongly in }L^2(0,T;L^2(D)),
\]
together with convergence of the weighted kinetic energies. Define the fixed-parameter temperature $\vartheta(t)\in H^1_\diamond(\Omega)$ by
\[
 \int_\Omega\bar\kappa_\eps\nabla\vartheta(t)\cdot\nabla\psi\dx
=\int_\Omega(\widetilde\bu(t)\cdot\bg)\psi\dx
 \qquad(\psi\in H^1(\Omega)).
\]
Compatibility follows from the zero normal trace of the extended velocity. On this fixed conductivity field, \cref{lem:bounded-scalar} then implies
\[
 \|\vartheta_n-\vartheta\|_{L^2_tH^1(\Omega)}
+\|\vartheta_n-\vartheta\|_{L^2_tL^\infty(\Omega)}
       \le C_\Omega\|\bu_n-\bu\|_{L^2_tL^2(D)}\to0.
\]
This identifies the variable-viscosity stress, while strong density and velocity convergence identify convection. Spectral projection errors are removed using $L^{4/3}_tL^2_x$ integrability of $\rho_n\bu_n\otimes \bu_n$ and $L^4_tV$ convergence of the tests. For pressure reconstruction, use the standard Bogovski\u\i{} right inverse $\mathcal B_D:L^2_\diamond(D)\longrightarrow H^1_0(D;\R^3)$ on this fixed bounded perforated domain $D$. Its operator norm may depend on $\eps$. Only fixed-$\eps$ boundedness of this operator is used in the existence proof. If $\mathcal F_\eps$ denotes the vector distribution obtained by moving all non-pressure terms in the momentum equation to the right-hand side, define for $\psi\in C_c^\infty(0,T;L^2_\diamond(D))$
\[
 \langle p_\eps,\psi\rangle
 =-\langle\mathcal F_\eps,\mathcal B_D\psi\rangle.
\]
Exactly as in the periodic reconstruction, this yields
$p_\eps\in\mathcal D'(0,T;L^2_\diamond(D))$ and $\nabla p_\eps=\mathcal F_\eps$ in distributions. The zero-mean condition fixes the additive time distribution. No $\eps$-uniform pressure estimate is used here. The energy inequality follows by lower semicontinuity.

The initial traces and admissibility of comparison tests follow from the same fixed-domain operators. The density-weighted mass operator in \cref{lem:mass-operator} is a fixed-domain Hilbert-space result and therefore applies to the present $D,V,H$. It recovers the weakly continuous velocity and the full momentum trace from the projected momentum trace. The comparison-test density statement used in \cref{thm:bounded-homogenization} follows independently from \cref{lem:finite-boundary-density}. Scalar squares of zero-extended comparison fields may be mollified on $\mathbb T_L^3$ and then restricted to $\Omega$. This completes the existence proof in \cref{prop:bounded-micro}. Finally, apply \eqref{eq:bounded-poincare} and \cref{lem:bounded-scalar} directly to the energy inequality to obtain the uniform weighted estimates \eqref{eq:bounded-uniform}.
\end{proof}
\begingroup
\small
\setlength{\baselineskip}{11pt}
\pdfbookmark[0]{References}{references}

\endgroup
\end{document}